\documentclass[11pt]{amsart}

\usepackage{multirow,bigdelim}
\usepackage{amsfonts}
\usepackage{dsfont}
\usepackage{amsmath}
\usepackage{mathrsfs}
\usepackage{hyperref}
\usepackage{eucal}
\DeclareSymbolFont{SY}{U}{psy}{m}{n}
\DeclareMathSymbol{\emptyset}{\mathord}{SY}{'306}

\theoremstyle{plain}

\newtheorem*{bigT}{Theorem}

\newtheorem{thm}{Theorem}[section]

\newtheorem{lem}[thm]{Lemma}
\newtheorem{prop}[thm]{Proposition}
\newtheorem{defn}[thm]{Definition}

\theoremstyle{definition}
\newtheorem{rem}[thm]{Remark}
\newtheorem{ex}[thm]{Example}
\newtheorem*{qn}{Question}

\numberwithin{equation}{section}

\def\beq{\begin{eqnarray}}
\def\eeq{\end{eqnarray}}
\def\beqa{\begin{eqnarray*}}
\def\eeqa{\end{eqnarray*}}

\newcommand\bigzero{\makebox(0,0){\text{\huge0}}}
\begin{document}
\title[Curvature, Second fundamental form and quasi-homogeneous operators]{Curvature and the Second fundamental form in classifying  quasi-homogeneous holomorphic curves and operators\\ in the Cowen-Douglas class}

\author{Chunlan Jiang, Kui Ji and Gadadhar Misra}
\curraddr{Department of Mathematics, Hebei Normal University,
Shijiazhuang, Hebei 050016, China} \email{cljiang@hebtu.edu.cn}
\curraddr{Department of Mathematics, Hebei Normal University,
Shijiazhuang, Hebei 050016, China} \email{jikuikui@gmail.com}
\curraddr{Department of Mathematics, Indian Institute of Science,
Bangalore, 560 012, India}\email{gm@math.iisc.ernet.in}
\thanks{
The first author was Supported by National Natural Science
Foundation of China (Grant No. A010602), the second author is Supported by the Foundation for the Author of National Excellent Doctoral Dissertation of China (Grant No. 201116) and third author was supported by the J C Bose National Fellowship and the UGC, SAP -- IV}





\begin{abstract}
In this paper we study quasi-homogeneous operators,  which include the homogeneous operators, in the Cowen-Douglas class. We give two separate theorems describing canonical models (with respect to equivalence under unitary and invertible operators, respectively) for these operators using techniques from complex geometry. This considerably extends the  similarity and unitary classification of homogeneous operators in the Cowen-Douglas class obtained recently by the last author and A. Kor\'{a}nyi.  Specifically, the complex geometric invariants used  for our classification are the curvature and the second fundamental forms inherent in the definition of a quasi-homogeneous operator. We show that these operators are irreducible and determine when they are strongly irreducible. Applications include the equality of the topological and algebraic $K$-group of a quasi-homogeneous operator and an affirmative answer
to a well-known question of Halmos on similarity for these  operators.
\end{abstract}
\maketitle

\section{Introduction}
Let $\mathcal H$ be a complex separable Hilbert space and let ${\mathcal L}({\mathcal H})$ be the algebra of bounded linear operators on $\mathcal H$.
For an  open connected subset $\Omega$  of the complex plane
$\mathbb{C},$ and $n\in \mathbb N,$ Cowen and Douglas introduced the class of operators $B_n(\Omega)$ in their very influential paper \cite{CD}. An operator $T$ acting on a Hilbert space $\mathcal H$ belongs to this class if each $w\in \Omega,$ is an eigenvalue of the operator $T$ of constant multiplicity $n$, these eigenvectors span the Hilbert space $\mathcal H$ and the operator $T-w,\; w\in \Omega,$ is surjective. They showed that for an operator $T$ in $B_n(\Omega),$  there exists a holomorphic choice of $n$ linearly independent eigenvectors, that is, the map
$w\rightarrow \ker(T-w)$ is holomorphic. Thus
$\pi : E_T\rightarrow \Omega,$ where
$$E_T=\{\ker(T-w):w \in \Omega,
\pi(\,\ker(T-w)\,)=w\}$$ defines a Hermitian
holomorphic vector bundle on $\Omega$.

We recall some of the basic definitions from \cite{CD} before stating one of its main results.  The Grassmannian $\mbox{Gr}(n,{\mathcal H}),$ is the set of all $n$-dimensional subspaces of the Hilbert space ${\mathcal H}.$
A map $t: \Omega \rightarrow \mbox{Gr}(n,{\mathcal H})$ is said to be a holomorphic curve, if there exist $n$ (point-wise linearly independent) holomorphic functions $ \gamma_1 ,\gamma_2 ,\cdots, \gamma_n$ on ${\Omega}$ taking values in a Hilbert space $\mathcal H$ such that $t(w)=\bigvee \{\gamma_1(w),\cdots,\gamma_n(w)\},$ $w\in \Omega.$   Any holomorphic curve $t:\Omega \rightarrow \mbox{Gr}(n,{\mathcal H})$
gives rise to a $n$-dimensional Hermitian holomorphic vector bundle $E_t$ over $\Omega$, namely,
$$E_{t}=\{(x,w)\in {\mathcal H}\times \Omega \mid x\in t(w)\}
\,\,\mbox{and} \,\,\pi:E_t\rightarrow \Omega,\,\,\mbox{where}
\,\,\pi(x,w)=w.$$
Given two holomorphic curves $t,\,\tilde{t}:
\Omega\rightarrow \mbox{Gr}(n,{\mathcal H})$, if there exists a
unitary operator $U$ on ${\mathcal H}$ such that
$\tilde{t}=Ut$, that is, the restriction $U(w):=U_{|{E_t}(w)}$ of the unitary operator $U$ to the fiber $E_t(w)$ of $E$ at $w$ maps it to the fiber of $E_{\tilde{t}}(w)$, then $t$ and
$\tilde{t}$ are said to be congruent. If $t$ and $\tilde{t}$ are congruent, then clearly the vector bundles  $E_t$ and $E_{\tilde{t}}$ are equivalent via the holomorphic bundle map  induced by the unitary operator $U.$  Furthermore,  $t$ and
$\tilde{t}$ are said to be similar if there
exists an invertible operator $X\in {\mathcal L(H)}$ such that
${\tilde{t}}=Xt$, that is, $X(w):= X_{|{E_t}(w)}$ is an isomorphism except that $X(w)$ is no longer an isometry. In this case, we say that the vector bundles $E_t$ and $E_{\tilde{t}}$ are similar.

An operator $T$ in the class $B_n(\Omega)$ determines a holomorphic
curve $t:\Omega \to \mbox{Gr}(n,{\mathcal H})$, namely,  $t(w)=
\ker(T-w), w\in \Omega.$ However, if $t$ is a holomorphic curve,
setting $T t(w)= w t(w),$ defines a linear transformation on a dense
subspace of the Hilbert space $\mathcal H.$ In general, we have to
impose additional conditions to ensure that the operator $T$ is
bounded. Assuming that $t$ defines a bounded linear operator $T$,
unitary and similarity invariants for the operator $T$ are then
obtained from those of the vector bundle $E_t.$ To describe these
invariants, we need the curvature of the vector bundle $E_t$ along
with its covariant derivatives. Let us recall some of these notions
following \cite{CD}.

The Hermitian structure of the holomorphic bundle $E_t,$ with respect to a holomorphic frame $\gamma$ is given by the Grammian
$$h_\gamma(w)=\big (\!\!\big (\langle \gamma_j(w),\gamma_i(w)\rangle \big ) \!\!\big )_{i,j=1}^n,\;
w\in \Omega.$$
If we let $\bar{\partial}$ denote the complex structure of the vector bundle $E_t,$ then the connection compatible with both the complex structure $\partial$ and the metric $h$ is canonically determined and is given by the formula $h^{-1}\partial h dz.$ The curvature of the holomorphic Hermitian  tor bundle $E_f$ is then the $(1,1)$ form
$${\mathsf K}(w)
=-\partial\big (h_\gamma^{-1}\bar{\partial}
h_\gamma\big )dw \wedge d\bar{w}.$$ We let $\mathcal K(w)$ denote the coefficient of this $(1,1)$ form, that is, $\mathcal K(w):= -\frac{\partial}{\partial \overline{w}}\big (h_\gamma^{-1}(w)\frac{\partial}{\partial w}
h_\gamma(w)\big ).$ Thus it is an endomorphism of the fiber $E_t(w).$

Let $E$ be a $C^{\infty}$ vector bundle with a Hermitian structure.
We are not  assumed to be holomorphic. The derivatives of a bundle
map $\phi:E \to E$ with respect to a frame $\gamma$ is defined to be
\begin{enumerate}
\item
$ (\phi_\gamma  )_{\overline{w}}=\frac{\partial}{\overline{\partial}w}(\phi_\gamma);$
\item $(\phi_\gamma  )_w =\frac{\partial}{\partial w}(\phi_\gamma)+[h_\gamma^{-1}\frac{\partial}{w}h_\gamma,\phi_\gamma],  w\in \Omega.$
\end{enumerate}
Since the curvature $\mathcal K$ may be thought of  as a bundle map,
its  partial derivatives $\mathcal K_{w^i\overline{w}^j}$, $i,j\in
\mathbb{N}\cup \{0\},$ may be defined inductively. The curvature and
it's derivatives  are unitarily invariants of the holomorphic
Hermitian vector bundle $E_t,$ what is more, a finite subset of
these form a complete set of invariants as was shown in \cite{CD}.
For this and other deep connections between operator theory and
complex geometry, we refer the reader to \cite{CD}.


\begin{bigT}[Proposition 2.8, \cite{CD}]
Two holomorphic Hermitian bundles $E_t$ and $E_{\tilde{t}}$ are  equivalent if
and only if there exists an isometric (holomorphic) bundle map $V: E_{t}\rightarrow E_{\tilde{t}}$
such that
$$V \big ( (\mathcal K_t)_{w^i\overline{w}^j} \big )=\big ( (\mathcal K_{\tilde{t}})_{w^i\overline{w}^j}\big ) V,\;
i,j=0,1,\cdots,n-1.$$
\end{bigT}

It was observed in \cite{CD} that the local nature of the Complex geometric invariants limits their use in the study of the equivalence under an invertible linear transformation.
The global nature of such an equivalence is not easily detected by local invariants like the curvature and its derivatives.
However,  many interesting results were obtained in \cite{CD} involving the question of similarity. In the absence of a characterization of the equivalence classes under an invertible linear transformation, a conjecture was made for two operators in $B_1(\mathbb D)$ to be similar. Unfortunately, this conjecture turned out to be false (cf. \cite{ CM_1, CM}).  More recently, Jiang and Ji obtained the following result on similarity, which is best described in terms of the commutant $$\mathcal A^\prime(t\oplus \tilde{t})=\{T\in {\mathcal L}({\mathcal
H})|T(t(w)\oplus \tilde{t}(w))\subseteq t(w)\oplus
\tilde{t}(w),\; w\in \Omega\}$$
of two holomorphic curves $t$ and $\tilde{t}.$
\begin{bigT}[Theorem 3.1,\cite{Jia}] Suppose $t,\tilde{t}
:\Omega \rightarrow Gr(n,{\mathcal H})$ are two holomorphic curves. Then the ordered
$K_0$ group of ${\mathcal A}^{\prime}(t\oplus \tilde{t})$ is a complete similarity invariant of $t$ and $\tilde{t}$.
\end{bigT}

Describing similarity invariants in terms of the curvature and its derivatives has been somewhat more elusive except for the very recent results of R. G. Douglas, H. Kwon and S. Treil \cite{KT, DKT}.

The motivation for this work comes from three very different directions.  The attempt is to describe a canonical model and obtain invariants for operators in the Cowen-Douglas class with respect to equivalence under conjugation under a unitary or  invertible linear transformation. These questions  have been successfully addressed using ideas from $K$-theory and representation theory of Lie groups.

First, the detailed study of the Cowen-Douglas class of operators, reported in the book \cite{JW}, begins with the following basic structure theorem for these operators.
\begin{bigT}[Theorem 1.49, \cite{JW}]\label{ut}
If $T$ is an operator in the Cowen-Douglas class $B_n(\Omega),$ then there exists operators $T_{0},T_1,\ldots ,T_{n-1}$ in $B_1(\Omega)$ such that
\renewcommand\arraystretch{0.875}
\begin{equation} \label{1.1T}
T=\left ( \begin{smallmatrix}T_{0} & S_{0,1}& *& \cdots & *\\
0&T_{1}&S_{1,2}&\cdots&*\\
\vdots&\ddots&\ddots&\ddots&\vdots\\
0&\cdots&0&T_{n-2}&S_{n-2,n-1}\\
0&\cdots&\cdots&0&T_{n-1}
\end{smallmatrix}\right ).
\end{equation}
\end{bigT}
\renewcommand\arraystretch{1}
A slight paraphrasing clearly implies that if
$\{\gamma_0,\gamma_1,\cdots,\gamma_{n-1}\}$ is a holomorphic frame
for the vector bundle $E_t,$ and ${\mathcal H}=
\bigvee\{\gamma_i(w),\;w\in \Omega,\;0\leq i \leq n-1\},$ then there
exists non-vanishing holomorphic curves $t_i:\Omega \to
\mbox{Gr}(1,{\mathcal H}_i),$ $0\leq i \leq n-1,$ such that
\begin{equation}\label{1.1}
\gamma_{j} = \phi_{0,j}(t_0)+\cdots+\phi_{i,j}
(t_i)+\cdots+\phi_{j-1,j}(t_{j-1})+t_{j},\; 0\leq j \leq n-1,
\end{equation}
where $\phi_{i,j}$ are certain holomorphic bundle maps. One would expect these bundle maps to reflect the properties of the operator $T$. However the tenuous relationship between the operator $T$ and the bundle maps $\phi_{i,j}$ becomes a little more transparent \emph{only} after we impose a natural set of constraints.

Here we have chosen not to distinguish between the holomorphic curve $t$ of rank $1$ in the projective space $\mbox{\rm Gr}(1,\mathcal H)$ and a non-vanishing section of the line bundle $E_t,$ that is, we have let $t$ represent them both. This will be our convention through out the paper.

Secondly, to a large extent, these constraints were anticipated in the recent paper \cite{JJKMCR,JJKMPLMS}. In that paper, a class of operators $\mathcal FB_n(\Omega)$ in $B_n(\Omega)$  possessing, what we called, a flag structure were isolated.
The flag structure was shown to be rigid. It was then shown that the complex geometric invariants like the curvature and the second fundamental form of the vector bundle $E_T$ are indeed  unitary invariants of the operator $T$. However, to show that these form a complete set of unitary invariants, we had to impose additional constraints and introduce an even smaller class $\tilde{\mathcal F}B_n(\Omega).$  For the operators $\mathcal F B_n(\Omega)$, it turned out that the bundle maps $\phi_{j,j+1},$ $0\leq j\leq n-1,$ were constant while for those in the smaller class $\tilde{\mathcal F}B_n(\Omega),$ the remaining maps  $\phi_{j,    k}$ were all zero.

Finally, recall that an operators $T$ in $B_n(\mathbb D)$ is said to be homogeneous if the unitary orbit of $T$ under the action of the M\"{o}bius group is itself, that is, $\varphi(T)$ is unitarily  equivalent to $T$ for $\varphi$ in some open neighbourhood of the identity in the M\"{o}bius group (cf. \cite{BMShift}). A canonical element $T^{(\lambda, \boldsymbol \mu)}$ in each  unitary equivalence class of the homogeneous  operators in $B_n(\mathbb D)$ was constructed in \cite{KM}. It was then shown that two operators $T^{(\lambda, \boldsymbol \mu)}$ and $T^{(\lambda^\prime, \boldsymbol \mu^\prime)}$ are similar if and only if $\lambda = \lambda^\prime.$  In particular choosing $\boldsymbol \mu=0,$ one verifies that a homogeneous operator in $B_n(\mathbb D)$ is similar to the $n$-fold direct sum $T_0\oplus \cdots \oplus T_n,$ where $T_i$ is the adjoint of the multiplication operator $M^{(\lambda_i)}$ acting on the weighted Bergman space $\mathbb A^{(\lambda_i)}(\mathbb D)$ determined by the positive definite kernel $\frac{1}{(1-z\bar{w})^{\lambda_i}}$ defined the unit disc $\mathbb D,$  $0\leq i \leq n-1,$ $\lambda_i >0.$

The homogeneous operators are easily seen to be in the class $\mathcal F B_n(\mathbb D)$ and the operator corresponding to the Hilbert module $\mathcal M_{\rm loc}$ is in $\tilde{\mathcal F}B_n(\Omega),$ (cf. \cite{JJKMPLMS})

In this paper we study a class of operators, to be called quasi-homogeneous, for which we can prove results very similar to those for the homogeneous operators building on the techniques developed in \cite{JJKMPLMS}. This class of operators, as one may expect, contains the homogeneous operators and is characterized by the requirement that all the bundle maps of \eqref{1.1} take their values in a certain (full) jet bundle $\mathcal J_i(t)$
of the holomorphic curve $t.$  For a detailed account of the jet bundles, we refer the reader to \cite{S}.

\begin{defn}
If $t$ is a holomorphic curve in the Grassmannian of rank $1,$ that is, $t:\Omega \to \mbox{Gr}(1,{\mathcal H}).$ Let $\gamma(w)$ be a non-vanishing holomorphic section for the line bundle $E_t.$
The derivatives $\gamma^{(j)},$ $j \in \mathbb N,$ taking values again in the Hilbert space $\mathcal H$ are holomorphic. (It can be shown that they are linearly independent.)
The jet bundle $\mathcal J_{n}E_t(\gamma)$ is defined by the holomorphic frame $\{\gamma^{(0)}(:=\gamma),\gamma^{(1)},\cdots,\gamma^{(n)}\}.$
Since $t$ is a holomorphic curve, the vectors $\gamma^{(i)}(w)$ and $\gamma^{(j)}(w)$ are in the Hilbert space $\mathcal H.$ Therefore the inner product of these two vectors is defined using that of the Hilbert space $\mathcal H.$
\end{defn}
In the following definition we assume, implicitly, that  the bundle map $\phi_{i,j}$ of \eqref{1.1} are from the holomorphic line bundles
$E_{i}$ to a jet bundle $\mathcal J_{j}E_{i},$ where for brevity of notation and when there is no possibility of confusion, we will let $E_i$  denote the vector bundle induced by the holomorphic curve $t_i,$ $0\leq i \leq n-1.$
\begin{defn}\label{defq}
Let $t$ be a holomorphic curve with a holomorphic frame
$\{\gamma_0,\gamma_1,\cdots, \gamma_{n-1}\}$ in the Grassmannian $\mbox{Gr}(n,{\mathcal H})$ of a complex separable Hilbert space $\mathcal H$.  We say that $t$ has an atomic decomposition if there exists holomorphic curves $t_i:\Omega \to \mbox{Gr}(1,{\mathcal H}_i),$ to be called the atoms of $t,$ corresponding to operators $T_i:\mathcal H_i\to \mathcal H_i$ in $B_1(\mathbb D)$ and complex numbers $\mu_{i,\,j}\in\mathbb C,$ $0\leq j \leq i \leq n-1,$
such that $\mathcal H = \mathcal H_0 \oplus \cdots \oplus \mathcal H_{n-1}$ and
$$\begin{array}{llll}\gamma_0&=&\mu_{0,0}t_0\\
                     \gamma_{1}&=&\mu_{0,1}t_0^{(1)}+\mu_{1,1}t_{1}\\
                     \gamma_{2}&=&\mu_{0,2}t_{0}^{(2)}+\mu_{1,2}t_{1}^{(1)}+\mu_{2,2}t_{2}\\
                     \vdots &\vdots&  \\
                     \gamma_{j}&=&\mu_{0,j}t_{0}^{(k)}+\cdots+\mu_{i,j}t_{i}^{(j-i)}+\cdots+\mu_{j,j}t_{j}\\
                       \vdots &\vdots&  \\
                     \gamma_{n-1}&=&\mu_{0,n-1}t_{0}^{(n-1)}+\cdots+\mu_{i,n-1}t_{i}^{(n-1-i)}+\cdots+\mu_{n-1,n-1}t_{n-1}.\\
\end{array}$$

Fix $i$ in $\{0, \ldots ,n-1\}.$ We say that the holomorphic curve $t_i$ is homogeneous if for $w\in \mathbb D,$ $\mathbb C[t_i(w)] =\ker(T_i-w)$  for some  homogeneous operator $T_i$ in $B_1(\mathbb D).$  We realize, up to unitary equivalence, such a homogeneous operator $T_i$ in $B_1(\mathbb D)$  as the adjoint of the multiplication operator ${M^{(\lambda_i)}}$  on the weighted Bergman spaces $\mathbb A^{(\lambda_i)}(\mathbb D).$  Thus for a fixed $w\in \mathbb D,$ there exists a canonical (holomorphic) choice of eigenvectors $t_i(w),$ namely, $(1-z\bar{w})^{-\lambda_i}.$

We say that $t$ is quasi-homogeneous if it admits an atomic decomposition, where each of the atoms $t_i$ is homogeneous,
$\lambda_0 \leq \lambda_1 \leq \cdots \leq \lambda_{n-1}$ and the difference  $\lambda_{i+1}-\lambda_i,$ $0\leq i \leq n-2,$ is a
fixed positive real number, say, $\Lambda(t).$

\end{defn}
When the holomorphic curve defines a bounded linear operator, we shall use the terms quasi-homogeneous holomorphic curve $t$, quasi-homogeneous operator $T$ and quasi-homogeneous holomorphic vector bundle $E_t$ (or, even $E_T$) interchangeably.

If $T$ is a quasi-homogeneous operator and $\big (\!\! \big (
S_{i,j}\big )\!\!\big )$ is its   upper triangular decomposition
given in Theorem \ref{ut}, then we show that
\begin{equation} \label{int}
T_{i} S_{i, i+1} = S_{i, i+1} T_{i+1}, \;0\leq i \leq n-2.
\end{equation}
In consequence, all quasi-homogeneous operators belong to the class $\mathcal FB_n(\mathbb D)$ introduced recently in the paper \cite{JJKMCR, JJKMPLMS}.

One of the  points of this definition is that a quasi-homogeneous vector bundle $E_t$ is indeed homogeneous if  $\Lambda(t)=2$ and the constants $\mu_{i,j}$ are certain explicit functions of $\lambda$ as we point out at the end of the following section.
However, a quasi-homogeneous vector bundle need not be homogeneous as the following example shows.
\begin{ex}\label{1.6}
Let $S$ be the adjoint of the multiplication operator on arbitrary
weighted Bergmann space ${\mathbb A}^{(\lambda)}(\mathbb{D})$ and let $T$ be the operator
\renewcommand\arraystretch{1}
$$T=\left (\begin{smallmatrix}S&&\mu_{1}\, I&&0&&\cdots&&0\\
0&&S&&\mu_{2}\,I&&\cdots&&0\\
\vdots&&\ddots&&\ddots&&\ddots&&\vdots\\
0&&\cdots&&0&&S&&\mu_{n}\,I\\
0&&\cdots&&\cdots&&0&&S
\end{smallmatrix} \right ),\;\mu_{i}\in \mathbb C,
$$
\renewcommand\arraystretch{1}\noindent
defined on the $n+1$ fold direct sum $\bigoplus\limits^{n+1}{\mathbb
A}^{(\lambda)}(\mathbb{D}).$  Then $T$ is in $\tilde{\mathcal F}
B_{n+1}(\mathbb D)$ and therefore belongs to $B_{n+1}(\mathbb D)$
and the corresponding holomorphic curve $t(w) = \ker(T-w),\;w\in
\mathbb{D},$ is quasi homogeneous with $\Lambda(t)=0$. In fact, in
this Example, if we replace $S$ with an arbitrary operator, say $R,$
from $B_1(\mathbb{D}),$ then the resulting operator $T$ while no
longer quasi-homogeneous, remains a member of $\tilde{\mathcal F}
B_{n+1}(\mathbb D).$ Indeed, it has already appeared, via module
tensor products, in our earlier work \cite[Section 3.1]{JJKMPLMS}.
\end{ex}
The class of quasi-homogeneous operators, contrary to what might appear to be a rather small class of operators, contains apart from the homogeneous operators, many other operators. Indeed, in rank $2,$ for instance, it is parametrized by the multiplier algebra of two homogeneous operators.
In the definition of the quasi-homogeneous operators given above, if we let the atoms occur with some multiplicity rather than being multiplicity-free, it will make it even larger.  This would cause additional complications, which we are not able to resolve at this time. In another direction, we need not assume that the atoms themselves are homogeneous. Most of our results would appear to go through if we merely assume that the kernel function $K^{(\lambda)}(w,w)\sim \tfrac{1}{(1-|w|^2)^\lambda},\;|w| < 1.$ Deep results about such functions were obtained by
Hardy and Littlewood (cf. \cite{HL}) and have already appeared in the context of similarity, see \cite{CM}.

A Hermitian holomorphic  bundle $E$ is said to be irreducible,
if $E$ can not be written as orthogonal direct sum of two
holomorphic sub-bundles of $E$. In the paper \cite{CD}, among other things, it is shown that an operator $T$ on $B_n(\Omega)$ is irreducible if and only if the holomorphic Hermitian vector bundle $E_T$ is irreducible.

Let $X:\mathcal H \to \mathcal H$ be an invertible bounded linear
operator and let $XE_T$ be the holomorphic Hermitian vector bundle
obtained by prescribing the fiber at $w\in \Omega$ to be $X\big
(E_T(w)\big )$. A Hermitian holomorphic  bundle $E_T$ is said to be
strongly irreducible, if $XE_T$ cannot  be written as orthogonal
direct sum of two holomorphic sub-bundles for any invertible linear
operator $X,$ again, the vector bundle $E_T$ is strongly irreducible
if and only if the operator $T$ is strongly irreducible (cf.
\cite{Jia, JW, JW1}). It was proved in \cite{Jia} that a
holomorphic curve is strongly irreducible if and only if there is no
non-trivial idempotent in the commutant $\mathcal A^\prime(t).$ We
determine which of the quasi-homogeneous operators is strongly
irreducible and use this information to give a canonical model for
the equivalence (under unitary as well as invertible
transformations) class of quasi-homogeneous operators. We recall one
more notion from complex geometry which will be necessary to
describe the main results of this paper.

If $E$ is a holomorphic Hermitian vector bundle and $E_0$ is a holomorphic sub-bundle of $E,$ then $E=E_0 \oplus E_0^\perp,$ where $E_0^\perp(w)$ is  orthogonal compliment of $E_0(w),$ $w\in\Omega.$
However, as is well-known \cite{Kob}, this decomposition is holomorphic if and only if the second fundamental form of $E_0$ in $E$ is zero.
The quasi-homogeneous holomorphic Hermitian vector bundles  admit an atomic decomposition and each of these atoms are holomorphic line bundles each of which is nested in the next one via a bundle map. Therefore, it turns out, the second fundamental form of a neighboring pair of atoms is all that matters. Fortunately, this can be explicitly described as follows. The $2\times 2$ block $\Big (\begin{smallmatrix} S_{i,i} & S_{i, i+1}\\ 0 & S_{i+1,i+1}\end{smallmatrix}\Big)$ in the decomposition of the operator $T$ given in Theorem \ref{ut} is in $\mathcal F B_2(\mathbb D)$ because of the intertwining property \eqref{int}. Hence the corresponding second fundamental form $\theta_i$ of the inclusion $E_{\gamma_i}$ in $E_{\{\gamma_i, \gamma_{i+1}\}}$ (cf. \cite[Section 2.5]{JJKMPLMS} and \cite[Section 5.1]{DM}) is given by the formula
\begin{equation} \label{sf}
\theta_{i}(z) =  \frac{\mu_{i,i+1} \mathcal K_i(z) \,d\bar{z}}
{\big (\frac{\|t_{i+1}(z)\|^2}{\|t_{i}(z)\|^2} - |\mu_{i,i+1}|^2\mathcal K_i(z)\big )^{1/2}}.\end{equation}

We now describe, without going in to too many details, the main results of this paper. Fix a quasi-homogeneous operator $T$ (respectively $\tilde{T}$), or equivalently, a holomorphic curve $t$ (respectively $\tilde{t}$).
\begin{enumerate}
\item A quasi-homogeneous operator admits an upper triangular representation in terms of its atoms and it belongs to the class $\mathcal FB_n(\mathbb D)$ introduced recently in \cite{JJKMPLMS}.
\item A quasi-homogeneous operator, or equivalently, a quasi-homogeneous holomorphic vector bundle $E_T$ is irreducible;
\item If the operator $T$ is quasi-homogeneous and $\Lambda(t)<  2,$ then $T$ is strongly irreducible, and if $\Lambda(t)\geq 2$ then $T$ is strongly reducible.
\item If $t$ and $\tilde{t}$ are quasi-homogeneous holomorphic curves, which are unitarily equivalent, then we have
\begin{enumerate}
\item $\mathcal K_{t_i}=\mathcal K_{\tilde{t}_i},$
$i=0,1,\cdots, n-1,$
\item $\theta_{i,i+1}=\widetilde{\theta}_{i,i+1},\;i=0,1,\cdots,n-2,$
where  $\theta_{i,i+1}$  (respectively, $\tilde{\theta}_{i,i+1}$)  are the second fundamental forms of the inclusion $E_{\gamma_{i}}$ in $E_{\{\gamma_i,\gamma_{i+1}\}}$ (respectively, $E_{\tilde{\gamma}_{i}}$ in $E_{\{\tilde{\gamma}_i, \tilde{\gamma}_{i+1}\}}$), $0\leq i \leq n-2.$
\end{enumerate}
This was proved in \cite{JJKMPLMS} for operators in $\mathcal FB_n(\Omega)$. However, here we describe a canonical element in each unitary equivalence class of a quasi-homogeneous operator and compare these canonical elements to decide if two such operators  are unitarily equivalent. This appears to be a surprising rigidity property of quasi-homogeneous operators.
\item  Assume that  $E_t$ is a quasi-homogeneous vector bundle with atoms $t_i,$ $0\leq i \leq n-1.$ If $\Lambda(t) \geq 2,$  then $E_t$  is similar to the $n$-fold direct sum of the line bundles
$E_{t_0}, E_{t_1},\ldots , E_{t_{n-1}}.$
If $\Lambda(t) < 2,$ then the description is more complicated. However, we determine exactly when two quasi-homogeneous vector bundles are similar, even in this case.
\item We show that the  Halmos question, namely, if a bounded homomorphism of the disc algebra must be similar to a contraction, has an affirmative answer for quasi-homogeneous operators.
\end{enumerate}

The paper is organized as follows. In Section 2, we describe several properties of quasi-homogeneous operators. In particular, we determine when a quasi-homogeneous holomorphic curve defines a bounded linear operator. We show that the operators appearing in the atomic decomposition of a quasi-homogeneous operator possess an important intertwining property, which is the key to much of our study. In Section 3, we find conditions which ensure  two quasi-homogeneous operators are similar. It turns out that the answer depend only on the operators that appear on the diagonal and the first super diagonal of the decomposition given in Theorem \ref{ut}.
Section 4 contains several applications of our results on unitary equivalence and similarity of quasi-homogeneous operators.
We show that if the homomorphism of the polynomial ring induced by a quasi-homogeneous operator $T$ is bounded (for any polynomial $p$, there exists a constant $K$ independent of $p$ such that $\|p(T)\| \leq K \|p\|_{\infty, \mathbb D}$), then $T$ is similar to a contraction. This gives an affirmative answer to the well-known Halmos question. We define a topological $K^0$ group using equivalence classes of quasi-homogeneous operators under invertible linear transformations. As a second application of our results, we  show that the group $K^0$ is equal to the algebraic $K_0$ group consisting of equivalence classes of idempotents in the commutant of a quasi-homogeneous operator. In the context of the usual (topological) $K^0$ and (algebraic) $K_0$ groups, this is a consequence of the well-known theorem of R. G. Swan.

We finish this Introduction with a list of notations and conventions that we will use through out this paper.

\begin{enumerate}
\item[(a)] $t:\mathbb D\rightarrow Gr(n, {\mathcal H})$ is a fixed but arbitrary holomorphic curve in the Grassmannian of rank $n$ in some complex separable Hilbert space $\mathcal H$.
\end{enumerate}
\begin{enumerate}
\item[(b)] $E_t$ is the rank $n$ holomorphic Hermitian vector bundle obtained by setting $E_{t}(w):=t(w),\; w\in \Omega.$
If $n=1,$ we let $t$ denote a non-vanishing holomorphic section of the line bundle $E_t$ as well. In general, we will let $\gamma:=\{\gamma_0, \ldots,\gamma_{n-1}\}$ be a holomorphic frame of the vector bundle $E_t.$
\item[(c)] There exists constants $\mu_{i,j} \in \mathbb C,$ and a holomorphic frame
$\gamma:=\{\gamma_0,\gamma_1,\cdots,\gamma_{n-1}\}$ of the vector bundle  $E_t$ of the form $\gamma_{j} = \mu_{0,j}t_{0}^{(k)}+\cdots+\mu_{i,j}t_{i}^{(j-i)}+\cdots+\mu_{j,j-1}t_{j-1} + t_{j},$ $j=0,1,\ldots ,n-1,$
where $t_i:\mathbb D \to \mbox{\rm Gr}(1,\mathcal H_i),$
$\mathcal H= \mathcal H_0 \oplus \mathcal H_1 \cdots \oplus \mathcal H_{n-1},$ are holomorphic curves of rank $1$ desigmated as the atoms of $t$.

\item[(d)] $T$ is the linear transformation defined on the dense subset $\bigvee\{t(w):w\in \Omega\}$ by the rule $T(t(w))=wt(w), w\in \Omega.$
\item[(e)] In this paper, we will only consider
those holomorphic curves $t$ for which the linear transformation $T$ extends to a bounded linear operator on ${\mathcal H}$ and is in $B_n(\mathbb D)$. We will let  $E_T$ and $E_t$ denote the same holomorphic Hermitian vector bundle.
\item[(f)] A decomposition of the operator $T: \mathcal H_0 \oplus \cdots \oplus \mathcal H_{n-1} \to  \mathcal H_0 \oplus \cdots \oplus \mathcal H_{n-1}$ of the form
\renewcommand{\arraystretch}{0.875}
$$T=\left (\begin{smallmatrix}
T_0 & S_{0,1} & S_{0,2}&\cdots&S_{0,n-1}\\
0&T_{1}&S_{1,2}&\cdots&S_{1,n-1} \\
\vdots&\ddots&\ddots&\ddots&\vdots\\
0&\cdots&0&T_{n-2}&S_{n-2,n-1}\\
0&\cdots&\cdots&0&T_{n-1}\\
\end{smallmatrix} \right )$$
is said to be {\em atomic} with atoms $T_i:\mathcal H_i \to \mathcal H_i,$
which are assumed to be in $B_1(\mathbb D)$ and required to intertwine $S_{i,i+1}$, that is, $T_i S_{i,i+1} = S_{i,i+1} T_{i+1},\;0\leq i \leq n-2.$
\item[(g)] The operators $S_{i,j}$ define certain holomorphic bundle maps $s_{i,j}$ given by the rule $$s_{i,j}(t_j(w) ) = m_{i,j} t_i^{(j-i-1)}(w),\;w\in \mathbb D,$$  where the constants $m_{i,j}$ and $\mu_{i,j}$ determine each other recursively.
\item[(h)] The atoms  $T_0,T_1,\ldots ,T_{n-1}$ of the operator  $T$ and the atoms $t_0,t_1, \ldots, t_{n-1}$ of the holomorphic curve $t$ determine each other.
\item[(i)] The atoms are homogeneous, that is, for $i=0,1,\ldots n-1,$  the operator $T_i$ is the adjoint of the multiplication operator on the weighted Bergman space $\mathbb A^{(\lambda_i)}$ (cf. \cite{M}.) The weights are assumed to be increasing, the difference $\lambda_{i+1} - \lambda_i$ is assumed to be constant, say $\Lambda(t),$ which is called the {\em valency} of the operator $T$.
\item[(j)] If $T$ admits an atomic decomposition, the atoms are homogeneous and the valency  $\Lambda(t)$ is constant, then the operator $T$ (respectively, the holomorphic curve $t$ and the vector bundle $E_t$) is said to be {\em quasi-homogeneous.}

In this case, the atoms $T_i$ are assumed, without loss of generality,  to have been  realized as the adjoint of the multiplication operators on weighted Bergman spaces $\mathbb A^{(\lambda_i)}(\mathbb D).$
\end{enumerate}
\subsection*{Acknowledgement} The research we report here was initiated during a post-doctoral visit of the second author to the Indian Institute of Science in 2012 and was completed during a visit of the third author to the Hebai Normal University in the month of May 2014. We thank the Indian Institute of Science and the Hebei Normal University for providing excellent environment for our collaborative research.


\section{Canonical model under unitary equivalence}

An operator $T$ in the Cowen and Douglas class $B_n(\Omega)$ is determined, modulo unitary equivalence, by the curvature (of the vector bundle $E_T$) together with a finite number of its partial derivatives. However, if the rank $n$ of this vector bundle is $>1,$ then the computation of the curvature and its derivatives is
somewhat impractical. Here we show that if the operator is quasi-homogeneous, it is enough to restrict ourselves to the computation of the curvature of the atoms and a $n-1$ second fundamental forms of pair-wise neighbouring vector bundles. We first recall, following \cite{CD,CS}, that an operator $T$ in $B_n(\Omega)$ may be realized as the adjoint of a multiplication operator on a Hilbert space of holomorphic functions on $\Omega^*:=\{w:\bar{w} \in \Omega\}$ possessing a reproducing kernel.

\subsection{\sf Holomorphic Curve and Reproducing Kernel} For any operator $T$ in $B_n(\Omega),$ let $E_{T}$ denote the Hermitian holomorphic vector bundle with $E_T(w)=\ker (T-w),$ $w \in \Omega.$ Let $\{\gamma_0,\ldots,\gamma_{n-1}\}$ be a holomorphic
frame for $E_{T}$.  Define the map $\Gamma:\mathcal{H}\to
\mathcal{O}(\Omega^*,\mathbb{C}^n)$ by the rule
$$ \Gamma(x)(z)=\big( \langle x,\gamma_0(\bar z)\rangle ,
\ldots,\langle x,\gamma_{n-1}(\bar z)\rangle\big)^{\rm
tr},\; z\in\Omega^*,\;x\in\mathcal{H},
$$
where $\mathcal{O}(\Omega^*,\mathbb{C}^n)$ is the space of
holomorphic functions defined on $\Omega^*$ taking values in
$\mathbb{C}^n$. Since the map $\Gamma$ is evidently injective,  we transplant the inner product from $\mathcal{H}$ on the range of $\Gamma,$ making it a Hilbert space, say $\mathcal H_\Gamma.$  Thus $\Gamma$ is now unitary by definition.  Define
$K_{\Gamma}$ to be the function on $\Omega^*\times \Omega^*$ taking
values in the $n\times n$ matrices $\mathcal{M}_n(\mathbb{C})$:
\begin{eqnarray*}
K_{\Gamma}(z,w)= \big(\!\big(\langle \gamma_j(\bar w),\gamma_i(\bar
z)\rangle\big)\!\big)_{i,j=0}^{n-1},\;z,w\in\Omega^*.
\end{eqnarray*}
Setting $(K_{\Gamma})_w(\cdot)=K_{\Gamma}(\cdot,w),$ we verify that
\begin{eqnarray*}
\langle\Gamma(x)(\cdot),(K_{\Gamma})_w(\cdot)\eta\rangle_{\rm{ran}\:\Gamma}
&=&\langle
\Gamma(x)(w),\eta\rangle_{\mathbb{C}^n},\;\;x\in\mathcal{H},\eta\in\mathbb{C}^n,w\in\Omega^*.
\end{eqnarray*}
Thus $\big (K_\Gamma\big )_w$ has the reproducing property. The unitary operator $\Gamma$ intertwines the operator $T$ with the adjoint of the multiplication operator $M$ on the Hilbert space $(\mathcal H_\Gamma, K_\Gamma).$ We describe how this works for quasi-homogeneous operators. For such an operator $T$ acting on a Hilbert space $\mathcal H,$ there is a holomorphic frame $\{\gamma_0, \gamma_1, \cdots, \gamma_{n-1}\}$ and atoms $t_0, \ldots, t_{n-1},$ for which
we have
%
$$\gamma_{i} = \mu_{0,i}t_{0}^{(i)}+\cdots+\mu_{j,i}t_{i}^{(i-j)}+\cdots+\mu_{i,i}t_{i},\; \mu_{ij}\in \mathbb C.$$
%
%
At this point, assuming that the operator is quasi-homogeneous makes the atoms $T_0, T_1, \ldots ,T_{n-1}$ homogeneous. Conjugating with a diagonal unitary, if necessary, we assume without loss of generality that $t_i$ is the holomorphic curve
defined by
$$
t_i(w):= (1-\bar{w}z)^{-\lambda_i},\; \lambda_i = \lambda_0 + i\; \Lambda(t),\;0\leq i \leq n-1,\; \lambda_0 > 0,
$$
in the weighted Bergman space $\mathbb A^{(\lambda_i)}(\mathbb D).$
 Let $\Gamma:\mathcal{H}\to \mathcal{O}(\Omega^*,\mathbb{C}^n)$
be the map intertwining $T$ with the adjoint of the multiplication operator on the Hilbert space of holomorphic functions on $\Omega^*=\{w:\bar{w}\in \Omega\}$ determined by the positive definite kernel
\begin{eqnarray*}
\big(\!\! \big (K_{\Gamma}(z,w)\big )\!\!\big )_{ij}= \big(\!\!\big(\langle \gamma_j(\bar w),\gamma_i(\bar
z)\rangle\big)\!\!\big),\;\,z,w\in \Omega^*
\end{eqnarray*}


\noindent
For $1\leq i \leq n-1,$ recalling that $\Lambda(t)= \lambda_{i+1}-\lambda_i$ is a constant, the remaining atoms $t_i$ are determined, upto a constant, say, $\mu_{ii}$.
As a consequence, setting $K_i(z,w)=<t_i(\overline{w}), t_i(\overline{z})>$  and $D_i=\mbox{diag}(0,\cdots,0,\mu_{i,i}, \mu_{i,i+1}, \cdots,
\mu_{i,n-1})$, $0\leq i\leq n-1,$ then we have that
$$\big (\!\!\big (K_\Gamma \big )\!\!\big )_{kl}=\sum\limits_{i=0}^{n-1}D_i^* \big(\!\! \big ( \partial^{k}\overline{\partial}^lK_i \big )\!\!\big ) D_i.$$
Since the unitary equivalence class of the $M^*$ on $\mathcal H_\Gamma$ remains unchanged when it is conjugated by a holomorphic function, we may replace $K_\Gamma$ with
$$D^{*-1}K_\Gamma D^{-1} = \sum\limits_{i=0}^{n-1}D^{*-1}D_i^*\big(\!\! \big (\partial^{l}\overline{\partial}^kK_i\big )\!\!\big )D_iD^{-1},$$
where $D$ is the $n\times n$ diagonal matrix with $\mu_{0,0}, \mu_{1,1}, \ldots , \mu_{n-1,n-1}$ on its diagonal.
Having done this, we assume through out this paper that $\mu_{i,i}=1,\; 0\leq i\leq n-1.$

\subsection{\sf Atomic decomposition}
Let $t$ be a quasi-homogeneous holomorphic curve in $\mbox{\rm Gr}(n,\mathcal H).$ Assume that it defines a bounded linear operator $T$ on the Hilbert space $\mathcal H.$ An appeal to Theorem \ref{ut} provides, what we would {\em now} call an atomic decomposition for the operator $T.$ This decomposition has several additional properties arising out of our assumption f quasi-homogeneity.

\begin{lem} \label{atomic}
Let $t$ be a holomorphic quasi-homogeneous curve,
$\{t_0,\ldots , t_{n-1}\}$ be a set of its atoms and $\{\gamma_0, \ldots , \gamma_{n-1}\}$ be a holomorphic frame for $E_t.$ Let $\mathcal H$ be the closed linear span of the set of vectors $\{\gamma_0(w), \ldots, \gamma_{n-1}(w):w\in \mathbb D\}$ and  $\mathcal H_{i}$ be the closed linear span of the set of vectors $\{t_i(w), w\in \mathbb D\},$ $0\leq i \leq n-1.$ We have
\begin{enumerate}
\item ${\mathcal H}={\mathcal H}_{0}\oplus {\mathcal
H}_{1}\oplus \cdots \oplus {\mathcal H}_{{n-1}};$
\item There exists an operator $T,$ defined on a dense subset of vectors in $\mathcal H_t,$
which is  upper triangular with respect to the direct sum decomposition $\mathcal H_t = \mathcal H_{t_0} \oplus \cdots \oplus \mathcal H_{t_{n-1}}:$
\renewcommand\arraystretch{0.875}
$$T=\left (\begin{smallmatrix}
T_0 & S_{0,1} & S_{0,2}&\cdots&S_{0,n-1}\\
0&T_1&S_{1,2}&\cdots&S_{1,n-1} \\
\vdots &\ddots&\ddots&\ddots&\vdots\\
0&\ldots&0&T_{n-2}&S_{n-2,n-1}\\
0&0&\ldots&0&T_{n-1}\\
\end{smallmatrix}\right ),$$
where $S_{i,j}(t_j(w))=m_{i,j}t_i^{(j-i-1)}(w), T_{i}(t_{i}(w)))=w\,
t_{i}(w),\; w\in \mathbb D,\;i,j=0,1,\cdots, n-1,$ for some choice
of complex constants $m_{i,j}$ depending on the $\mu_{ij}.$ In this
case, we have $S_{i,i}S_{i,i+1}=S_{i,i+1}S_{i+1,i+1},
i=0,1,\cdots,n-2;$
\item The constants $m_{i,j}$ and $\mu_{i,j}$ determine each other.
\renewcommand\arraystretch{1}
\end{enumerate}
\end{lem}
For convenience of notation, in the proof below, we set $S_{i,i}:=T_i,\; 0\leq i \leq n-1,$ in the proof. We will adopt this practice often and call $T_0, T_1, \ldots , T_{n-1},$ the atoms of $T$. Also, $S_{i,i+1}(t_{i+1}) = \mu_{i,i+1} t_i,$ with the assumption that $\mu_{i,i}=1,\; 0\leq i\leq n-2.$
\begin{proof}
Note that $\{\gamma_0,\gamma_1,\cdots,\gamma_{n-1}\}$ is a frame
for $E_t$ and the atoms $t_i,\;0\leq i \leq {n-1}$ are pairwise  orthogonal. From  Definition \ref{defq}, we have
\begin{equation} \label{qcond}
\begin{array}{lll}
\gamma_0&=&t_0;\\
\gamma_{k}&=&\mu_{0,k}t_{0}^{(k)}+\mu_{1,k}t^{(k-1)}_1+\cdots +\mu_{k,k}t_k,\;  0\leq k<n-1;\\
{\mathcal H}&=&\bigoplus\limits_{i=0}^{n-1}{\mathcal H}_i,\;\mbox{ \rm where}\;  {\mathcal H}_i=\bigvee\{t_i(w):w\in \Omega\},\; i\in\{0,1,\cdots,n-1\}.
\end{array}
\end{equation}
In particular, the first statement of the Lemma is included in the definition of a holomorphic quasi-homogeneous curve.

For $0\leq i \leq j \leq n-1,$ let
$S_{i,j}:\mathcal H_j\to \mathcal H$  be the linear transformation induced by bundle maps
$s_{i,j}:E_{t_j}\rightarrow  {\mathcal J}_{j-i-1}E_{t_{i}},$ namely,
 $$\sum\limits_{i\leq j}
s_{i,j}(\gamma_k(w))=w\gamma_k(w),\;w\in
\mathbb{D}.$$
Equivalently, for any $k=0,1,\cdots,n-1,$ we have the
following formulae:\renewcommand\arraystretch{0.875}
$$\begin{pmatrix}
s_{0,0}-w & s_{0,1} & s_{0,2}&\cdots&\cdots&s_{0,n-1}\\
&s_{1,1}-w&s_{1,2}&\cdots&\cdots&s_{1,n-1} \\
&&\ddots&&&\vdots\\
&&&s_{k-1,k-1}-w&\cdots&s_{k-1,n-1}\\
&&\bigzero&&\ddots&\vdots\\
&&&&&s_{n-1,n-1}-w\\
\end{pmatrix}\begin{pmatrix}\mu_{0,k}t_{0}^{(k)}(w)\\
\vdots \\
\mu_{k,k}t_k(w)\\0\\ \vdots \\0 \end{pmatrix}=0. $$
\renewcommand\arraystretch{1}
It follows that
\begin{equation}\label{gammat}
(s_{k,k}-w)(\mu_{k,k}t_k(w))=0,
(s_{k-1,k-1}-w)(\mu_{k-1,k}t^{(1)}_{k-1}(w))+s_{k-1,k}(\mu_{k,k}t_{k}(w))=0,\end{equation}
Thus $s_{k,k}$ induces an operator $S_{k,k}$ with
${\ker}(S_{k,k}-w)=\bigvee\{t_{k}(w)\}$ and $s_{k-1,k}$ is a
bundle map from $E_{t_k}(w)\,(:=\mathbb C[t_{k}(w)])$ to
$E_{t_{k-1}}(w)\,(:=\mathbb C[t_{k-1}(w)])$.

{\sf Claim 1:}\,\, For any $i\leq j\leq n-1,$ $s_{i,j}$ is a bundle
map from $E_{t_j}$ to ${\mathcal J}_{j-i-1}E_{t_{i}}$ and there
exists $m_{i,j}\in \mathbb{C}$ such that
$S_{i,j}(t_j(w))=m_{i,j}t_i^{(j-i-1)}(w), w\in \mathbb{D}$.

Since $(s_{0,0}-w)\gamma_1(w) =
(s_{0,0}-w)(\mu_{0,1}t^{(1)}_{0}(w))+s_{0,1}(\mu_{1,1}t_{1}(w))=0$,
we have $$s_{0,1}(t_1(w))=m_{0,1}t_0(w),$$ where
$m_{0,1}=-\frac{\mu_{0,1}}{\mu_{1,1}}$. Similarly,
$(s_{0,0}-w)(\mu_{0,2}t^{(2)}_0(w))+s_{0,1}(\mu_{1,2}t^{(1)}_{1}(w))+s_{0,2}(\mu_{2,2}t_2(w))=0$,
and  we have
$$s_{0,2}(t_2(w))=-\frac{2\mu_{0,2}+\mu_{1,2}m_{0,1}}{\mu_{2,2}}t^{(1)}_0(w)=m_{0,2}t^{(1)}_0(w).$$
Now assume that for any fixed $k$ and some $k<j\leq n-1$, there
exits $m_{k,i}\in \mathbb{C}$ such that
$$s_{k,i}(t_i(w))=m_{k,i}t^{(i-k-1)}_{k}(w), i<j.$$
Then from equation  \eqref{gammat}, we have
$$
(s_{k,k}-w)(\mu_{k,j}t_k^{(j-k)})(w)+s_{k,k+1}(\mu_{k+1,j}t_{k+1}^{(j-k-1)}(w))+\cdots+s_{k,j}(\mu_{j,j}t_j(w))=0$$
and from the induction hypothesis, we may rewrite this as
$$\mu_{k,j}(j-k)t_k^{(j-k-1)}(w)+\mu_{k+1,j}m_{k,k+1}t_{k}^{(j-k-1)}(w)+\cdots+\mu_{j,j}s_{k,j}(t_j(w))=0.$$
Thus
$$s_{k,j}(t_j(w))=m_{k,j}t_k^{(j-k-1)}(w),$$
or, equivalently
\begin{equation}
m_{k,j}=-\frac{\mu_{k,j}(j-k)+\sum\limits_{l=1}^{j-k-1}\mu_{k+l,j}m_{k,k+l}}{\mu_{j,j}} \label{2.1.1}\end{equation}
completing the proof of our claim.

Now that we have found constants $m_{i,j}\in \mathbb{C}$ such that
$$S_{i,j}(t_j)=m_{i,j}t^{(j-i-1)}_i, i<j=0,1,\cdots,n-1$$
and since $(S_{i,i}-w)(t_{i}(w))=0,\; w\in \Omega,$ it follows that
$$\begin{array}{llll}
S_{i,i}S_{i,i+1}(t_{i+1}(w))&=&S_{i,i}(m_{i,i+1}t_{i}(w))\\
                        &=&m_{i,i+1}S_{i,i}(t_i(w))\\
                        &=&m_{i,i+1}t_i(w)\\
                        &=&S_{i,i+1}S_{i+1,i+1}(t_{i+1}(w)).\\
                        \end{array}.$$
We have ${\mathcal H}_i=\mbox{Span}_{w\in \Omega}\{t_i(w)\},
i=0,1\cdots,n-1,$ therefore
$$S_{i,i}S_{i,i+1}=S_{i,i+1}S_{i+1,i+1},  i=0,1,\cdots,n-2.$$
This completes the proof of the second statement of the lemma.

{\sf Claim 2:}\,\,For any operator $T$ in $B_n(\Omega)$ with atomic decomposition exactly as in the second statement of the lemma,
there exists  $\mu_{i,j}$ satisfying the conditions set in the equations \eqref{qcond}, that is, there exists a holomorphic frame
for $E_T,$ which is a linear combination of the non-vanishing holomorphic sections of $E_{t_i}$ and a certain number of jets.

Indeed, the proof of Claim 1 already verifies Claim 2 for $n\leq 2.$ To prove Claim 2 by induction, let us assume that  it is valid for  $k \leq n-2.$
Note that the
operator  $\big (\!\!\big (S_{i,j}\big )\!\!\big )_{i,j\leq n-2}$ is in $B_{n-1}(\Omega).$
By the induction hypothesis,  we can find $m_{i,j}, i,j\leq n-2$ verifying Claim 2 for any operator  $\big (\!\!\big (S_{i,j}\big )\!\!\big )_{i,j\leq n-2}.$   If we consider the operator
$$\begin{pmatrix}
T_{n-2}&S_{n-2,n-1}\\
0&T_{n-1}\\
\end{pmatrix},
$$ then we have that $S_{n-2,n-1}(t_{n-1})=m_{n-2,n-1}t_{n-2}.$ Now,  setting
$\mu_{n-2,n-1}=-m_{n-2,n-1},$  we can define all the
coefficients $\mu_{n-k,n-1}, 2\leq k\leq n$ recursively.  In fact,
if we consider
\renewcommand\arraystretch{0.875}
$$\begin{pmatrix}
T_{n-k} & S_{n-k,n-k+1} & S_{n-k,n-k+2}&\cdots&S_{n-k,n-1}\\
&T_{n-k+1}&S_{n-k+1,n-k+2}&\cdots&S_{n-k+1,n-1} \\
&\ddots&&&\vdots\\
&&\ddots&&\vdots\\
&\bigzero&&&\\
&&&T_{n-2}&S_{n-2,n-1}\\
&&&&T_{n-1}\\
\end{pmatrix},$$ where $2\leq k\leq n$, and set
$$\mu_{n-k,n-1}=-\frac{\sum\limits_{i=1}^{k-2}m_{n-k,n-k+i}\mu_{n-k+i,n-1}+m_{n-k,n-1}}{k-1},$$
\renewcommand\arraystretch{1}\noindent
then $\mu_{n-k,n-1}$ is defined  involving only the coefficients
$\mu_{n-k+i,n-1}$ which exist by the induction hypothesis.
Thus, coefficients $\mu_{i,j}$ depends only on the $m_{i,j}, i,j\leq n-1$.
By a direct computation,
$\gamma_{k}=\mu_{0,k}t_{0}^{(k)}+\mu_{1,k}t^{(k-1)}_1+\cdots
+\mu_{k,k}t_k,  0\leq k<n-1 $  together defines a frame for $E_T$. This completes the proof of Claim 2 and the third statement of the lemma.
\end{proof}

\subsection{\sf Boundedness} Having shown that a holomorphic quasi-homogeneous curve $t$ defines a linear transformation on a dense subset of $\mathcal H_t,$ we determine when it extends to a bounded linear operator on all of $\mathcal H_t.$ We make the following conventions here which will be in force throughout this paper.
\subsubsection{\sf Conventions} \label{conventions}

The positive definite kernel $K^{(\lambda)}(z,w)$ is the function $(1-\bar{w}z)^{-\lambda}$ defined on $\mathbb D\times \mathbb D$ and is the reproducing kernel for the weighted Bergman space $\mathbb A^{(\lambda)}(\mathbb D).$
The coefficient $a_n(\lambda)$ of $\bar{w}^nz^n$ in the power series expansion for $K^{(\lambda)}$ (in powers of $z\bar{w}$) is of the form
\begin{eqnarray*}
a_n(\lambda) &=& \tfrac{\lambda(\lambda+1)\cdots (\lambda+n-1)}{n!}\\
&=&\tfrac{\Gamma(\lambda+n)}{\Gamma(\lambda)n\Gamma(n)}\\
&\sim &n^{\lambda-1},\,(a_0(\lambda) = 1),
\end{eqnarray*}
where we have used the well-known formula due to Stirling, namely, $\tfrac{\Gamma(\lambda+n)}{\Gamma(n)}\sim n^\lambda.$

The set of vectors $e^{(\lambda)}_n:= \sqrt{a_n(\lambda)}\, z^n,\; n \geq 0,$ is an orthonormal basis in $\mathbb A^{(\lambda)}(\mathbb D).$ The action of the multiplication operator on $\mathbb A^{(\lambda)}(\mathbb D)$ is easily determined:
\begin{eqnarray*}
M (e^{(\lambda)}_n)& = &z (\sqrt{a_n(\lambda)}\, z^n)\\
&=&  \tfrac{\sqrt{a_n(\lambda)}}{\sqrt{a_{n+1}(\lambda)}}    \sqrt{a_{n+1}(\lambda)} z^{n+1}\\
&=& \tfrac{\sqrt{a_n(\lambda)}}{\sqrt{a_{n+1}(\lambda)}} e_{n+1}^{(\lambda)}\\
&\sim &\Big (\frac{n}{n+1} \Big )^{\tfrac{\lambda-1}{2}} e_{n+1}^{(\lambda)}.
\end{eqnarray*}
Often, one sets $w_n^{(\lambda)}:= \tfrac{\sqrt{a_n(\lambda)}}{\sqrt{a_{n+1}(\lambda)}}$ and says that $M$ is a weighted shift with weights $w_n^{(\lambda)}$ since $M(e_n^{(\lambda)}) = w_n^{(\lambda)} e^{(\lambda)}_{n+1}.$ The other way round, $\prod\limits_{i=0}^n w_i^{(\lambda)} =\sqrt{ \tfrac{a_0(\lambda)}{a_{n+1}(\lambda)}}\sim (n+1)^{\tfrac{1-\lambda}{2}}.$ The adjoint of this operator is then given by the formula:
$$
M^*(e_n^{(\lambda)}) = w_{n-1}^{(\lambda)} e_{n-1}^{(\lambda)}\sim \Big ( \frac{n-1}{n} \Big )^{\tfrac{\lambda-1}{2}} e_{n-1}^{(\lambda)}.
$$

The following Lemma shows that if the valency $\Lambda(t)$ is less than $2,$ then every possible linear combination of the atoms and their jets need not define a bounded linear transformation. However, from the proof of this lemma, one may infer no such obstruction can occur if $\Lambda(t) \geq 2.$
\begin{lem}\label{bdd}
Fix a natural number $n\geq 2.$ Let $t$ be a quasi-homogeneous
holomorphic curve with atoms $t_i,\;i=0,1,\ldots , n-1.$ For $0\leq
i,j \leq n-1,$ let $s_{i,j}(t_j(w)) = m_{i,j} t_i^{(j-i-1)}(w)$ be
the bundle map from $E_{t_j}(w)$ to ${\mathcal J}_{j-i-1}E_{t_{i}}$
and $S_{i,j}:\mathcal H_j \to \mathcal H_i$ be the densely defined
linear transformation  induced by the maps $s_{i,j}.$
The linear transformation of the form
\renewcommand\arraystretch{0.875}
$$T=\begin{pmatrix}
T_0 & S_{0,1} & S_{0,2}&\cdots&S_{0,n-1}\\
0&T_1&S_{1,2}&\cdots&S_{1,n-1} \\
\vdots &\ddots&\ddots&\ddots&\vdots\\
0&\ldots&0&T_{n-2}&S_{n-2,n-1}\\
0&0&\ldots&0&T_{n-1}\\
\end{pmatrix}$$
\renewcommand\arraystretch{1}
is densely defined on the Hilbert space $\mathbb A^{(\lambda_0)}(\mathbb D)\oplus \cdots \oplus \mathbb A^{(\lambda_{n-1})}(\mathbb D).$ Suppose that $\Lambda(t) < 2.$
\begin{enumerate}
\item If $\Lambda(t)\in[1+\frac{n-3}{n-1}, 2)$, $n\geq 2,$ then
$T$ is bounded.
\item If $\Lambda(t)\in [1+\frac{n-k-4}{n-k-2},
1+\frac{n-k-3}{n-k-1})$, the operator $T$ is bounded only if we set
$m_{i,j}=0$ whenever  $j-i\geq n-k-2, n-1>k\geq
0,\;n\geq 4,$ that is, $T$ must be of the form
\renewcommand\arraystretch{0.875}
$$
\begin{pmatrix}
S_{0,0} & S_{0,1}&\cdots&S_{0,n-k-2}&0&\cdots&0&0\\
&S_{1,1}&\cdots&S_{1,n-k-2}&S_{1,n-k-1}&0&\cdots&0 \\
&&\ddots&\ddots&\ddots&\ddots&\ddots&\vdots\\
&&&\ddots&\ddots&\ddots&\ddots&0\\
&&\mbox{\huge$0$}&&S_{k+1,k+1}&S_{k+1,k+2}&\cdots&S_{k+1,n-1}\\
&&&&&\ddots&\ddots&\vdots\\
&&&&&&S_{n-2,n-2}&S_{n-2,n-1}\\
&&&&&&&S_{n-1,n-1}\\
\end{pmatrix}
$$
\renewcommand\arraystretch{1}
\item  If $\Lambda(t)\in (0,1)$, then the densely defined linear transformation $T$ is bounded only if we set $m_{i,j}=0, i< j+1,
i=0,1,\cdots, n-2, n\geq 3.$
\end{enumerate}
\end{lem}

\begin{proof}

For $i=0,1,\cdots, n-1,$ the operators $S_{i,i}$ are homogeneous by definition. Thus the operator $S_{i,i}$, as we have said before, is realized as the adjoint of the multiplication operator on the weighted Bergman space $\mathbb A^{(\lambda_i)}(\mathbb D).$ The reproducing kernel $K^{(\lambda_i)}(z,w)$ for this Hilbert space is of the form $\tfrac{1}{(1-z\bar{w})^{\lambda_i}}.$  Consequently,
$$\ker\,(S_{i,i}-w)^*= \mathbb{C}[t_i(\bar{w})]=\mathbb{C}[K^{(\lambda_i)}(z,w)],w\in \mathbb{D}.$$
{\sf Claim :} If $\lambda_{j}-\lambda_i> 2(j-i)-2,
j>i=0,1,2,\cdots, n-2$, then each $s_{i,j}$ induces a non-zero
linear bounded operator $S_{i,j}$.

Without loss of generality, we set $s_{i,j}(t_j)=
m_{i,j}t_i^{(j-i-1)}, m_{i,j}\in \mathbb{C},
 i,j=0,1,\cdots,n-1$ and
$$t_i(w)=\frac{1}{(1-zw)^{\lambda_i}}, t_j(w)= \frac{1}{(1-zw
)^{\lambda_j}}.$$
Then the linear transformation $S_{i,j}:{\mathcal H}_j \rightarrow
{\mathcal H}_i$ induced by $s_{i,j}$ is densely defined by the rule
 $$S_{i,j}(t_j)= m_{i,j}t_i^{(j-i-1)},
 i,j=0,1,\cdots,n-1.$$
We have that
$$\begin{array}{lllll}||S_{i,j}||
&=&|m_{i,j}|\max\limits_{\ell}\{\frac{\sqrt{a_{\ell}(\lambda_i)}}{\sqrt{a_{\ell-(j-i-1)}(\lambda_j)}}\ell(\ell-1)\cdots
(\ell-(j-i)+2)\}\\
&=&|m_{i,j}|\max\limits_{\ell}\{\frac{\sqrt{\prod\limits_{l=0}^{\ell-(j-i)}w_l(\lambda_j)}}{\sqrt{\prod\limits_{l=0}^{\ell-1}w_l(\lambda_i)}}\ell(\ell-1)\cdots
(\ell-(j-i)+2)\}\\
\end{array}$$
By a direct computation,
$$\frac{\sqrt{\prod\limits_{l=0}^{\ell-(j-i-1)}w_l(\lambda_j)}}{\sqrt{\prod\limits_{l=1}^{\ell-1}w_l(\lambda_i)}}\ell(\ell-1)\cdots
(\ell-(j-i)+2)\sim \Big (\frac{1}{\ell^{\frac{\lambda_j-\lambda_i}{2}-(j-i-1)}}\Big ).$$ It
follows that each $S_{i,j}$ is a non-zero bounded linear operator if and only if
$$\frac{\lambda_j-\lambda_i}{2}\geq j-i-1,\;\mbox{\rm that is,}\; \lambda_j-\lambda_i\geq
2(j-i)-2.$$
Now recall that $T=\big (\!\!\big (S_{i,j}\big )\!\!\big)_{n\times n}$ is of the form:
$$T=\left ( \begin{smallmatrix}
S_{0,0} & S_{0,1} & S_{0,2}&\cdots&S_{0,n-1}\\
0&S_{1,1}&S_{1,2}&\cdots&S_{1,n-1} \\
\vdots&\ddots&\ddots&\ddots&\vdots\\
0&\cdots&0&S_{n-2,n-2}&S_{n-2,n-1}\\
0&\cdots&\cdots&0&S_{n-1,n-1}\\
\end{smallmatrix}\right ).$$
If $\Lambda(t)\geq 1+\frac{n-3}{n-1},$ then
$$\lambda_{n-1}-\lambda_0=(n-1)\Lambda(t)\geq 2(n-2).$$ By the argument given above, we obtain $S_{0,n-1}$ is non-zero and bounded.  If
$\Lambda(t)<1+\frac{n-3}{n-1}$, then we might deduce that
$m_{0,n-1}=0$ or $\mu_{0,n-1}=0$, i.e. $S_{0,n}=0.$ Thus the proof of the first statement is complete.

For the general case, if $\Lambda(t)\in [1+\frac{n-k-4}{n-k-2},
1+\frac{n-k-3}{n-k-1})$, $k\geq 0$,  then we have
$$(n-k-1)\Lambda(t)<2(n-k-1)-2.$$
And if $j-i\geq n-k-1,$ then we obtain
$$\begin{array}{lllll}
\lambda_j-\lambda_i&=&(j-i)\Lambda(t)\\
&\leq &(j-i)\frac{2(n-k-1)-2}{n-k-1}\\
&\leq &(j-i)\frac{2(j-i)-2}{j-i}\\
&=&2(j-i)-2.
\end{array}
$$
By the argument above, we have $S_{i,j}=0, j-i\geq n-k-1.$ And $S$
has the following matrix form:
\begin{equation} T=
\left (\begin{smallmatrix}
S_{0,0} & S_{0,1}&\cdots&S_{0,n-k-2}&0&\cdots&0&0\\
&S_{1,1}&\cdots&S_{1,n-k-2}&S_{1,n-k-1}&0&\cdots&0 \\
&&\ddots&\ddots&\ddots&\ddots&\ddots&\vdots\\
&&&\ddots&\ddots&\ddots&\ddots&0\\
&&\bigzero&&S_{k+1,k+1}&S_{k+1,k+2}&\cdots&S_{k+1,n-1}\\
&&&&&\ddots&\ddots&\vdots\\
&&&&&&S_{n-2,n-2}&S_{n-2,n-1}\\
&&&&&&&S_{n-1,n-1}\\
\end{smallmatrix} \right )\label{(2.3.2)}
\end{equation}
This completes the proof of the second statement.

In particular, if $0\leq
\Lambda(t)<1$ and $j-i\geq 2$, then we have
$$\begin{array}{lllll}
\lambda_j-\lambda_i&=&(j-i)\Lambda(t)\\
&<&(j-i)\\
 &\leq&2(j-i)-2,
\end{array}
$$
which implies
$$T=\left ( \begin{smallmatrix}
S_{0,0} & S_{0,1} &0&\cdots&0\\
0&S_{1,1}&S_{1,2}&\cdots&0 \\
\vdots&\ddots&\ddots&\ddots&\vdots\\
0&\cdots&0&S_{n-2,n-2}&S_{n-2,n-1}\\
0&\cdots&\cdots&0\phantom{Gadadh}&S_{n-1,n-1}\phantom{Ji}\\
\end{smallmatrix}\right ), \;\Lambda(t)\in [0,1).$$
The completes the proof of the third statement.
\end{proof}

\subsection{\sf Rigidity}
Let $T$  and ${\tilde{T}}$ be two quasi-homogeneous operators acting on the Hilbert space ${\mathcal H}=\mathcal H_0\oplus \cdots \oplus \mathcal H_{n-1}$ with atomic decompositions  $\big(\!\! \big (S_{i,j}\big )\!\!\big )$ and $\big(\!\! \big ( \widetilde{S}_{i,j}\big )\!\!\big ),$ respectively.
For a  bounded linear operators $X$ on the Hilbert space $\mathcal H,$ the following statements are
equivalent by definition.
\begin{enumerate}
\item $X \big(\!\! \big ( S_{i,j}\big )\!\!\big )_{n\times n}= \big(\!\! \big ( \widetilde{S}_{i,j}\big )\!\!\big )_{n\times n}X;$
\item $X \;t(w)\subseteq \tilde{t}\;(w),\;w\in \Omega.$
\end{enumerate}
It will be convenient to separately state as a lemma what we have already proved in the third statement of Lemma \ref{atomic}.
\begin{lem} \label{intS}
Suppose $T$ is a quasi-homogeneous  operator and
$\big(\!\! \big ( S_{i,j}\big )\!\!\big )_{n\times n}$ is its atomic decomposition. Then we have
$$S_{i,i}S_{i,i+1}=S_{i,i+1}S_{i+1,i+1}, i=0,1,\cdots,n-2.$$
\end{lem}

%
%
%
%
%
%
%
%
%
%
%
%
%
Recall that if $A$ and $B$ are two operators in $\mathcal L(\mathcal H),$ then the Rosenblum operator $\tau_{A,B}$ is defined to be the operator $\tau_{A,B}(X) = AX - XB,$ $X\in \mathcal L(\mathcal H).$ If $A=B,$ then we set $\sigma_A:=\tau_{A,B}.$ An operator is said to be quasi-nilpotent if $\lim_{n\to \infty}\|T^n\|=0.$ We will make repeated use of the following lemma, which appears as problem 232 in \cite{H}.

\begin{lem}\label{nil}
Let $P, T_0\in {\mathcal L}({\mathcal H})$ and $\sigma_{T_0}$ denote
the Rosenblum's operator.  If $P\in
\mbox{\rm ran}\,\sigma_{T_0}$ and $P$ commutes with $T_0$, then $P$ is
quasi-nilpotent.
\end{lem}

%
%
\begin{lem}\label{zero}
Let $P$ be a bounded linear operator on a Hilbert space $\mathcal H$ and $t, \tilde{t}:\Omega\rightarrow Gr(1,{\mathcal H})$ be  holomorphic curves. If $Pt(w)=\tilde{t}(w),$
$w \in \Omega,$ then either $P$ is zero or range of $P$ is dense.
\end{lem}

\begin{proof}
 Suppose that $P$ is non-zero. Let $t,\tilde{t}$ be two holomorphic curves of rank $1.$ Suppose that
 $t(w)=\mathbb C[\gamma(w)]$ and $\tilde{t}(w)=\mathbb C[\tilde{\gamma}(w)],$  $w \in \Omega$.
 Now, $\gamma,\tilde{\gamma}$ are holomorphic functions taking values in the Hilbert space $\mathcal H$ by definition. We claim that for any sequence $\{w_n\}\subseteq \Omega$, if $\lim\limits_{n\rightarrow \infty} w_n=w_0\in
\mbox{int}(\Omega)$, then we have
$$\bigvee\{\gamma(w_n):n\in \mathbb N\}={\mathcal H}.$$ In fact,
any $x\in {\mathcal H}$ which is orthogonal to
$\bigvee\{\gamma(w_n):n\in \mathbb N\}$ must be zero. This follows
since \mbox{$\langle\gamma(w),x\rangle$} is holomorphic. Pick a  a sequence $\{w_n\}^{\infty}_{n=1}$ with $\lim\limits_{n\rightarrow \infty}w_n=w_0.$ Then
$S:=\{w_i|P(\gamma(w_i))=0\}$ must be finite set. Otherwise, we have
$\{\gamma(w_n)|w_n\in S\}={\mathcal H}$ making  $P=0$, which is a
contradiction. So there exists  a subsequence
$\{w_{n_k}\}^{\infty}_{k=1}$ and $a_{n_k}\neq 0$ such that
$P(w_{n_k})=a_{n_k}\widetilde{\gamma}_{n_k}$. Then it follows
that $$\bigvee\limits_{k\in
\mathbb{N}}\{P(\gamma_{n_k})\}=\bigvee\limits_{k\in
\mathbb{N}}\{a_{n_k}\widetilde{\gamma}(w_{n_k})\}={\mathcal H}.$$
Thus $P$ has dense range.
\end{proof}

Let $t$ and ${\tilde{t}}$ be two holomorphic curves in the Grassmannian of rank $n$ in some Hilbert space $\mathcal H.$  If there exists  injective operators $X$ and $Y$ such that
$X t={\tilde{t}}$ and $Y{\tilde{t}}=t,$ then $t$ and ${\tilde{t}}$ are said to be quasi similar.
The following lemma shows that if two quasi-homogeneous holomorphic curves $t$ and ${\tilde{t}}$ are quasi similar via the operators $X$ and $Y,$ then these operators must be upper triangular with respect to the atomic decomposition of $t$ and ${\tilde{t}}$.

\begin{lem} \label{utXY}
Let $t$ and ${\tilde{t}}$ be two quasi-homogeneous holomorphic
curves with atomic decomposition $\{t_i:i=0,1,\ldots , n-1\}$ and
$\{\tilde{t}_i:i=0,1,\ldots , n-1\},$ respectively. If they are
quasi-similar via the intertwining operators $X$ and $Y$, that is,
$X t=\tilde{t}$ and $Y\tilde{t}=t$,  then for $i\leq n-1,$ we have
\begin{eqnarray*}
\lefteqn{X\big ({\bigvee}\{t_0(w),t_1(w),\cdots, t_i(w): w\in \Omega\}\big )\subseteq
\bigvee\{\tilde{t}_0(w),\tilde{t}_1(w),\cdots,
\tilde{t}_i(w):w\in \Omega\},\;\;\;}\\
&&Y\big ( \bigvee\{\tilde{t}_0(w),\tilde{t}_1(w),\cdots,
\tilde{t}_i(w)w\in\Omega\}\big )\subseteq \bigvee\{t_0(w),t_1(w),\cdots, t_i(w):w\in \Omega\}.
\end{eqnarray*}
\end{lem}

\begin{proof}

We give the proof for the case of $n=2$.
By Lemma \ref{bdd}, we may assume that the holomorphic curves $t$ and $\tilde{t}$ define quasi-homogeneous bounded operators $T$ and $\tilde{T},$ respectively. Now, Lemma \ref{atomic} gives an atomic decomposition, say $\Big (\begin{smallmatrix}
S_{0,0}&S_{0,1}\\
0&S_{11}
\end{smallmatrix} \Big )$ and $\Big (\begin{smallmatrix}
\widetilde{S}_{0,0}&\widetilde{S}_{0,1}\\
0&\widetilde{S}_{1,1}
\end{smallmatrix}\Big )$ for these operators. Assume that $X$ and $Y$ are of the form
$$X=\begin{pmatrix}
X_{1,1}&X_{1,2}\\
X_{2,1}&X_{2,2}
\end{pmatrix},Y=\begin{pmatrix}
Y_{1,1}&Y_{1,2}\\
Y_{2,1}&Y_{2,2}
\end{pmatrix},$$
with respect to the atomic decomposition of $T$ and $\tilde{T},$ respectively.
We have to only show that $X$ and $Y$ are upper-triangular. By
hypothesis, we have that
$$\begin{pmatrix}
X_{1,1}&X_{1,2}\\
X_{2,1}&X_{22}
\end{pmatrix}\begin{pmatrix}
S_{0,0}&S_{0,1}\\
0&S_{11}
\end{pmatrix}=\begin{pmatrix}
\widetilde{S}_{0,0}&\widetilde{S}_{0,1}\\
0&\widetilde{S}_{11}
\end{pmatrix}\begin{pmatrix}
X_{1,1}&X_{1,2}\\
X_{2,1}&X_{2,2}
\end{pmatrix},$$
and
$$\begin{pmatrix}
S_{0,0}&S_{0,1}\\
0&S_{11}
\end{pmatrix}\begin{pmatrix}
Y_{1,1}&Y_{1,2}\\
Y_{2,1}&Y_{2,2}
\end{pmatrix}=\begin{pmatrix}
Y_{1,1}&Y_{1,2}\\
Y_{2,1}&Y_{2,2}
\end{pmatrix}\begin{pmatrix}
\widetilde{S}_{0,0}&\widetilde{S}_{0,1}\\
0&\widetilde{S}_{11}
\end{pmatrix}.$$
Consequently,
$$ X_{2,1}S_{0,0}=\widetilde{S}_{1,1}X_{2,1},$$
$$X_{2,1}S_{0,1}+X_{2,2}S_{1,1}=\widetilde{S}_{1,1}X_{2,2},$$
and $$S_{1,1}Y_{2,1}=Y_{2,1}\widetilde{S}_{0,0}.$$ From the
intertwining relationships guaranteed by Lemma \ref{atomic},
$S_{0,0}S_{0,1}=S_{0,1}S_{1,1},$ and
$\widetilde{S}_{0,0}\widetilde{S}_{0,1}=\widetilde{S}_{0,1}\widetilde{S}_{1,1},$
it follows that
$$ X_{2,1}S_{0,1}Y_{2,1}+X_{2,2}S_{1,1}Y_{2,1}=\widetilde{S}_{1,1}X_{2,2}Y_{2,1},$$
$$ X_{2,1}S_{0,1}Y_{2,1}+X_{2,2}Y_{2,1}\widetilde{S}_{0,0}=\widetilde{S}_{1,1}X_{2,2}Y_{2,1}.$$
Furthermore,
$$ X_{2,1}S_{0,1}Y_{2,1}\widetilde{S}_{0,1}+X_{2,2}Y_{2,1}\widetilde{S}_{0,0}\widetilde{S}_{0,1}=\widetilde{S}_{1,1}X_{2,2}Y_{2,1}\widetilde{S}_{0,1},$$
$$ X_{2,1}S_{0,1}Y_{2,1}\widetilde{S}_{0,1}+X_{2,2}Y_{2,1}\widetilde{S}_{0,1}\widetilde{S}_{1,1}=\widetilde{S}_{1,1}X_{2,2}Y_{2,1}\widetilde{S}_{0,1},$$
$$ X_{2,1}S_{0,1}Y_{2,1}\widetilde{S}_{0,1}=\widetilde{S}_{1,1}X_{2,2}Y_{2,1}\widetilde{S}_{0,1}-X_{2,2}Y_{2,1}\widetilde{S}_{0,1}\widetilde{S}_{1,1}.$$
Thus
$$X_{2,1}S_{0,1}Y_{2,1}\widetilde{S}_{0,1}\in
\mbox{\rm ran}\;{\sigma_{\widetilde{S}_{1,1}}}.$$
We also have
$$\begin{array}{llll}X_{2,1}S_{0,1}Y_{2,1}\widetilde{S}_{0,1}\widetilde{S}_{1,1}&=&X_{2,1}S_{0,1}Y_{2,1}\widetilde{S}_{0,0} \widetilde{S}_{0,1}\\
&=&X_{2,1}S_{0,1}S_{1,1}Y_{2,1} \widetilde{S}_{0,1}\\
&=&X_{2,1}S_{0,0}S_{0,1}Y_{2,1} \widetilde{S}_{0,1}\\
&=&\widetilde{S}_{1,1}X_{2,1}S_{0,1}Y_{2,1} \widetilde{S}_{0,1}\\
\end{array}$$
showing that
$$X_{2,1}S_{0,1}Y_{2,1}\widetilde{S}_{0,1}\in
{\mathcal A}^{\prime}(\widetilde{S}_{1,1}).$$
We conclude, using Lemma \ref{nil}, that the operator $X_{2,1}S_{0,1}Y_{2,1}\widetilde{S}_{0,1}$ is quasi-nilpotent. Note that $\widetilde{S}_{1,1}$ is a homogeneous operator, which therefore must be unitarily equivalent to the adjoint of  the multiplication operator on the weighted Bergman space with reproducing kernel
$$K^{(\lambda_1)}(z,w)=\frac{1}{(1-z\overline{w} )^{\lambda_1}},\; \lambda_1>0.$$
Since the operator $X_{2,1}S_{0,1}Y_{2,1} \widetilde{S}_{0,1}$ commutes with $\widetilde{S}_{1,1}$, applying Lemma \ref{zero}, we conclude  that
$X_{2,1}S_{0,1}Y_{2,1}
\widetilde{S}_{0,1}=0.$
Note that each $S_{0,1}, \widetilde{S}_{0,1}, X_{0,1}$ and $Y_{0,1}$ is an intertwining operator between two holomorphic curves of rank one. Therefore, Lemma \ref{zero} shows that if any one of these operators  is non-zero, then it must has dense range. Since $S_{0,1}$ and
$\widetilde{S}_{0,1}$ are both non-zero operators, we have that
$X_{2,1}=0,$ or $Y_{2,1}=0.$
Without loss of generality, we suppose that $X_{2,1}=0$. Given that $XS=\widetilde{S}X$ and
$SY=Y\widetilde{S}$,  we have $$SYX=Y\widetilde{S}X,\;\text{ and }
XSY=\widetilde{S}XY.$$ Then we also have $$ SYX=YXS,\;
\text{ and }XY\widetilde{S}=\widetilde{S}XY.$$ So we conclude that both $XY$ and $YX$ are upper triangular.

Since $X$ is upper triangular, we have
$X_{2,2}S_{1,1}=\widetilde{S}_{1,1}X_{2,2}$. Therefore   $X_{2,2}$ has dense range. Since $XY$ and  $YX$ are both upper triangular, we see that $X_{2,2}Y_{2,1}=0.$ Since $X_{2,2}$  has
dense range, it follows that $Y_{2,1}=0$.

The proof is now completed by induction on the rank $n$ in pretty much the same way as in the proof of Proposition 3.2 given  in \cite{JJKMPLMS}.
\end{proof}

Repeating the proof of Lemma \ref{utXY}, we can obtain the following
lemma.

\begin{lem} \label{commutant}

Let $E_t$ be a quasi-homogeneous  bundle. If $X\in {\mathcal
A}^{\prime}(E_t)$, then $X$ is upper-triangular.

\end{lem}

\subsection{\sf The Second fundamental form}

In \cite[page. 2244]{DM}, an explicit formula for the second fundamental form of a holomorphic Hermitian line bundle in its first order jet bundle of rank $2$ was given. The second fundamental form, in a slightly different guise, was shown to be a unitary invariant for the class of operators $\tilde{\mathcal F}B_n(\Omega)$ in \cite{JJKMPLMS}.
We give  the computation of the second fundamental form here, yet again, keeping track of certain constants which appear in the description of the quasi-homogeneous operators.
We compute the second fundamental form of the inclusion $E_0$ in $E,$
where $\{\gamma_0, \gamma_1\}$ is a frame for $E$ with atoms $t_0$ and  $t_1.$  The line bundle defined by the atom $t_0$ is $E_0.$ By necessity, we have
\begin{eqnarray*}
\gamma_0 &=& t_0 \;\; \gamma_1 = \mu_{01}  t_0^\prime + t_1
\end{eqnarray*}
with $t_0 \perp t_1.$  As in \cite{DM,JJKMPLMS}, setting $h =
\langle\gamma_0,\gamma_0 \rangle,$ the second fundamental form
$\theta_{01}$ is seen to be of the form
$$\theta_{01} = - h^{1/2}\frac{\bar{\partial}(h^{-1}\langle\gamma_1,\gamma_0\rangle)}
{\big (\|\gamma_1\|^2-\frac{|\langle\gamma_1,\gamma_0\rangle|^2}
{\|\gamma_0\|^2}\big )^{1/2}}.$$
It is important, for what follows, to express $\theta_{01}$ in terms of the atoms $t_0$ and $t_1$ giving the formula
\begin{equation}\label{second}
\theta_{01} =  \frac{\mu_{01}\mathcal K_0}
{\big (\tfrac{\|t_1\|^2}{\|t_0\|^2}-|\mu_{01}|^2 \mathcal K_0 \big )^{1/2}},
\end{equation}
where $\mathcal K_0$ is the curvature of the line bundle $E_{t_0}$
given by the formula $- \bar{\partial}\partial \log \|t_0\|^2.$
The following lemma shows the key role of the second fundamental form in determining the unitary equivalence class of a quasi-homogeneous holomorphic curve.
\begin{lem} \label{eqSec}
Suppose that $t$ and $\tilde{t}$ are quasi-holomorphic curves with the same atoms $t_0,\,t_1.$ Then the following statements are equivalent.
\begin{enumerate}
\item  The two curves $t$ and $\tilde{t}$ are unitarily equivalent;
\item The second fundamental forms $\theta_{01}$ and  $\tilde{\theta}_{01}$ are
equal;
\item The two constants $\mu_{0,1}$ and $\tilde{\mu}_{0,1}.$
are equal.
\end{enumerate}
\end{lem}
\begin{proof}
The equivalence of the first two statements was proved in
\cite[Corollary 2.8]{JJKMPLMS}.  The equality of $\theta_{01}$ and
$\tilde{\theta}_{01}$ is clearly equivalent to
\begin{eqnarray*}
\tilde{\mu}_{01} \big (\tfrac{\|t_1\|^2}{\|t_0\|^2}+|\mu_{01}|^2 \bar{\partial}\partial \log \|t_0\|^2 \big )^{1/2} = \mu_{01} \big (\tfrac{\|t_1\|^2}{\|t_0\|^2}+|\tilde{\mu}_{01}|^2 \bar{\partial}\partial \log \|t_0\|^2 \big )^{1/2}.
\end{eqnarray*}
From this equality, we infer that $\arg(\mu_{01})=\arg(\tilde{\mu}_{01}).$

Now, squaring both sides and then taking the difference, we have
$$
\tfrac{\|t_1\|^2}{\|t_0\|^2}(\tilde{\mu}^2_{01} - \mu_{01}^2) - \tilde{\mu}^2_{01} \mu_{01}^2 \big (\bar{\partial}\partial \log \|t_0\|^2 \big )(\bar{\mu}_{01}^2 - \bar{\tilde{\mu}}_{01}^2 )=0.
$$
Given that we have assumed, without loss of generality, $\|t_0\|^2 = (1-|w|^2)^{-\lambda_0}$ and $\|t_0\|^2 = (1-|w|^2)^{-\lambda_1},$ we find that
$$
\bar{\partial}\partial \log \|t_0\|^2 = \lambda_0 (1-|w|^2)^{-2},
$$
which can be equal to $\tfrac{\|t_1\|^2}{\|t_0\|^2}$ if and only if $\lambda_1-\lambda_0 = 2.$
Thus except when $\Lambda(t)=2,$ we  must have $\mu^2_{01}-\tilde{\mu}^2_{01}=0.$ Clearly, $\tilde{\mu}_{01} = - \mu_{01}$ is not an admissible solution. So, we must have $\tilde{\mu}_{01} = - \mu_{01}.$ In case $\lambda_1-\lambda_0 = 2,$ if we assume $\tilde{\mu}_{01} \not =\mu_{01},$ then we must have
$$
\Big (\frac{1+ \lambda_0 |\tilde{\mu}_{01}|^2}{1+\lambda_0|\mu_{01}|^2}\Big )^{\tfrac{1}{2}}=\frac{|\tilde{\mu}_{01}|}{|\mu_{01}|},
$$
from which it follows that $|\tilde{\mu}_{01}|=|\mu_{01}|.$ The arguments of these complex numbers being equal, they must be
actually equal.
\end{proof}
When we consider the inclusion of the line bundle  $E_{t_i}$ in the vector bundle $E_{\{t_i, \tfrac{m_i,j}{j-i}ti^{(j-i)} + t_j\}}$ of rank $2,$ the situation is slightly different. This is the vector bundle which corresponds to the $2\times 2$
operator block $T_{i,j}:=\Big (\begin{smallmatrix} S_{i,i} &
S_{i,j}\\ 0 & S_{j,j}\end{smallmatrix}\Big).$

Clearly, $\{t_i, -\frac{m_{i,j}}{j-i}t^{(j-i)}_i+t_j\}$ is the
frame for $E_{T_{i,j}}.$ By the formulae above, setting {\em temporarily}  $\gamma_0=t_i,
\gamma_1=-\frac{m_{i,j}}{j-i}t^{(j-i)}_i+t_j$,  we have that
\begin{enumerate}
\item $h_i=||\gamma_0||^2=||t_i||^2,h_j=||t_j||^2;$
\item
$||\gamma_1||^2=|\frac{m_{i,j}}{j-i}|^2\partial^{j-i}\overline{\partial}^{j-i}||t_i||^2+||t_j||^2
=|\frac{m_{i,j}}{j-i}|^2\partial^{j-i}\overline{\partial}^{j-i}h_i+h_j;$
\item
$<\gamma_1,\gamma_0>=-\frac{m_{i,j}}{j-i}\partial^{j-i}||t_i||^2=-\frac{m_{i,j}}{j-i}\partial^{j-i}h_i;$
\item $|<\gamma_1,\gamma_0>|^2=|\frac{m_{i,j}}{j-i}|^2\partial^{j-i}h_i\overline{\partial}^{j-i}h_i.$
\end{enumerate}

The second fundamental form $\theta_{i,j}$ for the inclusion
$E_{t_i} \subseteq E_{\{t_i, \tfrac{m_i,j}{j-i}ti^{(j-i)} + t_j\}}$
is given by the formula
\begin{equation} \label{sffformula}
\theta_{i,j} = \frac{\frac{m_{i,j}}{j-i}\overline{\partial}(h_i^{-1}\partial^{j-i}h_i)}{(\frac{h_j}{h_i}+|\frac{m_{i,j}}{j-i}|^2(\frac{h_i\partial^{j-i}\overline{\partial}^{j-i}h_i-\partial^{j-i}h_i\overline{\partial}^{j-i}h_i}{h^2_i}) )^{\frac{1}{2}}}.
\end{equation}

\begin{lem}\label{sff}
Let $T_{i,j}:=\Big (\begin{smallmatrix} S_{i,i} &
S_{i,j}\\ 0 & S_{j,j}\end{smallmatrix}\Big)$ and
$\widetilde{T}_{i,j}:=\Big (\begin{smallmatrix} S_{i,i} & \widetilde{S}_{i,j}\\
0 & S_{j,j}\end{smallmatrix}\Big)$ with
$\widetilde{S}_{i,j}(t_j)=\widetilde{m}_{i,j}t^{(j-i-1)}_i.$ The second fundamental forms 
$\theta_{i,j}$ and $\widetilde{\theta}_{i,j}$ of the operators $T_{i,j}$ and $\widetilde{T}_{i,j}$
are equal, that is, $\theta_{i,j}=\widetilde{\theta}_{i,j}$ if and only if $m_{i,j}=\widetilde{m}_{i,j}.$
\end{lem}
\begin{proof} Without loss of generality, we will give the proof only for the case $i=0,j=k, j\not = 1.$ In this case,
$\theta_{0,k}=\widetilde{\theta}_{0,k}$ is equivalent to the equality:
$$\frac{(\frac{h_k}{h_0}+|\frac{m_{0,k}}{k}|^2(\frac{h_0\partial^{k}\overline{\partial}^{k}h_0-\partial^{k}h_0\overline{\partial}^{k}h_0}{h^2_0}) )^{\frac{1}{2}}}
{(\frac{h_k}{h_0}+|\frac{\widetilde{m}_{0,k}}{k}|^2(\frac{h_0\partial^{k}\overline{\partial}^{k}h_0-\partial^{k}h_0\overline{\partial}^{k}h_0}{h^2_0})
)^{\frac{1}{2}}}=\frac{m_{0,k}}{\widetilde{m}_{0,k}} $$
For simplicity, let $g_0$ denote 
$(\frac{h_0\partial^{k}\overline{\partial}^{k}h_0-\partial^{k}h_0\overline{\partial}^{k}h_0}{h^2_0})$ and let $m,$ $\widetilde{m} $ denote 
$\frac{m_{0,k}}{k},$ $\frac{\widetilde{m}_{0,k}}{k}$ respectively.
Then the equation given above may be rewritten as 
$$\frac{(\frac{h_k}{h_0}+|m|^2g_0 )^{\frac{1}{2}}}
{(\frac{h_k}{h_0}+|\widetilde{m}|^2g_0
)^{\frac{1}{2}}}=\frac{m}{\widetilde{m}}
$$
From this equality, we infer that $\arg(m)=\arg(\widetilde{m}).$
Now, squaring both sides and then taking the difference, we have
$$
\tfrac{h_k}{h_0}(\tilde{m}^2 - m^2) - \tilde{m}^2 m^2 g_0 (\bar{m}^2
- \bar{\tilde{m}}^2 )=0.
$$
Having assumed, without loss of generality, $h_0 =
(1-|w|^2)^{-\lambda_0}$ and $h_k = (1-|w|^2)^{-\lambda_1},$ we find
that $g_0$ is a polynomial of degree $>1$ in $(1-|w|^2)^{-1}.$  Thus $g_0$ can be equal to $\tfrac{h_k}{h_0}$ if and only if
$\lambda_1-\lambda_0 = 2.$ Therefore, except when $\Lambda(t)=2,$ we  must have $m^2-\tilde{m}^2=0.$ Clearly, $m = -\widetilde{ m}$ is not an admissible solution. So, we must have $m = \widetilde{m}.$ Hence $m_{0,k}=\widetilde{m}_{0,k}.$
\end{proof}

\subsection{\sf  Unitary equivalence}  Recall that a positive definite kernel $K:\Omega \times \Omega \to \mathbb C^{n\times n}$ is said to be normalized at $w_0\in \Omega,$ if $K(z,w_0) = I,$ $z\in \Omega.$ An operator $T$ in $B_n(\Omega)$ may be realized, up to unitary equivalence, as the adjoint of a multiplication operator on a Hilbert space possessing a normalized reproducing kernel (cf. \cite{CS}). Realized in this form, the operator is determined completely modulo multiplication by a constant unitary operator acting on $\mathbb C^n.$ As one might expect, finding the normalized kernel if $n>1$ is not easy. The second statement of the theorem below is a rigidity theorem in the spirit of what was proved by  Curto and Salinas for operators in $B_n(\mathbb D).$ For quasi-homogeneous operators, the atoms are homogeneous operators in $B_1(\mathbb D).$ These are assumed to be realized in normal form. Consequently, if $T$ is a quasi-homogeneous operator, a set of $n-1$ fundamental forms determine the operator $T$ completely, that is, two of them are unitarily equivalent if and only if they are equal assuming they have the same set second fundamental forms. The first of the two statements given in the theorem below was proved for operators in $\mathcal F B_n(\Omega)$ (cf. \cite[Proposition 3.5]{JJKMPLMS}). We have included it here only for the sake of completeness.
\begin{thm}
For any two holomorphic curves $t$ and ${\tilde{t}}$ with atoms $\{t_i:0\leq i \leq n-1 \}$ and $\{\tilde{t}_i:0\leq i \leq n-1\},$ respectively, we have the following.
\begin{enumerate}
\item If $t$ and $\tilde{t}$ are unitarily equivalent, then for $0\leq i \leq n-1,$
\begin{enumerate}
\item $\mathcal K_{t_i}=\mathcal K_{\tilde{t}_i}, \; 0\leq i \leq n-1;$
\item $\theta_{i,i+1}=\tilde{\theta}_{i,i+1},\; 0\leq i \leq n-2.$
\end{enumerate}
\item Suppose that $t$ and $\tilde{t}$ are unitarily equivalent. 
Then if the second fundamental forms are the same, that is, $\theta_{i,i+1}=\tilde{\theta}_{i,i+1},\; 0\leq i \leq n-2,$ then $t=\tilde{t}.$
\end{enumerate}
\end{thm}
\begin{proof}

If necessary, conjugating by a diagonal unitary, without loss of generality, we may assume that the atoms of the operators $T$ and $\tilde{T}$ are the same.
If there exists a unitary operator $U$ such that
$T U=U\tilde{T}$, then $U$ must be diagonal with
unitaries  $U_{0},U_1,\ldots U_{n-1}$ on its diagonal.
Then we have $$U_{i}S_{i,j}=\tilde{S}_{i,j}U_{j},\;
i,j=0,1,\ldots,n-1.$$
In particular,  $U_{i}$ commutes with the fixed set of atoms $T_i,$ which are irreducible, therefore  there exists $\beta_i \in [0,2\pi]$ such that
$$U_{i}=e^{\imath\beta_i}I_{\mathcal H_i}, i=0,1,\cdots,n-1.$$
Then on the one hand, we have 
$$U_{i}S_{i,i+1}(t_{i+1}) = U_{i}(-\mu_{i,i+1}t_i)= -\mu_{i,i+1} e^{\imath \beta_i}t_i$$ and on the other hand, we have
$$\tilde{S}_{i,i+1}U_{i+1}(t_{i+1})
 = S_{i,i+1}(e^{\imath \beta_{i+1}}t_{i+1}) = -\tilde{\mu}_{i,i+1}e^{\imath \beta_{i+1}} t_i.$$
Consequently,
$$ -\mu_{i,i+1} e^{\imath \beta_i} = -\tilde{\mu}_{i,i+1}e^{\imath\beta_{i+1}},\;0\leq i \leq n-2.$$
The assumption that the second fundamental forms are the same for
the two operators $T$ and $\widetilde{T}$ implies that
$\mu_{i,i+1}=\tilde{\mu}_{i,i+1}.$ Therefore, we have
$\theta_{i,i+1}=\widetilde{\theta}_{i,i+1},\; i=0,1,\ldots,n-1.$
\end{proof}
\begin{rem}
It is natural to ask which of the quasi-homogeneous operators are homogeneous. A comparison with the homogeneous operators given in \cite{KM_1} shows that a quasi-homogeneous operator is homogeneous if and only if
\begin{equation}
\mu_{i,j} = \frac{\Gamma_{i,j}(\lambda)\mu_i}{\mu_j},\;
\Gamma_{i,j}(\lambda) = \binom{i}{j}\frac{1}{(2
\lambda_j)_{i-j}},\;\lambda_j = \lambda - \tfrac{m}{2}+j,
\end{equation}
for some choice of positive constants $\mu_0 (:=1), \mu_1, \ldots ,\mu_{n-1}.$ Here $(\alpha)_\ell:= \alpha(\alpha+1)\cdots (\alpha +\ell-1)$ is the Pochhammer symbol. Clearly, if two homogeneous operators with $(\lambda, \boldsymbol \mu)$ and $(\tilde{\lambda}, \tilde{\boldsymbol \mu})$ were unitarily equivalent, then $\lambda$   must equal $\tilde{\lambda}.$ Since it is easy to see that $\mu_{i,i+1} = \tilde{\mu}_{i,i+1}$ if and only if $\mu_i = \tilde{\mu}_{i+1},$ we conclude two of these homogeneous operators are unitarily equivalent if and only if they are equal recovering previous results of \cite{KM_1}.
\end{rem}

\section{Canonical model under similarity}

In this section, our main focus is on the question of reducibility and strong irreducibility of a quasi-homogeneous operator. We recall that an operator $T$ is said to be strongly irreducible if there is no idempotent in its commutant, or equivalently, there does not exist an invertible operator $L$ for which $LTL^{-1}$ is reducible. The (multiplicity-free) homogeneous operators in the Cowen-Douglas class of rank $n$ are irreducible (cf. \cite{KM_1}). However, they were shown (cf. \cite{KM}) to be similar to the $n$ - fold direct sum of their atoms making them strongly reducible. It is this phenomenon that we investigate here for quasi-homogeneous operators. Along the way, we determine when two quasi-homogeneous operators are similar. Our investigations show that there is dichotomy which depends on whether or not the valency $\Lambda(t)$ is less than $2$ or greater or equal to $2.$
In what follows, we will say that a holomorphic curve $t:\mathbb D\rightarrow Gr(n,{\mathcal H})$ is strongly irreducible if there is no invertible operator $L$ on the Hilbert space $\mathcal H$ for which $X t$ splits into orthogonal direct sum of two holomorphic curves, say $t_1$ and $t_2$, in $Gr(n_1,{\mathcal H})$ and $Gr(n_2,{\mathcal H}),$ $n_1+n_2=n,$ respectively.

Suppose $t:\mathbb D\to Gr(n,{\mathcal H})$ is a quasi-homogeneous holomorphic curve  with atoms $t_0, t_1, \ldots , t_{n-1}.$ Then $t$ is strongly reducible, $t\sim t_0\oplus t_1\cdots \oplus t_{n-1},$ if $\Lambda(t)\geq 2$ and strongly irreducible otherwise.
The dichotomy involving the valency $\Lambda(t)$ is also clear from the main theorem on similarity of quasi-homogeneous holomorphic curves.

The atoms of a quasi-homogeneous operator are homogeneous operators in $B_1(\mathbb D)$ by definition. Therefore,  they are uniquely determined not only up to unitary equivalence but upto similarity as well. Now, pick any two quasi-homogeneous operators. They possess  an atomic decomposition by virtue of Lemma \ref{atomic}. Any invertible operator intertwining these two quasi-homogeneous  operators is necessarily upper triangular by Lemma \ref{utXY} with respect to their respective atomic decomposition.  Hence if two quasi-homogeneous operators are similar, then each of the atoms for one must be similar to the other.
Consequently, to determine equivalence of quasi-homogeneous operators $T$ under an invertible linear transformation, we may assume (as before) without loss of generality that the atoms are fixed with the weight $\lambda_0$ and the valency $\Lambda(t).$
Clearly, the valency $\Lambda(t)$ is both an unitary as well as a similarity invariant of the quasi-homogeneous curve $t.$

Note that if we let $R$ be the $n\times n$ diagonal matrix with
$\big (\prod\limits_{\ell=0}^{i}\mu_{\ell,\ell+1})\big ({\prod\limits_{\ell=0}^{i} \tilde{\mu}_{\ell,\ell+1}}\big )^{-1}$ on its diagonal and set $\tilde{t}=R\, t \,R^{-1},$ then $\tilde{S}_{i,i+1}(t_{i+1}) = \tilde{\mu}_{i,i+1},$ $0\leq i \leq n-2.$
Thus up to similarity, we may assume that the constants $\mu_{i,i+1}$ and $\tilde{\mu}_{i,i+1}$ are the same. Or equivalently (see Lemma \ref {eqSec}), we may assume that the choice of the second fundamental forms $\theta_{i,i+1},$ $0\leq i\leq n-2,$ does not change the similarity class of a quasi-homogeneous holomorphic curve. Therefore the condition in the second statement of the theorem given below is not a restriction on the similarity class of the holomorphic curves $t$ and $\tilde{t}.$

\begin{thm}\label{mainthmSim}
Suppose $t$ and $\tilde{t}$ are  quasi-homogeneous holomorphic curves.
\begin{enumerate}
\item If $\Lambda(t)\geq 2$, then  $t$  is similar to
the $n$ - fold direct sum of the atoms $t_0\oplus  t_1 \oplus \cdots \oplus t_{n-1}.$
\item If $\Lambda(t) = \Lambda(\tilde{t}) < 2$ and
$\theta_{i,i+1}=\tilde{\theta}_{i,i+1},\;i=0,1,\cdots,n-2,$ then $t$ and $\tilde{t}$ are similar if and only if they are equal.
\end{enumerate}
\end{thm}

In what follows, for brevity of notation, we let 
$T_0:=S_{0,0}$ and $T_{k+1}:=S_{k+1,k+1}$ ($k,$ $1\leq k\leq n-2$ is fixed but arbitrary) be the two atoms of a quasi-homogeneous operator $T$. As always, we assume they have been realized as the adjoint of the multiplication operators on the weighted Bergman spaces $\mathbb A^{(\lambda_0)}(\mathbb D)$ and $\mathbb A^{(\lambda_{k+1})}(\mathbb D),$ respectively.

\subsection{\sf The Key Lemma} The following lemma is the key to determining when a bundle map that intertwines two quasi-homogeneous holomorphic vector bundles extends to an invertible bounded operator. It reveals the intrinsic structure of the intertwiners between two quasi-homogeneous bundles. We follow the conventions set up in Section \ref{conventions}.

\begin{lem} \label{keyL}
Let $E_t$ be a quasi-homogeneous vector bundle and $s_{i,j},
i,j=0,1,\cdots, n-1$ be the induced bundle maps. There exists a bundle map $X:E_{t_{n-1}}\rightarrow {\mathcal
J}_{n-1}(E_{t_{0}})$ with the intertwining property
$$s_{0,0}X-Xs_{n-1,n-1}=s_{0,n-1}$$ that extends to a bounded linear operator only if $\Lambda(t) \geq 2.$

\end{lem}

\begin{proof}

Let $T_{0}$ and $T_{k+1}$ be the operators induced by $s_{0,0}$
and $s_{k+1,k+1}$ as in in Lemma \ref{atomic}. These are then necessarily the operators ${M^{(\lambda_0)}}^*$ and ${M^{(\lambda_{k+1})}}^*$ acting on the weighted Bergman spaces $\mathbb A^{(\lambda_0)}(\mathbb D)$ and $\mathbb A^{(\lambda_{k+1})}(\mathbb D),$ respectively.

The kernel of the operator $(T_i-w),\;w\in \mathbb D,$ is spanned by the vector $t_i(w):=(1-z\bar{w})^{-\lambda_i}, \; i=0,k+1.$ By hypothesis, for each fixed $w\in\mathbb D,$ we have
$S_{0,k+1}((1-z\bar{w})^{-\lambda_{k+1}}) = \bar{\partial}^{k}(1-z\bar{w})^{-\lambda_{0}}.$ Differentiating both sides of this equation $\ell$ times and then evaluating at $w=0$, we get $S_{0,k+1}\big ( (\lambda_{k+1})_\ell z^\ell \big ) = (\lambda_0)_{\ell+k} z^{\ell+k}.$
For  $j=0$ or $j=k-1,$ the set of vectors $e^{(\lambda_{j})}_\ell:= \sqrt{a_\ell(\lambda_j)}\,z^\ell,\;\ell\geq 0$ is an orthonormal basis in $\mathbb A^{(\lambda_j)}(\mathbb D).$ The matrix representation for the operator $S_{0,k+1}:\mathbb A^{(\lambda_{k+1})}(\mathbb D)  \to \mathbb A^{(\lambda_{0})}(\mathbb D)$ with respect to this orthonormal basis is obtained from the computation:
\begin{eqnarray*}
S_{0,k+1} \big (e^{(\lambda_{k+1})}_\ell \big ) &=&
S_{0,k+1} \big (\sqrt{a_\ell(\lambda_{k+1})}\,{z^\ell}\big ) \\ &=& \sqrt {\tfrac{a_{\ell}(\lambda_{k+1})}{a_{\ell+k}(\lambda_{0})}} \tfrac{(\lambda_0)_{\ell+k}}{(\lambda_{k+1})_\ell} \sqrt{a_{\ell+k}(\lambda_0)}\; z^{\ell+k}\\
&=& \tfrac{(\ell+k)!}{\ell!}\sqrt{\tfrac{a_{\ell+k}(\lambda_{0})}{a_{\ell}(\lambda_{k+1})}} e^{(\lambda_{0})}_{\ell+k}.
\end{eqnarray*}
Thus $S_{0,k+1}$ is a forward shift of multiplicity $k.$
We claim that if $\Lambda(t) \geq 2,$ then we can find a forward shift $X$
of multiplicity $k+1,$ namely,
$X(e^{(\lambda_{k+1})}_\ell) = x_{\ell} e^{(\lambda_0)}_{\ell+k+1}$
which has the required intertwining property. Thus evaluating the equation
 $S_{0,0}X-XS_{n-1,k+1}=S_{0,k+1}$ on the vectors $e_{\ell-1}^{(\lambda_{k+1})},\;\ell\geq 0,$ we obtain
\begin{eqnarray}\label{recursive1}
\tfrac{(\ell+k)!}{\ell!} \tfrac{\prod\limits_{i=0}^{\ell-1}w_i^{(\lambda_{k+1})}}{\prod\limits_{i=0}^{\ell+k-1}w_i^{(\lambda_0)}} e^{(\lambda_{0})}_{\ell+k} &=&
\tfrac{(\ell+k)!}{\ell!}\sqrt{\tfrac{a_{\ell+k}(\lambda_{0})}{a_{\ell}(\lambda_{k+1})}} e^{(\lambda_{0})}_{\ell+k}\\
&=&S_{0,k+1} \big (e^{(\lambda_{k+1})}_\ell \big )\nonumber\\
&=&\big (S_{0,0} X - X S_{k+1,k+1}\big )\big ( e^{(\lambda_{k+1})}_\ell\big )\nonumber \\
&=&\big ( x_\ell w_{\ell+k}^{(\lambda_0)} - x_{\ell-1} w_{\ell-1}^{(\lambda_{k+1})}\big )  e^{(\lambda_{0})}_{\ell+k} \label{recursive2}.
\end{eqnarray}
From this we obtain $x_\ell$ recursively:
\begin{eqnarray*}
w_k^{(\lambda_0)} x_0
&=&k!\tfrac{\sqrt{a_k(\lambda_0)}}{\sqrt{a_0(\lambda_{(k+1)})}}\\
\end{eqnarray*}
and for $\ell\geq 1,$
\begin{eqnarray*}
x_\ell=\sqrt{\tfrac{a_{k+\ell}(\lambda_0)}{a_{\ell}(\lambda_{k+1})}}\sum_{i=1}^k (\ell)_i&\sim& \Big ((k+\ell)^{\tfrac{\lambda_0-1}{2}}\Big ) \Big (\ell^{\tfrac{-\lambda_{k+1}+1}{2}}\Big ) (\ell^{k+1})\\
&\sim& \Big ( \ell^{\tfrac{\lambda_0-\lambda_{k+1}+2k+2}{2}}\Big ),
\end{eqnarray*}
where $(\ell)_k := \ell(\ell+1) \cdots (\ell+k-1) = \frac{\Gamma(\ell+k)}{\Gamma(k)}$ is the Pochhammer symbol as before. Here, using the Stirling approximation for the $\Gamma$ function, we infer that $\sum_{i=1}^k (\ell)_i\sim \ell^{k+1}.$

If $\Lambda(t)\geq 2$, then  $\lambda_1-\lambda_0\geq 2,
\lambda_2-\lambda_1\geq 2,\cdots, \lambda_{k+1}-\lambda_k\geq 2.$
Consequently, $\lambda_{k+1}-\lambda_0\geq 2k+2$ making the operator $X$ bounded.

It follows that if  $\Lambda(t)\geq 2,$ then the shift $X$ of multiplicity $n$ that we have constructed is bounded and has the desired intertwining property. To show that there is no such intertwining operator if $\Lambda(t) < 2,$ assume to the contrary the existence of such an operator. Then we show that there must also exist a shift of multiplicity $k+1$ with this property leading to a contradiction. For the proof, suppose
$$X\big (e^{(\lambda_{k+1})}_\ell \big ) = \sum_{i=0}^\infty x_{i,\ell}\, e^{(\lambda_0)}_{i},\; X= \big (\!\!\big ( x_{i,\ell} \big ) \!\!\big ).$$
Then
$$
\big ( S_{0,0} X - X S_{k+1,k+1}\big ) \big (e^{(\lambda_{k+1})}_\ell \big ) = \sum_{i=0}^\infty \big (x_{i+1,\ell+1} w_{i}^{(\lambda_0)} - x_{i,\ell} w_{\ell-1}^{(\lambda_{k+1})} \big )\big ( e^{(\lambda_0)}_{i}\big ).
$$
In particular, we have
$$
(x_{\ell+k+1,\ell+1}\, w_{\ell+k}^{(\lambda_0)} - x_{\ell+k,\ell}\,
w_{\ell-1}^{(\lambda_{k+1})})(e^{(\lambda_0)}_{l+k}) = S_{0,k+1}\big
(e_\ell^{(\lambda_{k+1})}\big ).
$$
Repeating the proof above, we will have the conclusion
$x_{l+k,l}\rightarrow \infty, l\rightarrow \infty$ which proving the
claim.

\end{proof}

\begin{lem}\label{keyLb}
Let $t$ be a quasi-homogeneous holomorphic curve with atoms
$t_i,\,0\leq i \leq {n-1}.$ Let  $T:=\big (\!\!\big ( S_{i,j}\big )\!\! \big )$  be the atomic decomposition of the operator $T$  representing $t$ as in Lemma \ref{atomic}.
\begin{enumerate}
\item If  $\Lambda(t)\in [1+\frac{n-3}{n-1}, 1+\frac{n-2}{n}),$ then for any $1\leq r< n-1,$ we have
$$S_{0,r}S_{r,r+1}\cdots S_{n-2,n-1}\in
\mbox{\rm ran}\,\sigma_{S_{0,0},S_{n-1,n-1}}.$$
\item Suppose that $\Lambda(t)\geq 2$. Then there exists a bounded linear
operator $X\in {\mathcal L}({\mathcal H}_{n-1},{\mathcal H}_{n-2})$
such that
$$S_{n-2,n-2}X-XS_{n-1,n-1}=S_{n-2,n-1}$$ and
$$S_{n-3,n-2}X\in \mbox{\rm ran}\sigma_{S_{n-3,n-3},S_{n-1,n-1}}.$$
\end{enumerate}
\end{lem}

\begin{proof}
We only prove that
$S_{0,n-2}S_{n-2,n-1}$ is in $\mbox{ran}\sigma_{S_{0,0},S_{n-1,n-1}}.$
Clearly, as can be seen from the proof we present below, the proof in all the other cases are exactly the same.

Let $T_{0}$, $T_{n-2}$ and  $T_{n-1}$ be the operators induced by
$s_{0,0}$, $s_{n-2,n-2}$ and $s_{n-1}$ as in in Lemma \ref{atomic}.
These are then necessarily the operators ${M^{(\lambda_0)}}^*$,
${M^{(\lambda_{n-2})}}^*$  and ${M^{(\lambda_{n-1})}}^*$ acting on
the weighted Bergman spaces
 $\mathbb A^{(\lambda_0)}(\mathbb D)$, $\mathbb A^{(\lambda_{n-2})}(\mathbb D)$ and $\mathbb A^{(\lambda_{n-1})}(\mathbb D),$ respectively.

As in the proof of Lemma \ref{keyL}, equations \eqref{recursive1} and \eqref{recursive2}, we have that
\begin{eqnarray*}
S_{0,n-2} \big (e^{(\lambda_{n-2})}_\ell \big ) &=&
S_{0,n-2} \big (\sqrt{a_\ell(\lambda_{n-2})}\,{z^\ell}\big ) \\ &=& \sqrt {\tfrac{a_{\ell}(\lambda_{n-2})}{a_{\ell+n-3}(\lambda_{0})}} \tfrac{(\lambda_0)_{\ell+n-3}}{(\lambda_{n-2})_\ell} \sqrt{a_{\ell+n-3}(\lambda_0)}\; z^{\ell+n-3}\\
&=&
\tfrac{(\ell+n-3)!}{\ell!}\sqrt{\tfrac{a_{\ell+n-3}(\lambda_{0})}{a_{\ell}(\lambda_{n-2})}}
e^{(\lambda_{0})}_{\ell+n-3}
\end{eqnarray*}

$$S_{n-2,n-1}(e^{(\lambda_{n-1})}_{\ell})=\tfrac{\sqrt{a_{\ell}(\lambda_{n-2})}}{\sqrt{a_{\ell}(\lambda_{n-1})}}e^{(\lambda_{n-2})}_{\ell};$$
and
$$S_{0,n-2}S_{n-2,n-1}(e^{(\lambda_{n-1})}_{\ell})= \tfrac{(\ell+n-3)!}{\ell!}\sqrt{\tfrac{a_{\ell+n-3}(\lambda_{0})}{a_{\ell}(\lambda_{n-1})}}
e^{(\lambda_{0})}_{\ell+n-3}.$$

 Thus $S_{0,n-2}S_{n-2,n-1}$ is a forward shift of multiplicity
 $n-3$.
We claim that if $\Lambda(t) \geq 1+\frac{n-3}{n-1},$ then we can
find a forward shift $X$ of multiplicity $n-2,$ namely,
$X(e^{(\lambda_{n-1})}_\ell) = x_{\ell} e^{(\lambda_0)}_{\ell+n-2}$
which has the required intertwining property. Thus evaluating the
equation $S_{0,0}X-XS_{n-1,n-1}=S_{0,n-1}$ on the vectors
$e^{(\lambda_{n-1})}_\ell,\;\ell\geq 0,$ we obtain

\begin{eqnarray*}
w_{n-3}^{(\lambda_0)} x_0
&=&(n-3)!\tfrac{\sqrt{a_{n-3}(\lambda_0)}}{\sqrt{a_0(\lambda_{(n-1)})}}\\
\end{eqnarray*}
and for $\ell\geq 1,$ we have that

$$w_{l+n-3}^{(\lambda_0)}x_{\ell}-x_{\ell-1}w_{l}^{(\lambda_n-1)}=\tfrac{(\ell+n-3)!}{\ell
!}\tfrac{\sqrt{a_{\ell+n-3}(\lambda_0)}}{\sqrt{a_{\ell}(\lambda_{(n-1)})}}.$$

It follows that \begin{eqnarray*}
x_\ell=\tfrac{\sqrt{a_{\ell+n-3}(\lambda_0)}}{\sqrt{a_{\ell}(\lambda_{n-1})}}\sum_{i=1}^{n-3} (\ell)_i&\sim& \Big ((n-3+\ell)^{\tfrac{\lambda_0-1}{2}}\Big ) \Big (\ell^{\tfrac{-\lambda_{n-1}+1}{2}}\Big ) (\ell^{n-2})\\
&\sim& \Big ( \ell^{\tfrac{\lambda_0-\lambda_{n-1}+2n-4}{2}}\Big ).
\end{eqnarray*}

Note that when $\Lambda(t)>1+\frac{n-3}{n-1}$, we obtain
$$\lambda_{n-1}-\lambda_0=(n-1)\Lambda(t)>(n-1)\frac{2n-4}{n-1}=2n-4$$
making $X$ bounded. This completes the proof of the first statement.

For the proof of the second statement, note that by virtue of Lemma \ref{keyL}, we have $S_{n-2,n-1}\in
\mbox{Ran}\sigma_{S_{n-2,n-1}}.$ So there exists a bounded operator  $X$  such that
$$S_{n-2,n-2}X-XS_{n-1,n-1}=S_{n-2,n-1}.$$
Repeating the proof for the first part, we conclude
$$S_{n-3,n-2}X\in \mbox{\rm ran}\,\sigma_{S_{n-3,n-3},S_{n-1,n-1}}.$$
\end{proof}

\subsection{\sf Strong irreducibility} We  now show that a quasi-homogeneous holomorphic curve $t$ is strongly irreducible or strongly reducible according as $\Lambda(t)$ is less than $2$ or greater equal to $2.$ We recall that homogeneous operators (in this case, $\Lambda(t) =2$) were shown to be irreducible but strongly reducible in \cite{KM}
\begin{lem} \label{Sirrd}
Fix  a quasi-homogeneous holomorphic curve $t$ with atoms $t_i$ and let $T=\big (\!\! \big (S_{i,j}\big )\!\!\big )$ be its atomic decomposition.
\begin{enumerate}
\item If $\Lambda(t)\geq 2$, then $T$ is strongly reducible, indeed
$T$ is similar to the direct sum of its atoms, namely, $\bigoplus\limits_{i=0}^{n-1} T_i$ and
\item if $\Lambda(t) < 2,$ then $T$ is strongly irreducible.
\end{enumerate}
\end{lem}
\begin{proof}
If $\Lambda(t) \geq 2,$ then we claim that the operator $T$ is similar to $T_0\oplus T_1\oplus \cdots \oplus T_{n-1}.$

When $n=2$, Let
$T=\left (\begin{smallmatrix}S_{0,0}&S_{0,1}\\0&S_{1,1}\end{smallmatrix}\right )$. By Lemma
\ref{keyL}, there exists $X_{0,1}$ such that
$$S_{0,0}X_{0,1}-X_{0,1}S_{1,1}=S_{0,1}.$$
Set  $Y_{0,1}=\left (\begin{smallmatrix} I&X_{0,1}\\
0&I\end{smallmatrix} \right )$, then we have that
$$\begin{array}{lllll}Y_{0,1}TY^{-1}_{0,1}&=&\begin{pmatrix}S_{0,0}&S_{0,1}+X_{0,1}S_{1,1}\\0&S_{1,1}\end{pmatrix}\begin{pmatrix}I&-X_{0,1}\\0&I\end{pmatrix}\\
&=&\begin{pmatrix}S_{0,0}&S_{0,1}-S_{0,0}X_{0,1}+X_{0,1}S_{1,1}\\
0&S_{1,1}\end{pmatrix}\\
&=&\begin{pmatrix}S_{0,0}&0\\
0&S_{1,1}\end{pmatrix}\\
\end{array}$$
Notice that $Y_{0,1}$ is invertible,  we have that $T\sim
S_{0,0}\oplus S_{1,1}$.

In this case, using Lemma \ref{keyL}, we find an invertible  bounded linear operator $X_{0,n-1}$ such that
$$S_{0,0}X_{0,n-1}-X_{0,n-1}S_{n-1,n-1}=S_{0,n-1}.$$  For any $i<j,$ applying Lemma \ref{keyL} to the operators
$$\left ( \begin{smallmatrix}
S_{i,i}&S_{i,i+1}&S_{i,i+2}&\cdots&S_{i,j} \\
0&S_{i+1,i+1}&S_{i+1,i+2}&\cdots&S_{i+1,j}\\
\vdots &\ddots&\ddots&\ddots&\vdots\\
0&\ldots&0&S_{j-1,j-1}&S_{j-1,j}\\
0&0&\ldots&0&S_{j,j}\\
\end{smallmatrix}\right ),
$$
we find an invertible bounded linear operator $X_{i,j}$ such that
$S_{i,i}X_{i,j}-X_{i,j}S_{j,j}=S_{i,j}.$ Set $Y_{n-2,n-1}:=\left (
\begin{array}{c | c} I^{(n-2)} & 0\\\hline 0&\begin{array}{cc}
I & X_{n-2,n-1} \\
0 & I
\end{array}
\end{array} \right )$ and note that
$Y_{n-2,n-1}^{-1} = \left ( \begin{array}{c | c}
I^{(n-2)} & 0\\\hline
0&\begin{array}{cc}
I &\phantom{G}- X_{n-2,n-1} \\
0 & I
\end{array}
\end{array} \right ).$
%
%
Now, we have
$$\left ( \begin{array}{c | c}
I^{(n-2)} & 0\\\hline
0&\begin{array}{cc}
I & X_{n-2,n-1} \\
0 & I
\end{array}
\end{array} \right )
\left ( \begin{smallmatrix}
S_{0,0}&S_{0,1}&S_{0,2}&\cdots&S_{0,n-1} \\
0&S_{1,1}&S_{1,2}&\cdots&S_{1,n-1}\\
\vdots &\ddots&\ddots&\ddots&\vdots\\
0&\ldots&0&S_{n-2,n-2}&S_{n-2,n-1}\\
0&0&\ldots&0&S_{n-1,n-1}\\
\end{smallmatrix}\right ) \left ( \begin{array}{c | c}
I^{(n-2)} & 0\\\hline 0&\begin{array}{cc}
I & \phantom{G}- X_{n-2,n-1} \\
0 & I
\end{array}
\end{array} \right )$$
$$=\left ( \begin{smallmatrix}
S_{0,0}&S_{0,1}&S_{0,2}&\cdots&S_{0,n-1}-S_{0,n-2}X_{n-2,n-1} \\
0&\ddots&\ddots&\ddots&\vdots\\
\vdots&\ddots&S_{n-3,n-3}&S_{n-3,n-2}&S_{n-3,n-1}-S_{n-3,n-2}X_{n-2,n-1}\\
0&\ldots&0&S_{n-2,n-2}&0\\
0&\ldots&\ldots&0&S_{n-1,n-1}\\
\end{smallmatrix}\right ).
$$
By Lemma \ref{keyLb}, we have  $$S_{n-3,n-2}X_{n-2,n-1}\in
\mbox{\rm ran}\,\sigma_{S_{n-1,n-1},S_{n-3,n-3}}.$$
Therefore, there exists an
invertible bounded linear operator $\widetilde{X}$ such that
$$S_{n-3,n-3}\widetilde{X}-\widetilde{X}S_{n-1,n-1}=S_{n-3,n-1}-S_{n-3,n-2}X_{n-2,n-1}.$$
Let $X_{n-3,n-1}:=\widetilde{X}$ and
$Y_{n-3,n-1}= \left ( \begin{array}{c | c }
I^{(n-3)} & 0\\\hline
0&\begin{array}{ccc}
I & 0 & X_{n-3,n-2} \\
0 & I & 0 \\
0 & 0 & I
\end{array}
\end{array} \right ).$
Now, we have $$Y_{n-3,n-1}\left ( \begin{smallmatrix}
S_{0,0}&S_{0,1}&S_{0,2}&\cdots&S_{0,n-1}-S_{0,n-2}X_{n-2,n-1} \\
0&\ddots&\ddots&\ddots&\vdots\\
\vdots&\ddots&S_{n-3,n-3}&S_{n-3,n-2}&S_{n-3,n-1}-S_{n-3,n-2}X_{n-2,n-1}\\
0&\ldots&0&S_{n-2,n-2}&0\\
0&\ldots&\ldots&0&S_{n-1,n-1}\\
\end{smallmatrix}\right )
Y_{n-3,n-1}^{-1}$$
$$=\left ( \begin{smallmatrix}
S_{0,0}&S_{0,1}&S_{0,2}&\cdots&S_{0,n-1}-S_{0,n-2}X_{n-2,n-1} -S_{0,n-3}X_{n-3,n-1}\\
0&\ddots&\ddots&\ddots&\vdots\\
\vdots&\ddots&S_{n-3,n-3}&\cdots&0\\
0&\ldots&0&S_{n-2,n-2}&0\\
0&\ldots&\ldots&0&S_{n-1,n-1}\\
\end{smallmatrix}\right ).
$$
Continuing in this manner, we clearly have $$\left ( \begin{smallmatrix}
S_{0,0}&S_{0,1}&S_{0,2}&\cdots&S_{0,n-1}\\
0&\ddots&\ddots&\ddots&\vdots\\
\vdots&\ddots&S_{n-3,n-3}&S_{n-3,n-2}&S_{n-3,n-1}\\
0&\ldots&0&S_{n-2,n-2}&S_{n-2,n-1}\\
0&\ldots&\ldots&0&S_{n-1,n-1}\\
\end{smallmatrix}\right )\sim \left ( \begin{smallmatrix}
S_{0,0}&S_{0,1}&\cdots&S_{0,n-2}&0 \\
0&\ddots&\ddots&\ddots&\vdots\\
\vdots&\ddots&S_{n-3,n-3}&S_{n-3,n-2}&0\\
0&\ldots&0&S_{n-2,n-2}&0\\
0&\ldots&\ldots&0&S_{n-1,n-1}
\end{smallmatrix}\right ).$$
This completes the proof of the induction step. We have therefore proved the first statement.

To prove the second statement, assuming that $\Lambda(t) < 2,$ we must show that $E_t$ is  strongly irreducible.
First, we  prove that $E_t$ is irreducible. By Lemma \ref{utXY}, any
projection $P=(P_{i,j})_{n\times n}$ in ${\mathcal A}^{\prime}(E_t)$
is  diagonal. Thus
$$P_{i,i}^2=P_{i,i}\in {\mathcal A}^{\prime}(E_{t_i}).$$ It follows
that for any $0\leq i\leq n-1,$ $P_{i,i}=0$ or $P_{i,i}=I$. Since
$PS=SP$, we have
$$P_{i,i}S_{i,i+1}=S_{i,i+1}P_{i+1,i+1}.$$ Therefore
$$P_{i,i}=P_{j,j}, i,j=0,1,\cdots, n-1.$$ Consequently,  $P=0$ or $P=I$ and  $E_t$ is irreducible.

We first prove that $E_t$ is also strongly irreducible for $n=2.$ By
Lemma \ref{keyL}, we have $$S_{0,1} \not\in
\mbox{ran}\;\sigma_{S_{0,0}S_{1,1}}.$$
 Let $P\in
{\mathcal A}^{\prime}(E_t)$ be an idempotent. By Lemma \ref{intS},
$P$  has the following form $$P=\begin{pmatrix}
P_{0,0} & P_{0,1} \\
0 & P_{1,1} \\
\end{pmatrix}.$$ Since $PS=SP$,  we have  $$P_{0,0}S_{0,0}= S_{0,0}P_{0,0}, P_{1,1}S_{1,1}=
S_{1,1}P_{1,1}$$
 and $$P_{00}S_{0,1}-S_{0,1} P_{11}= S_{0,0}P_{0,1}-P_{0,1}S_{1,1}.$$
Since $P_{i,i}\in\{S_{i,i}\}^{\prime}$, for $0\leq i\leq 1$, so
$P_{i,i}$ can be either $I$ or $0$. If either $P_{1,1}=I$,
$P_{2,2}=0$ or $P_{0,0}=0$, $P_{1,1}=I$, then $S_{0,1} \in
\mbox{Ran}\;\sigma_{S_{0,0},S_{1,1}}$ which is a contradiction to
our conclusion that $S\not\in \mbox{ran}\;\sigma_{S_{0,0},S_{1,1}}$.
Thus the form of P will be
$$\begin{pmatrix}
I & P_{0,1} \\
0 & I \\ \end{pmatrix}~~\mbox{ or}~~ \begin{pmatrix}
0 & P_{0,1} \\
0 & 0 \\
\end{pmatrix}.$$ Since $P$ is an idempotent operator, so we have $P_{0,1}=0$. Hence $E_t$ is strongly
irreducible.

To complete the proof of the second statement by induction, suppose
that it is valid for any $n\leq k-1.$ For $n=k$, let $P\in {\mathcal
A}^{\prime}(E_t)$ be an idempotent operator. By Lemma \ref{utXY},
$P$ has the following form:
\renewcommand \arraystretch{0.875}
$$P=\left (\begin{matrix}
P_{0,0} & P_{0,1} & P_{0,2}&\cdots&P_{0,k}\\
0&P_{1,1}&P_{1,2}&\cdots&P_{1,k} \\
\vdots&\ddots&\ddots&\ddots&\vdots\\
0&\ldots&0&P_{k-1,k-1}&P_{k-1,k}\\
0&\ldots&\ldots&0&P_{k,k}\\
\end{matrix}\right ),\renewcommand \arraystretch{1}$$
and $P\big (\!\! \big ( S_{i,j}\big )\!\! \big )_{k\times k}=\big (\!\! \big ( S_{i,j}\big )\!\! \big )_{k\times k}P$. It follows that
$$\big (\!\! \big ( P_{i,j}\big )\!\! \big )_{i,j=0}^{k-1}\big (\!\! \big ( S_{i,j}\big )\!\! \big )_{i,j=0}^{k-1}=\big (\!\! \big ( S_{i,j}\big )\!\! \big )_{i,j=0}^{k-1}\big (\!\! \big ( P_{i,j}\big )\!\! \big )_{i,j=0}^{k-1},
\big (\!\! \big ( P_{i,j}\big )\!\! \big )_{i,j=1}^{k}\big (\!\! \big ( S_{i,j}\big )\!\! \big )_{i,j=1}^{k}=\big (\!\! \big ( S_{i,j}\big )\!\! \big )_{i,j=1}^{k}\big (\!\! \big ( P_{i,j}\big )\!\! \big )_{i,j=1}^{k}.$$
Both $\big (\!\! \big ( P_{ij}\big )\!\! \big )_{i,j=0}^{k-1}$ and $\big (\!\! \big ( P_{i,j}\big )\!\! \big )_{i,j=1}^{k}$ are idempotents.  Since $\Lambda(t)<2$, we have
$$S_{r,s}\not\in\mbox{\rm ran}\,\sigma_{S_{r,r},S_{s,s}},\; r,s \leq n.$$
By the induction hypothesis,  we have $$P_{i,j}=0, i\neq j\leq k-1,$$ and
$$P_{0,0}=P_{1,1}=\cdots=P_{k,k}=0,~~\mbox{or}~~P_{0,0}=P_{1,1}=\cdots=P_{k,k}=I.$$
Thus  $P$ has the following form:
\renewcommand \arraystretch{0.875}
$$P=\left (\begin{matrix}
\,I\, & \,0 \,&\, 0\,&\,\cdots\,&\,P_{0,k}\,\\
0&I&0&\cdots&0 \\
\vdots&\ddots&\ddots&\ddots&\vdots\\
0&\cdots&0&I&0\\
0&0&\ldots&0&I\\
\end{matrix}\right )~~\mbox{or}~~
P=\left (\begin{matrix}
\,0\, & \,0 \,& \,0\,&\,\cdots\,&\,P_{0,k}\,\\
0&0&0&\cdots&0 \\
\vdots&\vdots&\ddots&\ddots&\vdots\\
\vdots&0&\cdots&0&0\\
0&\cdots&\cdots&\cdots&0\\
\end{matrix} \right ). $$
\renewcommand \arraystretch{1}
Since $P$ is an idempotent, it follows that $P_{0,k}=0$.
\end{proof}
By Lemma \ref{utXY}, an intertwining operator between two quasi-homogeneous operators with respect to any atomic decomposition must be upper triangular. Thus any operator $X$ in the commutant of such an operator, say $T,$ must also
be upper-triangular. In particular, $X_{i,i}$ belongs to the
commutant of $S_{i,i},$ $0\leq i \leq n-1.$ Since $S_{i,i}$ is a homogeneous operator in $B_1(\mathbb D),$ it follows that
the commutant of $S_{i,i}$ is isomorphic to
${\mathcal H}^{\infty}(\mathbb{D}),$ the space of bounded analytic functions on the unit disc $\mathbb D$. Consequently, for any $\phi\in {\mathcal H}^{\infty}(\mathbb{D})$, the operator
$\phi(S_{i,i})$ is in the commutant ${\mathcal
A}^{\prime}(S_{i,i}).$ In the following lemma, we give a
description of the commutant of $T$. We will construct an operator
$X$ in the commutant of $T,$ where the diagonal elements are induced by the same holomorphic function $\phi\in {\mathcal
H}^{\infty}(\mathbb{D}),$ that is, $\phi(S_{i,i})=X_{i,i}$.

\begin{lem} \label{commutant2}
Let $t$ be a quasi-homogeneous holomorphic curve with atoms $t_i, 0\leq i \leq 1$. Let $T=\big (\!\! \big
(S_{i,j}\big )\!\!\big )$ be its atomic decomposition. Suppose that
$X=\big (\!\! \big (X_{i,j}\big )\!\!\big )$ is in ${\mathcal
A}^{\prime}(T).$ Then there exists $\phi\in {\mathcal
H}^{\infty}(\mathbb{D})$ such that $X_{i,i}=\phi(S_{i,i}),i=0,1$ and we also have that
$$S_{0,0}X_{0,1}-X_{0,1}S_{1,1}=X_{0,0}S_{0,1}-S_{0,1}X_{1,1}=0.$$
In particular, $X_{0,1}$ can be chosen as zero.
\end{lem}

\begin{proof}

Set $X=\big (\!\! \big (X_{i,j}\big )\!\!\big )\in {\mathcal
A}^{\prime}(T)$, we have the following equation
$$\begin{pmatrix}
S_{0,0}&S_{0,1}\\0&S_{1,1}\\
\end{pmatrix}\begin{pmatrix}
X_{0,0}&X_{0,1}\\X_{1,0}&X_{1,1}\\
\end{pmatrix}=\begin{pmatrix}
X_{0,0}&X_{0,1}\\X_{1,0}&X_{1,1}\\
\end{pmatrix}\begin{pmatrix}
S_{0,0}&S_{0,1}\\0&S_{1,1}\\
\end{pmatrix}.$$
By Lemma \ref{commutant}, we have $X_{1,0}=0$.  Then
$$S_{0,0}X_{0,1}+S_{0,1}X_{1,1}=X_{0,0}S_{0,1}+X_{0,1}S_{1,1},$$
and
$$S_{0,0}X_{0,1}-X_{0,1}S_{1,1}=X_{0,0}S_{0,1}-S_{0,1}X_{1,1}.$$
Note that there exist holomorphic functions $\phi_{0,0}$ and
$\phi_{1,1}$ such that
$$X_{0,0}(t_0)=\phi_{0,0}(t_0),  X_{1,1}(t_1)=\phi_{1,1}(t_1),$$
and by the definition of $S_{0,1}$, there exist constant function
$\phi_{0,1}$ such that
$$S_{0,1}(t_1)=\phi_{0,1}t_0.$$ Then
$$X_{0,0}S_{0,1}(t_1)-S_{0,1}X_{1,1}(t_1)=(\phi_{0,0}\phi_{0,1}-\phi_{1,1}\phi_{0,1})t_0.$$
and $X_{0,0}S_{0,1}-S_{0,1}X_{1,1}$ also intertwines $S_{0,0}$ and
$S_{1,1}$.  Taking $X_{0,0}S_{0,1}-S_{0,1}X_{1,1}$ the place of
$S_{0,1}$ and using the proof of Lemma 3.2, we might deduce that
$$S_{0,0}X_{0,1}-X_{0,1}S_{1,1}=X_{0,0}S_{0,1}-S_{0,1}X_{1,1}=0, \phi_{0,0}=\phi_{1,1}.$$
Thus, we can choose $X_{0,1}=0$ and there exists
$\phi=\phi_{0,0}=\phi_{1,1}\in {\mathcal H}^{\infty}(\mathbb{D})$
such that $X=\begin{pmatrix}
X_{0,0}&0\\0&X_{1,1}\\
\end{pmatrix}$ where $X_{i,i}=\phi(S_{i,i})$ satisfies that
$$\begin{pmatrix}
S_{0,0}&S_{0,1}\\0&S_{1,1}\\
\end{pmatrix}\begin{pmatrix}
X_{0,0}&0\\0&X_{1,1}\\
\end{pmatrix}=\begin{pmatrix}
X_{0,0}&0\\0&X_{1,1}\\
\end{pmatrix}\begin{pmatrix}
S_{0,0}&S_{0,1}\\0&S_{1,1}\\
\end{pmatrix}.$$

\end{proof}

\begin{lem} \label{ProfJiang}
Let $t$ be a quasi-homogeneous holomorphic curve with atoms
 $t_i, 0\leq i \leq n-1$. Let $T=\big (\!\! \big
(S_{i,j}\big )\!\!\big )$ be its atomic decomposition. Let $\phi\in
{\mathcal H}^{\infty}(\mathbb{D})$ be a holomorphic function. If
$\Lambda (t) < 2,$ then there exists a bounded linear operator $X\in
{\mathcal A}^{\prime}(T)$ such that $X_{i,i}=\phi(S_{i,i}),
i=0,1,\cdots, n-1.$
\end{lem}

\begin{proof}

Firstly, by Lemma \ref{commutant2}, the lemma is true for the case
of $n=2$.


For $n=3,$  let $X=\left ( \begin{smallmatrix}
X_{0,0} & X_{0,1} & X_{0,2}\\
0&X_{1,1}&X_{1,2}\\
0&0&X_{2,2}\\
\end{smallmatrix} \right ) \in {\mathcal
A}^{\prime}(E_t).$ Then we have $$\begin{pmatrix}
S_{0,0} & S_{0,1} &S_{0,2}\\
0&S_{1,1}&S_{1,2}\\
0&0&S_{2,2}\\
\end{pmatrix}\begin{pmatrix}
X_{0,0} & X_{0,1} & X_{0,2}\\
0&X_{1,1}&X_{1,2}\\
0&0&X_{2,2}\\
\end{pmatrix}=\begin{pmatrix}
X_{0,0} & X_{0,1} & X_{0,2}\\
0&X_{1,1}&X_{1,2}\\
0&0&X_{2,2}\\
\end{pmatrix}\begin{pmatrix}
S_{0,0} & S_{0,1} &S_{0,2}\\
0&S_{1,1}&S_{1,2}\\
0&0&S_{2,2}\\
\end{pmatrix}$$ and it follows that
\begin{enumerate}
\item $S_{0,0}X_{0,1}+S_{0,1}X_{1,1}=X_{0,0}S_{0,1}+X_{0,1}S_{1,1}, $ that is, $S_{0,0}X_{0,1}-X_{0,1}S_{1,1}=X_{0,0}S_{0,1}-S_{0,1}X_{1,1};$
\item $S_{1,1}X_{1,2}+S_{1,2}X_{2,2}=X_{1,1}S_{1,2}+X_{1,2}S_{2,2},$ that is, $S_{1,1}X_{1,2}-X_{1,2}S_{2,2}=X_{1,1}S_{1,2}-S_{1,2}X_{2,2}.$
\end{enumerate}
By  Lemma \ref{commutant2}, we may choose, without loss of
generality, $X_{0,1}=0$ and $X_{1,2}=0$.  And there exists $\phi\in
{\mathcal H}^{\infty}(\mathbb{D})$ such that
$X_{i,i}=\phi(S_{i,i}),i=0,1,2.$
It is therefore enough to find  an operator $X_{0,2}$   satisfying
 $$S_{0,0}X_{0,2}-X_{0,2}S_{2,2}=X_{0,0}S_{0,2}-S_{0,2}X_{2,2}.$$
Clearly, we have
\begin{eqnarray*}
(X_{0,0}S_{0,2}-S_{0,2}X_{2,2})(t_2(w))&=&X_{0,0}(m_{0,2}t^{(1)}_0(w))-S_{0,2}(\phi(w)t_2(w))\\
&=&m_{0,2}(\phi(w)t_0(w))^{(1)}-m_{0,2}\phi(w)t^{(1)}(w)\\
&=&m_{0,2}\phi^{(1)}(w)t_0(w).
\end{eqnarray*}
We therefore set  $X_{0,2}$ be the operator:
$X_{0,2}(t_2(w))=m_{0,2}\phi^{(1)}(w)t^{(1)}_0(w).$

To complete the proof by induction, we  assume that we have the validity of the conclusion  for $n=k.$ Thus we assume the existence of  a bounded linear operator
$X=\big (\!\! \big (X_{i,j}\big )\!\! \big )$ such that
$\big (\!\! \big (S_{i,j}\big )\!\! \big )\big (\!\! \big (X_{i,j}\big )\!\! \big )=\big (\!\! \big (X_{i,j}\big )\!\! \big )\big (\!\! \big (S_{i,j}\big )\!\! \big )$
where $X_{i,i}=\phi(S_{i,i})$ and $X_{i,i+1}=0$. And there exists
$l^r_{i,j}$ such that
$X_{i,j}(t_j)=\sum\limits_{r=1}^{j-i-1}l^r_{i,j}\phi^{(j-k)}t^{(k)}_i$.
To complete the inductive step, we only
need to find the operator $X_{0,k}$  satisfying the following equation:
\begin{equation} \label{intX0k}
S_{0,0}X_{0,k}-X_{0,k}S_{k,k}=X_{0,0}S_{0,k}-S_{0,k}X_{k,k}+(\sum\limits_{i=2}^{k-1}X_{0,i}S_{i,k}-\sum\limits_{i=1}^{k-2}S_{0,i}X_{i,k}) \end{equation}
Note that the induction hypothesis ensures the existence of constants
$c^{s}_{0,k}$ (depending on $m_{i,j}$) such that
$$(X_{0,0}S_{0,k}-S_{0,k}X_{k,k}+\sum\limits_{i=2}^{k-1}X_{0,i}S_{i,k}-\sum\limits_{i=1}^{k-2}S_{0,i}X_{i,k})(t_k)=\sum\limits_{s=1}^{k-1}c^{s}_{0,k}\phi^{(s)}t^{(k-s-1)}_0.$$

Now, suppose that
$X_{0,k}(t_k)=\sum\limits_{s=1}^{k-1}l^{s}_{0,k}\phi^{(s)}t^{(k-s)}_0$, where the constants $l^{s}_{0,k}$ are to be found. Then we must have
\begin{eqnarray*}
(S_{0,0}X_{0,k}-X_{0,k}S_{k,k})(t_k(w))&=&S_{0,0}(\sum\limits_{s=1}^{k-1}l^{s}_{0,k}\phi^{(s)}t^{(k-s)}_0(w))-w(\sum\limits_{s=1}^{k-1}l^{s}_{0,k}\phi^{(s)}t^{(k-s)}_0(w))\\
&=&\sum\limits_{s=1}^{k-1}l^{s}_{0,k}\phi^{(s)}(wt^{(k-s)}_0(w)+(k-s)t^{(k-s-1)}(w))-w(\sum\limits_{s=1}^{k-1}l^{s}_{0,k}\phi^{(s)}t^{(k-s)}_0(w))\\
&=&\sum\limits_{s=1}^{k-1}l^{s}_{0,k}(k-s)\phi^{(s)}t^{(k-s-1)}(w)\\
&=&\sum\limits_{s=1}^{k-1}c^{s}_{0,k}\phi^{(s)}t^{(k-1-s)}_0(w)
\end{eqnarray*}
It follows that if we choose
$l^{s}_{0,k}=\frac{c^{s}_{0,k}}{k-s},$ then $X_{0,k}$ with this choice of the constants  validates equation \eqref{intX0k}.
This completes the induction step.

In particular, when $\mu_{i,j}$ are all chosen to be $1$, and then
$m_{i,j}=-1$, i.e.  $S_{i,j}(t_j)=-t^{(j-i-1)}_j$. In this case,
$X_{0,k}(t_0)=-\sum\limits_{s=1}^{k-1}\phi^{(s)}t^{(k-s)}_0$. And if
$m_{i,j}=-1,i,j=0,1,\cdots,n-1$, then by a same argument, we have
that
\begin{equation} \label{keyformulae}
 X_{i,j}(t_j)=-\sum\limits_{s=1}^{j-i-1}\phi^{(s)} t^{(j-i-s)}_i, i,j=0,1,\cdots,n-1.
\end{equation}
\end{proof}

\subsection{\sf Proof of the main theorem}

\begin{proof}[Proof of Theorem \ref{mainthmSim}]
First, if ``$\Lambda(t)\geq 2$'', then the first conclusion of the theorem follows from Lemma 3.4. So, it remains for us to verify the second statement of the theorem, where $\Lambda(t) < 2.$

Let $T$ and $\widetilde{T}$  be the operators representing $t$ and
$\widetilde{t}$ respectively.
Recall that $S_{i,j}(t_j)=m_{i,j}t^{(j-i-1)}_i,$
$\widetilde{S}_{i,j}(t_j)=\widetilde{m}_{i,j}t^{(j-i-1)}_i.$ Up to
similarity, we can assume that $m_{i,i+1}=\widetilde{m}_{i,i+1}.$
Then  $T$ and $\widetilde{T}$ have the following atomic decomposition:
$$T=\left ( \begin{smallmatrix}
S_{0,0}&S_{0,1}&S_{0,2}&\cdots&S_{0,n-1} \\
0&S_{1,1}&S_{1,2}&\cdots&S_{1,n-1}\\
\vdots &\ddots&\ddots&\ddots&\vdots\\
0&\ldots&0&S_{n-2,n-2}&S_{n-2,n-1}\\
0&0&\ldots&0&S_{n-1,n-1}\\
\end{smallmatrix} \right )\;\mbox{\rm and}\; \widetilde{T}=\left (\begin{smallmatrix}
S_{0,0}&S_{0,1}&c_{0,2}S_{0,2}&\cdots&c_{0,n-1}S_{0,n-1} \\
0&S_{1,1}&S_{1,2}&\cdots&c_{1,n-1}S_{1,n-1}\\
\vdots &\ddots&\ddots&\ddots&\vdots\\
0&\ldots&0&S_{n-2,n-2}&c_{n-2,n-1}S_{n-2,n-1}\\
0&0&\ldots&0&S_{n-1,n-1}\\
\end{smallmatrix}\right )$$
Set $c_{i,j}=\frac{\widetilde{m}_{i,j}}{m_{i,j}}$.  Now it is enough to prove the  Claim stated below.

{\sf Claim:}\,\, If $T\sim \widetilde{T},$ then $c_{i,j}=1, i,j=0,1,\cdots, n$.

\noindent Consider the following possibilities:
\begin{enumerate}
\item $\Lambda(t)\in [0, 1)$

\item  $n=3$, $\Lambda(t)\in [1,2)$;   $n>3$,
$\Lambda(t)\in [1, \frac{4}{3})$

\item $n=4$, $\Lambda(t)\in [\frac{4}{3},2)$;  $n>4$,
$\Lambda(t)\in [\frac{4}{3}, \frac{3}{2})$

\item  $n=5$,$\Lambda(t)\in [\frac{3}{2}, 2)$;  $n>5$,
$\Lambda(t)\in [\frac{3}{2}, \frac{8}{5})$
\end{enumerate}
The proofs of the remaining cases are similar.

 In the following,
without loss of generality,  we will always choose
$m_{i,j}=-1,i,j=0,1,\cdots,n-1.$

Case (1): By Lemma \ref{bdd}, we have
$$ T=\widetilde{T}=\left (\begin{smallmatrix}
S_{0,0} & S_{0,1} & 0&\cdots&0\\
&S_{1,1}&S_{1,2}&\cdots&0 \\
&&\ddots&\ddots&\vdots\\
&\bigzero&&S_{n-2,n-2}&S_{n-1,n}\\
&&&&S_{n-1,n-1}\\
\end{smallmatrix}\right ).$$
In this case, we clearly have $K_{t_i}=K_{s_i}$ and
$\theta_{i,i+1}=\widetilde{\theta}_{i,i+1}, i=0,1,\cdots, n-1.$

Case (2): By Lemma 3.2, we have
$$ T=\left (\begin{smallmatrix}
S_{0,0} & S_{0,1} & S_{0,2}&\cdots&0&0\\
&S_{1,1}&S_{1,2}&S_{1,3}&\cdots&0 \\
&&\ddots&\ddots&\ddots&\vdots\\
&\bigzero&&&S_{n-2,n-2}&S_{n-1,n}\\
&&&&&S_{n-1,n-1}\\
\end{smallmatrix} \right ).$$
In the case, by Lemma \ref{bdd},  we first assume that $n=3$. Then  we have
\begin{equation} \begin{pmatrix}
S_{0,0} & S_{0,1} & S_{0,2}\\
0&S_{1,1}&S_{1,2}&\\
0&0&S_{2,2}\\
\end{pmatrix}\begin{pmatrix}
X_{0,0} & X_{0,1} & X_{0,2}\\
0&X_{1,1}&X_{1,2}&\\
0&0&X_{2,2}\\
\end{pmatrix}=
\begin{pmatrix}
X_{0,0} & X_{0,1} & X_{0,2}\\
0&X_{1,1}&X_{1,2}&\\
0&0&X_{2,2}\\
\end{pmatrix}
\begin{pmatrix}
S_{0,0} & S_{0,1} & c_{0,2}S_{0,2}\\
0&S_{1,1}&S_{1,2}&\\
0&0&S_{2,2}\\
\end{pmatrix} \label{(3.1.5)}
\end{equation}
By Lemma \ref{commutant2}, $X_{0,1}$ and $X_{1,2}$ may be chosen to
be zero.
 For the
general case, we may also choose $X_{i,i+1}, i=0,1,\cdots, n-1$ to
be zero by repeating the same argument.

Now we have the following equality
$$\begin{pmatrix}
S_{0,0} & S_{0,1} & S_{0,2}\\
0&S_{1,1}&S_{1,2}&\\
0&0&S_{2,2}\\
\end{pmatrix}\begin{pmatrix}
X_{0,0} & 0 & X_{0,2}\\
0&X_{1,1}&0&\\
0&0&X_{2,2}\\
\end{pmatrix}=
\begin{pmatrix}
X_{0,0} & 0 & X_{0,2}\\
0&X_{1,1}&0&\\
0&0&X_{2,2}\\
\end{pmatrix}
\begin{pmatrix}
S_{0,0} & S_{0,1} & c_{0,2}S_{0,2}\\
0&S_{1,1}&S_{1,2}&\\
0&0&S_{2,2}\\
\end{pmatrix}.
$$

And
$$S_{i,i+1}X_{i+1,i+1}=X_{i,i}S_{i,i+1}, i=0,1,$$
$$S_{0,0}X_{0,2}+S_{0,2}X_{2,2}=c_{0,2}X_{0,0}S_{0,2}+X_{0,2}S_{2,2},$$
$$S_{0,0}X_{0,2}-X_{0,2}S_{2,2}=c_{0,2}X_{0,0}S_{0,2}-S_{0,2}X_{2,2}.$$

By $$S_{i,i+1}X_{i+1,i+1}=X_{i,i}S_{i,i+1}, i=0,1,$$ and ${\mathcal
A}^{\prime}(S_{i,i})\cong {\mathcal H}^{\infty}(\mathbb{D})$, by
Lemma \ref{commutant2}, we can find a holomorphic function $\phi\in
{\mathcal H}^{\infty}(\mathbb{D})$ such that $X_{i,i}t_i=\phi t_i.$
Since $X_{i,i}$ is invertible, $\phi(S_{i,i})$ is also invertible.
Note that
\begin{equation}
\begin{array}{lllll}
        (c_{0,2}X_{0,0}S_{0,2}-S_{0,2}X_{2,2})(t_2)&=&c_{0,2}X_{0,0}(-t^{(1)}_0)-S_{0,2}(\phi t_2)\\
        &=&\phi t_0^{(1)}-c_{0,2}(\phi t_0)^{(1)}\\
        &=&\phi t^{(1)}_0-c_{0,2}\phi t^{(1)}_0-c_{0,2}\phi^{(1)}t_0\\
        &=&(1-c_{0,2})\phi t^{(1)}_0-c_{0,2}\phi^{(1)}t_0\\
        &=&(c_{0,2}-1)S_{0,2}\phi(S_{2,2})(t_2)-c_{0,2}S_{0,1}S_{1,2}\phi^{(1)}(S_{2,2})(t_2).\\
        \end{array}\label{c02}
\end{equation}
By Lemma 3.3, we have  $c_{0,2}S_{0,1}S_{1,2}\phi^{(1)}(S_{2,2})\in
\mbox{Ran}\sigma_{S_{0,0},S_{2,2}}$.  From \eqref{c02}, it follows
that
$$(c_{0,2}-1)S_{0,2}\phi(S_{2,2})\in
\mbox{\rm ran}\,\sigma_{S_{0,0},S_{2,2}}.$$ By Lemma 3.2, $S_{0,2}\not\in
\mbox{\rm ran}\,\sigma_{S_{0,0},S_{2,2}}$. Since $\phi(S_{2,2})$ is
invertible and $\phi(S_{2,2})\in {\mathcal A}^{\prime}(S_{2,2})$, we
have  $$S_{0,2}\phi(S_{2,2})\not\in
\mbox{\rm ran}\,\sigma_{S_{0,0},S_{2,2}}$$ it follows from Lemma 3.4. This
shows that $c_{0,2}=1.$ For the general case, by the above argument
and Lemma 2.2, we have
$$\widetilde{T}=\begin{pmatrix}
S_{0,0} &S_{0,1} & c_{0,2}S_{0,2}&0&\cdots&0\\
&S_{1,1}&S_{1,2}&c_{1,3}S_{1,3}&\cdots&0\\
&&\ddots&\ddots&\ddots&\vdots\\
&&&S_{n-2,n-2}&S_{n-2,n-1}&c_{n-2,n}S_{n-2,n}\\
&&\bigzero&&S_{n-1,n-1}&S_{n-1,n}&\\
&&&&&S_{n,n}\\
\end{pmatrix}.$$

Now suppose that we have proved Claim 1 for $n=k-1$. Pick
$X=\left (\begin{smallmatrix}
X_{0,0}& 0&\cdots&X_{0,k}\\
0& X_{1,1}& \cdots&X_{1,k}\\
\cdots& \cdots& \cdots&\cdots\\
0& 0& \cdots&X_{k,k}\\
\end{smallmatrix} \right )$ such that
$X\widetilde{T}=TX.$
Then it follows that
$$X_0((\widetilde{S}_{i,j})^{k-1}_{i,j=0})=((S_{i,j})^{k-1}_{i,j=0})X_0,X_1((\widetilde{S}_{i,j})^{k}_{i,j=1})=((S_{i,j})^{k}_{i,j=1})X_1,
$$
where $$\renewcommand\arraystretch{0.875} X_0=\left (\begin{matrix}
X_{0,0}& 0&\cdots&X_{0,k-1}\\
0& X_{1,1}& \cdots&X_{1,k-1}\\
\vdots& \ddots& \ddots&\vdots\\
0&\cdots & 0 &X_{k-1,k-1}\\
\end{matrix}\right ),\;\;X_1=\left (\begin{matrix}
X_{1,1}& 0&\cdots&\,\,X_{1,k}\,\,\\
0& X_{2,2}& \cdots&X_{2,k}\\
\vdots& \ddots& \ddots&\vdots\\
0& \cdots & 0&X_{k,k}\\
\end{matrix}\right ).\renewcommand\arraystretch{1}$$
Since $X$ is invertible, $X_0$ and $X_1$ are  both invertible.
By the induction hypothesis  $c_{i,i+2}=1, i=0,1,\cdots, n-3.$

Case (3) and Case (4):  By Lemma \ref{bdd}, $\widetilde{S}=\big (\!\!\big (\widetilde{S}_{i,j}\big )\!\!\big ),
\widetilde{S}_{i,j}=0, j-i\geq 4$ and
 $\widetilde{S}=\big (\!\!\big (\widetilde{S}_{i,j}\big )\!\!\big ), \widetilde{S}_{i,j}=0,j-i\geq 5$.
Following the proof given above, by Lemma \ref{bdd}, we only need to consider the case of $n=4$ and $n=5$. For case 3, we only consider  $n=4$ and the other cases
would follow by induction. In this case, we have
$$\left ( \begin{smallmatrix}
S_{0,0} & S_{0,1} & S_{0,2}&S_{0,3}\\
0&S_{1,1}&S_{1,2}&S_{1,3}\\
0&0&S_{2,2}&S_{2,3}\\
0&0&0&S_{3,3}
\end{smallmatrix} \right )\left ( \begin{smallmatrix}
X_{0,0} & 0 & X_{0,2}&X_{0,3}\\
0&X_{1,1}&0&X_{1,3}\\
0&0&X_{2,2}&0\\
0&0&0&X_{3,3}
\end{smallmatrix} \right )=
\left ( \begin{smallmatrix}
X_{0,0} & 0 & X_{0,2}&X_{0,3}\\
0&X_{1,1}&0&X_{1,3}\\
0&0&X_{2,2}&0\\
0&0&0&X_{3,3}
\end{smallmatrix} \right )
\left ( \begin{smallmatrix}
S_{0,0} & S_{0,1} & S_{0,2}&c_{0,3}S_{0,3}\\
0&S_{1,1}&S_{1,2}&S_{1,3}\\
0&0&S_{2,2}&S_{2,3}\\
0&0&0&S_{3,3}
\end{smallmatrix} \right ).
$$
It follows that $\left (\begin{smallmatrix} X_{0,0} & 0 & X_{0,2}\\
0&X_{1,1}&0&\\
 0&0&X_{2,2}
\end{smallmatrix}\right )$ commutes with $\left ( \begin{smallmatrix}
S_{0,0} & S_{0,1} & S_{0,2}\\
0&S_{1,1}&S_{1,2}\\
0&0&S_{2,2}\\
\end{smallmatrix} \right )$ and $\left ( \begin{smallmatrix} X_{1,1} & 0 & X_{1,3}\\
0&X_{2,2}&0&\\
 0&0&X_{3,3}
\end{smallmatrix} \right )$ commutes with $\left ( \begin{smallmatrix}
S_{1,1} & S_{1,2} & S_{1,3}\\
0&S_{2,2}&S_{2,3}\\
0&0&S_{3,3}\\
\end{smallmatrix} \right ).$ By \eqref{keyformulae} in the proof of Lemma \ref{ProfJiang}, we have $X_{0,2}$ and $X_{1,3}$
can be chosen as $S_{0,2}\phi^{(1)}(S_{2,2})$ and
$S_{1,3}\phi^{(1)}(S_{3,3})$.
Note that
$$S_{0,0}X_{0,3}+S_{0,1}X_{1,3}+S_{0,3}X_{3,3}=c_{0,3}X_{0,0}S_{0,3}+X_{0,2}S_{2,3}+X_{0,3}S_{3,3}.$$
Then
\begin{equation}\label{X03}S_{0,0}X_{0,3}-X_{0,3}S_{3,3}=(c_{0,3}X_{0,0}S_{0,3}-S_{0,3}X_{3,3})+X_{0,2}S_{2,3}-S_{0,1}X_{1,3}.\end{equation}
And $$\begin{array}{llll}
X_{0,2}S_{2,3}-S_{0,1}X_{1,3}&=&S_{0,2}\phi^{(1)}(S_{2,2})S_{2,3}-S_{0,1}S_{1,3}\phi^{(1)}(S_{3,3})\\
&=&S_{0,2}S_{2,3}\phi^{(1)}(S_{3,3})-S_{0,1}S_{1,3}\phi^{(1)}(S_{3,3})\\
&=&(S_{0,2}S_{2,3}-S_{0,1}S_{1,3})\phi^{(1)}(S_{3,3})\\
&=&0.
\end{array}
$$
So we only need to consider
$$S_{0,0}X_{0,3}-X_{0,3}S_{3,3}=c_{0,3}X_{0,0}S_{0,3}-S_{0,3}X_{3,3}.$$
Since $$\begin{array}{llll}
(c_{0,3}X_{0,0}S_{0,3}-S_{0,3}X_{3,3})(t_3)&=&c_{0,3}X_0S_{0,3}(t_3)-S_{0,3}(\phi
t_3)\\
&=&c_{0,3}X_0(-t^{(2)}_0)-\phi S_{0,3}(t_3)\\
&=&-c_{0,3}(\phi t_0)^{(2)}+\phi t^{(2)}_0\\
&=&(1-c_{0,3})\phi
t^{(2)}_0-2c_{0,3}\phi^{(1)}t^{(1)}_0-c_{0,3}\phi^{(2)}t_0,\\
\end{array}$$
we obtain
$$c_{0,3}X_{0,0}S_{0,3}-S_{0,3}X_{3,3}=(c_{0,3}-1)S_{0,3}\phi(S_{3,3})
+2c_{0,3}S_{0,1}S_{1,3}\phi^{(1)}(S_{3,3})+c_{0,3}S_{0,1}S_{1,2}S_{2,3}\phi^{(2)}(S_{3,3}).$$
By Lemma \ref{keyLb} and \eqref{X03}, we have
$$2c_{0,3}S_{0,1}S_{1,3}\phi^{(1)}(S_{3,3})+c_{0,3}S_{0,1}S_{1,2}S_{2,3}\phi^{(2)}(S_{3,3})\in
\mbox{Ran}\sigma_{S_{0,0},S_{3,3}}.$$  Since $\phi(S_{3,3})$ is
invertible,  we  deduce that
$$(c_{0,3}-1)S_{0,3}\in \mbox{\rm ran}\,\sigma_{S_{0,0},S_{3,3}}.$$  Note
that  $S_{0,3}\not\in \mbox{\rm ran}\,\sigma_{S_{0,0},S_{3,3}}$, we have
$c_{0,3}=1$.
For case 4, when $n=5$, the commutator
$$\!\!\!\!\!\left(\begin{smallmatrix}
 S_{0,0} & S_{0,1} & S_{0,2}&S_{0,3}&S_{0,4}\\
&S_{1,1}&S_{1,2}&S_{1,3}&S_{1,4}\\
&&S_{2,2}&S_{2,3}&S_{2,4}\\
&\bigzero&&S_{3,3}&S_{3,4}\\
&&&&S_{4,4}
\end{smallmatrix}\right )\left ( \begin{smallmatrix}
X_{0,0} & 0 & X_{0,2}&X_{0,3}&X_{0,4}\\
&X_{1,1}&0&X_{1,3}&X_{1,4}\\
&&X_{2,2}&0&X_{2,4}\\
&\bigzero&&X_{3,3}&0\\
&&&&X_{4,4}
\end{smallmatrix}\right ) \!= \!\left ( \begin{smallmatrix}
X_{0,0} & 0 & X_{0,2}&X_{0,3}&X_{0,4}\\
&X_{1,1}&0&X_{1,3}&X_{1,4}\\
&&X_{2,2}&0&X_{2,4}\\
&\bigzero&&X_{3,3}&0\\
&&&&X_{4,4}
\end{smallmatrix}\right ) \left(\begin{smallmatrix}
 S_{0,0} & S_{0,1} & S_{0,2}&S_{0,3}& c_{0,4}S_{0,4}\\
&S_{1,1}&S_{1,2}&S_{1,3}&S_{1,4}\\
&&S_{2,2}&S_{2,3}&S_{2,4}\\
&\bigzero&&S_{3,3}&S_{3,4}\\
&&&&S_{4,4}
\end{smallmatrix}\right )
$$
Therefore $\big (\!\! \big (X_{ij}\big )\!\! \big )_{4\times 4}, i,j=0,1,2,3$ commutes with
$\big (\!\! \big (S_{i,j}\big )\!\! \big )_{4\times 4},i,j=0,1,2,3$  and $\big (\!\! \big (X_{ij}\big )\!\! \big )_{4\times 4},
i,j=1,2,3,4$ commutes with $\big (\!\! \big (S_{i,j}\big )\!\! \big )_{4\times 4},i,j=1,2,3,4.$
Then we can find
$X_{i,j}, (i,j)\neq (0,4)$ from Lemma 3.5.
We also have
\begin{equation}S_{0,0}X_{0,4}-X_{0,4}S_{4,4}=(c_{0,4}X_{0,0}S_{0,4}-S_{0,4}X_{4,4})+(X_{0,2}S_{2,4}+X_{0,3}S_{3,4})-(S_{0,1}X_{1,4}+S_{0,2}X_{2,4})\label{(3.1.8)}.\end{equation}
By Lemma 3.5, we have $$\begin{array}{llll}
X_{0,2}S_{2,4}-S_{0,2}X_{2,4}&=&S_{0,2}\phi^{(1)}(S_{2,2})S_{2,4}-S_{0,2}S_{2,4}\phi^{(1)}(S_{4,4})\\
&=&S_{0,2}(S_{2,3}S_{3,4}\phi^{(2)}(S_{4,4})+S_{2,4}\phi^{(1)}(S_{4,4}))-S_{0,2}S_{2,4}\phi^{(1)}(S_{4,4})\\
&=&S_{0,2}S_{2,3}S_{3,4}\phi^{(2)}(S_{4,4}).
\end{array}
$$
By \eqref{keyformulae} in the proof of Lemma \ref{ProfJiang},  we
have
$$X_{0,3}=S_{0,2}S_{2,3}\phi^{(2)}(S_{3,3})+S_{0,3}\phi^{(1)}(S_{3,3}),$$
$$X_{1,4}=S_{1,3}S_{3,4}\phi^{(2)}(S_{4,4})+S_{1,4}\phi^{(1)}(S_{4,4}).$$

Note that $S_{0,2}S_{2,3}=S_{0,1}S_{1,3}$ and
$S_{0,3}S_{3,4}=S_{0,1}S_{1,4}$,  we also have
\begin{eqnarray*}
\lefteqn{X_{0,3}S_{3,4}-S_{0,1}X_{1,4}}\\
&=&(S_{0,2}S_{2,3}\phi^{(2)}(S_{3,3})+S_{0,3}\phi^{(1)}(S_{3,3}))S_{3,4}-S_{0,1}(S_{1,3}S_{3,4}\phi^{(2)}(S_{4,4})+S_{1,4}\phi^{(1)}(S_{4,4})\\
&=&S_{0,2}S_{2,3}S_{3,4}\phi^{(2)}(S_{4,4})+S_{0,3}S_{3,4}\phi^{(1)}(S_{4,4})-S_{0,1}S_{13}S_{3,4}\phi^{(2)}(S_{4,4})-S_{0,1}S_{1,4}\phi^{(1)}(S_{4,4})\\
&=&0.
\end{eqnarray*}

Since $$\begin{array}{llll}
(c_{0,4}X_{0,0}S_{0,4}-S_{0,4}X_{4,4})(t_4)&=&c_{0,4}X_{0,0}S_{0,4}(t_4)-S_{0,4}(\phi
t_4)\\
&=&c_{0,4}X_{0,0}(-t^{(3)}_0)-\phi S_{0,4}(t_4)\\
&=&-c_{0,4}(\phi t_0)^{(3)}+\phi t^{(3)}_0\\
&=&(1-c_{0,4})\phi
t^{(3)}_0-3c_{0,4}\phi^{(2)}t^{(1)}_0-3c_{0,4}\phi^{(1)}t^{(2)}_0-c_{0,4}\phi^{(3)}t_0,\\
\end{array}$$
and we also have
\begin{eqnarray*}
c_{0,4}X_{0,0}S_{0,4}-S_{0,4}X_{4,4}
&=&(c_{0,4}-1)S_{0,4}\phi(S_{4,4})
+3c_{0,4}S_{0,1}S_{1,3}\phi^{(1)}(S_{3,3})\\
&{}&\;\;\;\;\;+3c_{0,4}S_{0,1}S_{1,2}S_{2,3}\phi^{(2)}(S_{3,3})+c_{0,4}S_{0,1}S_{1,2}S_{2,3}\phi^{(3)}(S_{3,3}).
\end{eqnarray*}
By Lemma 3.3 and \eqref{(3.1.8)}, we obtain
$$3c_{0,4}S_{0,1}S_{1,3}\phi^{(1)}(S_{3,3})+3c_{0,4}S_{0,1}S_{1,2}S_{2,3}\phi^{(2)}(S_{3,3})+c_{0,4}S_{0,1}S_{1,2}S_{2,3}\phi^{(3)}(S_{3,3})\in\mbox{\rm ran}\,\sigma_{S_{0,0},S_{4,4}},$$
$$ S_{0,2}S_{2,3}S_{3,4}\phi^{(2)}(S_{4,4})\in
\mbox{Ran}\sigma_{S_{0,0},S_{4,4}}.$$ Then it follows that
$$(c_{0,4}-1)S_{0,4}\phi(S_{4,4})\in \mbox{\rm ran}\,\sigma_{S_{0,0},S_{4,4}}.$$ Note that
$\phi(S_{4,4})$ is invertible, therefore 
$$(c_{0,4}-1)S_{0,4}\in \mbox{\rm ran}\,\sigma_{S_{0,0},S_{4,4}}.$$
Since
$S_{0,4}\not\in \mbox{\rm ran}\,\sigma_{S_{0,0},S_{4,4}}$, it follows that 
$c_{0,4}=1$.

The proof in all the remaining cases are similar and therefore
the Claim is verified.
\end{proof}
%
%
%
%
%

\section{Applications}
We give three different applications of our results.
First of these shows that the topological and algebraic $K$-groups defined in our context must coincide.  Secondly, we show that our techniques apply to a slightly larger class of operators than the quasi-homogeneous ones that we have discussed in this paper.  Finally, we show that the Halmos' question on similarity has an affirmative answer for quasi-homogeneous operators. We begin with some preliminaries on $K$-  groups.

\subsection{\sf Preliminaries}
Let $t :\Omega\rightarrow Gr(n,{\mathcal H})$ be a holomorphic curve. Recall that the commutant $\mathcal {A}'(E_t)$ of such a holomorphic curve $t$ is defined to be $$\mathcal {A}'(E_t)=\{A\in\mathcal {L}(\mathcal {H}):A\,t({w})\subseteq t({w}),\; w \in \Omega.\}$$

%
%
%
\begin{defn} For a holomorphic curve $t:\Omega
 \rightarrow Gr(n,{\mathcal H})$, the Jocaboson radical $\mbox{Rad}\;{\mathcal A}^{\prime}(E_t)$
of ${\mathcal A}^{\prime}(E_t)$ is defined to be
$$\{S\in {\mathcal A}^{\prime}(E_t)|\sigma_{{\mathcal
A}^{\prime}(E_t)}(S A)=0, A \in {\mathcal A}^{\prime}(E_t)\},$$
where $\sigma_{{\mathcal A}^{\prime}(E_t)}(S A)$ denotes the spectrum of $S A$ in the algebra ${\mathcal A}^{\prime}(E_t).$
\end{defn}
The discussion below follows closely the paper \cite{Jia} of the
first two authors. In particular, Lemma \ref{min} and Lemma
\ref{fin} are proved there.

\begin{lem} \mbox{\rm (\cite[Theorem 1.2]{Jia})}\label{min} Let $t:\Omega\rightarrow Gr(n,{\mathcal H})$
be a holomorphic curve, and $P\in {\mathcal A}^{\prime}(E_t)$ be an
idempotent, then $P t:\Omega\rightarrow Gr(m,P{\mathcal H})$ is again a holomorphic curve, where $m=\dim\,\mbox{\rm ran}\,P(t(w))$ for $ w\in \Omega.$ The idempotent $P$ is  minimal  if and only if $P\, t$ is strongly irreducible.
\end{lem}
\begin{lem}\mbox{\rm (\cite[Theorem 1.3]{Jia})}\label{fin}
For a holomorphic curve $t:\Omega\rightarrow Gr(n,{\mathcal H}),$  the following statements are equivalent.
\begin{enumerate}
\item There exists $m$ minimal idempotents $P_1,P_2,\cdots ,P_m \in
{\mathcal A}^{\prime}(E_t)$  such that $P_iP_j=0$  and
$\sum\limits^{m}_{i=1}P_i=I_{{\mathcal H}}.$
\item There exists an invertible operator $X\in {\mathcal A}^{\prime}(E_t)$ such
that $Xt$ can be written as orthogonal direct sum of $m$
strongly irreducible holomorphic curves.
\end{enumerate}
\end{lem}
\begin{defn} \label{findecomp}A holomorphic curve $t:\Omega \rightarrow Gr(n,{\mathcal H})$ is said to be have  a finite decomposition if it meets  one of the equivalent conditions given in Lemma \ref{fin}.

Suppose $\{P_1, P_2, \cdots, P_m\}$ and $\{Q_1, Q_2,
 \cdots, Q_n\}$ are two distinct decompositions of $t.$
If $m=n,$ there exists a permutation $\Pi \in S_n$ such that $XQ_{\Pi(i)}X^{-1}=P_i$ for some invertible operator $X$ in $\mathcal A^\prime(E_t),$ $ 1\leq i \leq n,$ then we say that
$t$ (or $E_t$) has a unique decomposition up to similarity.
\end{defn}

For a holomorphic curve, $f:\Omega\rightarrow Gr(n,{\mathcal H})$, let $M_{k}({\mathcal A}^{\prime}(E_t))$ be the collection of $k\times k$ matrices with entries from ${\mathcal A}^{\prime}(E_t)$.
Let
 $$M_{\infty}({\mathcal A}^{\prime}(E_t))=\bigcup\limits^{\infty}_{k=1}M_{k}({\mathcal
 A}^{\prime}(E_t)),$$
and $\mbox{\rm Proj}(M_k({\mathcal A}^{\prime}(E_t)))$ be the
algebraic equivalence classes of idempotents in
$M_{\infty}({\mathcal A}^{\prime}(E_t))$.
  If $p, q$ are idempotents in $\mbox{\rm Proj}({\mathcal A}^{\prime}(E_t))$, then say that $p{\sim}_{st}q$ if  $p{\oplus}r{\sim}_aq{\oplus}r$ for some idempotent $r$ in $\mbox{\rm Proj}\,({\mathcal A}^{\prime}(E_t))$. The relation ${\sim}_{st}$ is known as stable equivalence.


Let $X$ be a compact Hausdorff space, and $\xi=(E,\pi, X)$ be a (topological) vector bundle. A well-known theorem due to  R. G. Swan (cf. \cite{SP}) says that a vector bundle $\xi=(E,\pi, X)$ is a direct summand of the trivial bundle, that is, $$\xi\oplus \eta \cong (X\times \mathbb{C}^n, \pi, X)$$
for some vector bundle $\eta=(F,\rho,X).$

Swan's Theorem relates the geometric notion of a vector bundle to the algebraic notion of a $K_0$ group which we now describe briefly.

Following the usual conventions, let $\mbox{\rm Vect}(X)$ be the set of all isomorphism classes $\overline{\xi}$ of vector bundles $\xi$ over $X$. Addition and multiplication are defined on $\mbox{\rm Vect}(X)$ by the rule
$$ \overline{\xi}+\overline{\eta}=\overline{\xi\oplus \eta}, \overline{\xi}  \overline{\eta}= \overline{\xi\otimes \eta}.$$
These operations are well defined. Thus  $(\mbox{\rm Vect}(X),
+)$ is an Abelian semi-group and $K^0(X)$ is defined as the Grothendieck group of $(\mbox{\rm Vect}(X), +)$ (see \cite{RLL} for more details).
 Swan's theorem is the main ingredient in showing that the topological $K$-group $K^0(X)$ is isomorphic to the algebraic $K_0$-group $K_0(C(X))$.

For any projection $p\in P(M_{\infty}(C(X))$, suppose that $p\in
M_{n}(C(X))$. From this $p,$ one may construct a vector bundle on $X:$
$$E(p):=\{(x.v)\in X\times \mathbb{C}^n:v\in p(x)(\mathbb{C}^n)\},$$
with the fiber $E_x(p)=p(x)(\mathbb{C}^n)$. Define an additive map
$\Gamma:\mbox{\rm Proj}({\mathcal A}^{\prime}(E_t))\rightarrow \mbox{Vect(X)}$ as follows:
$$\Gamma([p]_0)=\overline{\xi_p},\; [p]_0\in \mbox{\rm Proj}({\mathcal A}^{\prime}(E_t)).$$
Then $\Gamma$ is an isomorphism.

First, we show $\Gamma$ is injective.
If $\Gamma([P]_0)=\overline{\xi_p}=\overline{\xi_q}=\Gamma([q]_0),$
and $p\in P(M_{n}(C(X)),q\in P(M_{m}(C(X))$ then there exists an
isomorphism
$$\sigma: \overline{\xi_p}\rightarrow \overline{\xi_q}, $$ where
$\sigma(p(x)\mathbb{C}^n)\cong (q(x)\mathbb{C}^m).$ So we have
$Tr(p(x))=Tr(q(x)),$ where $Tr$ is the trace of $M_{\infty}(C(X))$.
Then we can find $v_x\in M_{m,n}(C(X))$ such that
$$v_x^*v_x=p(x), v_xv_x^*=q(x).$$ That means $[p]_0=[q]_0$. So
$\Gamma$ is injective.

Next, we show that $\Gamma$ is surjective. By Swan's theorem, for
any vector bundle $\xi=(E,\pi,X)$,  there exists a positive integer
$n$ and another vector bundle $\eta(F, \rho, X)$ such that
$$\xi\oplus \eta=(X\times \mathbb{C}^n, \pi, X).$$  The we can assume that $E_x\oplus F_x =\mathbb{C}^n$.
Set $p(x)$ be the projection from $\mathbb{C}^n$ onto $E_x$. Then
$p:X\rightarrow M_n(\mathbb{C})$ is continuous and $p\in
P(M_{n}(C(X))$, $\xi:=\xi_p$. Then we can see that $\Gamma$ is also
a surjective. So $\Gamma$ is an isomorphism.

Thus if $\xi\oplus \eta$ is a trivial bundle, then there exists $p\in M_{\infty}(C(X))$ such that $\xi=\xi_p, \eta=\xi_{I-p}.$
Now, if there exists another vector bundle $\eta^{\prime}$ such that $\xi\oplus \eta^{\prime}$ is also isomorphic to a trivial bundle, then there must exist a projection $p^{\prime}\in
 M_{\infty}(C(X))$ such that $\xi=\xi_{p^{\prime}}, \eta=\xi_{I-p^{\prime}}.$ Consequently,  $[p]_0=[p^{\prime}]_0,$ $[1-p]_0=[1-p^{\prime}]_0$ and we see that
 $\eta^{\prime}\cong \eta.$ So there is a unique vector bundle
 $\eta$ such that $$\xi\oplus \eta \cong (X\times \mathbb{C}^n, \pi, X).$$


\subsection{\sf Unique decomposition} None of what we have said so far applies to holomorphic vector bundles over an open subset of $\mathbb C$ since they are already trivial by Graut's theorem. However, the study of holomorphic vector bundles over an open subset of $\mathbb C$ is central to operator theory. In the context of operator theory, as shown in the foundational paper of Cowen and Douglas \cite{CD}, the vector bundles of interest are equipped with a Hermitian structure
inherited from a fixed inner product of some Hilbert space $\mathcal H.$ This makes it possible to ask questions
 about their equivalence under a unitary or an invertible linear transformation of $\mathcal H$.
 In the paper \cite{CD}, questions regarding unitary equivalence were dealt with quite successfully while equivalence under an invertible linear transformation remains somewhat of a mystery to date. However, we can ask if the uniqueness of the summand, which was a consequence of Swan's theorem, remains valid in the context of Cowen-Douglas operators.
\begin{qn}\label{mainQ}
Let $t:\Omega \rightarrow Gr(n,{\mathcal H})$  be a Hermitian   holomorphic  curve and the vector bundle $E_r$ be a direct
 summand of $E_t$ for some other holomorphic curve $r:\Omega \rightarrow Gr(n,{\mathcal H})$.
Does there a unique sub-bundle of $E_t,$ up
to similarity, such that $E_r\oplus E_s=E_t$? Here the uniqueness is meant to be in the sense of Definition \ref{findecomp}
\end{qn}
It was shown in \cite{JGJ} that an operator in the Cowen-Douglas class $B_n(\Omega)$ admits a unique decomposition. So, the answer to the question raised above is affirmative. However, here we give a different proof for quasi-homogeneous operators which is much more transparent.
%
%
For our proof, we will need the following lemma.
\begin{lem} \label{qcomm}
Let $E_t$ be a quasi-homogeneous  bundle. Then ${\mathcal
A}^{\prime}(E_t)/\mbox{\rm Rad}({\mathcal A}^{\prime}(E_t))$ is commutative.
\end{lem}

\begin{proof}
Let $$\mathcal S=\{Y: \sigma(Y)=0, Y\in {\mathcal A}^{\prime}(E_t)\}.$$

{\sf Claim 1:}\,\, $\mathcal S$ is an ideal of the algebra ${\mathcal A}^{\prime}(E_t)$.

By Lemma \ref{utXY}, $Y$ is upper-triangular if $Y\in \mathcal S$. Since the spectrum $\sigma(Y)$ of $Y$ is $\{0\}$, the operator $Y$ must be of the form
\renewcommand \arraystretch{0.875}
$$Y=\left ( \begin{matrix}
0 & Y_{0,1} & Y_{0,2}&\cdots&Y_{0,n-1}\\
0&0&Y_{1,2}&\cdots&Y_{1,n-1} \\
\vdots&\ddots&\ddots&\ddots&\vdots\\
0&\cdots&0&0&Y_{n-2,n-1}\\
0&\cdots&\cdots&0&0\\
\end{matrix}\right ),$$
\renewcommand \arraystretch{1}
and it follows that each quasi-nilpotent element in the commutant of the holomorphic curve $t$ of rank one is zero.  Using Lemma \ref{utXY}
again, each element $X\in {\mathcal A}^{\prime}(E_t)$ is
upper-triangular. Thus, $\sigma(XY)=\sigma(YX)=0.$ This completes  the
proof of Claim 1 and $\mathcal S=\mbox{\rm Rad}({\mathcal A}^{\prime}(E_t)).$

{\sf Claim 2:}\,\,${\mathcal A}^{\prime}(E_t)/\mbox{\rm Rad}({\mathcal A}^{\prime}(E_t))$ is commutative.

Note that if $X\in {\mathcal A}^{\prime}(E_t)$ is (block) nilpotent, then
$X\in \mathcal S.$
A simple computation shows that ${\mathcal
A}^{\prime}(E_t)/\mbox{\rm Rad}({\mathcal A}^{\prime}(E_t))$ is commutative.
\end{proof}

%
%
%

\begin{thm}\label{uniqueSim}
For any quasi-homogeneous holomorphic curve $t$ with atoms
$t_i,\;0\leq i \leq {n-1},$ we have that
\begin{enumerate}
\item $E_t$ has no non-trivial sub-bundle whenever $\Lambda(t) <2,$ and
\item if $\Lambda(t) \geq 2,$ then for any sub-bundle $E_r$ of
$E_t$, there exists a unique sub-bundle $E_s,$ up to equivalence
under an invertible map, such that $E_r\oplus E_s$ is similar to
$E_t.$
\end{enumerate}
\end{thm}


For any holomorphic curve $t,$ we let $t^n$ denote the $n$ - fold direct sum of $t.$
For any two natural numbers $n$ and $m,$  let $E_r$ and $E_s$ be the sub-bundles of
$E_{t^{n}}$ and $E_{t^{m}},$ respectively. If $m>n,$ then both $E_r$ and $E_s$ can be regarded as a sub-bundle of $E_{t^m}$.

Two holomorphic Hermitian vector bundles $E_r$ and $E_s$ are said to be similar if there exist an invertible operator $X\in {\mathcal
A}^{\prime}(E_{r})$ such that $XE_r=E_s.$ Analogous to the  definition of $\mbox{Vect(X)},$  we let
$\mbox{\rm Vect}^0(E_t)$ be the  set of equivalence classes $\overline{E_s}$ of the sub-bundles $E_s$ of $E_{t^n},$
$n=1,2,\cdots.$ An addition on $\mbox{\rm Vect}^0(E_t)$ is defined as follows, namely,
$$ \overline{E_r}+\overline{E_s}=\overline{E_r\oplus E_s},$$ where $E_r$ and $E_s$ are both
sub-bundles of $E_t$. Now, the group $K^0(E_t)$ is  the Grothendieck group of $(\mbox{\rm Vect}^0(E_t), +)$. In this notation, we have the following theorem.
\begin{thm} \label{K^0=K_0}
$K^0(E_t)\cong K_0({\mathcal A}^{\prime}(E_t)).$
\end{thm}
The proof of this theorem is split into a number of lemmas which are stated and proved below.

\begin{lem}  Let $E_t$ be a  quasi-homogeneous bundle. Then  $${\mbox{\rm Vect}}({\mathcal A}^{\prime}(E_t)) \cong {\mbox{\rm Vect}} ({\mathcal A}^{\prime}(E_t)/\mbox{Rad}{\mathcal
A}^{\prime}(E_t)).$$
\end{lem}

\begin{proof}
Note that $ M_n({\mathcal A}^{\prime}(E_t))\cong {\mathcal
A}^{\prime}(\bigoplus\limits^{n}E_t)$.  Let $p\in M_n({\mathcal
A}^{\prime}(E_t))$ be an idempotent. Define a map $\sigma:
{\mbox{\rm Vect}} ({\mathcal A}^{\prime}(E_t)) \rightarrow
{\mbox{\rm Vect}} ({\mathcal A}^{\prime}(E_t)/\mbox{Rad}{\mathcal
A}^{\prime}(E_t))$ as the following:
$$\sigma[P]=[\pi(P)],$$
where $\pi: {\mathcal A}^{\prime}(E_t) \rightarrow {\mbox{\rm Vect}}
({\mathcal A}^{\prime}(E_t)/\mbox{Rad}{\mathcal A}^{\prime}(E_t))$.

{\bf Claim}\,\, $\sigma$ is well defined and it is an isomorphism.

If $[p]=[q]$, where $p\in M_n({\mathcal A}^{\prime}(E_t))$ and $q\in
M_m({\mathcal A}^{\prime}(E_t))$ are both idempotents, then there
exists $k\geq \max\{m,n\}$ and an invertible element $u\in
M_k({\mathcal A}^{\prime}(E_t))$ such that
$$u(p\oplus 0^{k-n})u^{-1}=q\oplus 0^{k-m}.$$

Thus we have   $$\pi(u) \pi(p\oplus
0^{(k-n)})\pi(u)^{-1}=\pi(u(p\oplus 0^{k-n})u^{-1})=\pi(q\oplus
  0^{k-m}).$$

  That means $[\pi(p)]=[\pi(q)]$, and $\sigma$ is well defined.

Then we would prove that $\sigma$ is injective. In fact, if  $p\in
M_n({\mathcal A}^{\prime}(E_t))$ and $q\in M_m({\mathcal
A}^{\prime}(E_t))$ are idempotents with
$$\sigma[p]=[\pi(p)]=[\pi(q)]=\sigma [q],$$
then we can find $k\geq \max\{m,n\}$ and an invertible element
 $\pi(u)\in M_k({\mathcal A}^{\prime}(E_t))/\mbox{Rad}(M_k({\mathcal A}^{\prime}(E_t)))$ such that
$$\pi(u)(\pi(p\oplus 0^{k-n}))\pi(u)^{-1}=\pi(q\oplus 0^{k-m}).$$

Since $\pi(u)$ is invertible, there exists $\pi(s)\in
\mbox{Rad}(M_k({\mathcal A}^{\prime}(E_t)))$ such that
$\pi(u)^{-1}=\pi(s)$. Then we have $$us=I-R_1, su=I-R_2,$$ where
$R_1,R_2\in \mbox{Rad}(M_k({\mathcal A}^{\prime}(E_t)))$. Since
$\sigma(R_1)=\sigma(R_2)=\{0\}$, then $us, su$ are both invertible.
Therefore, $u$ is invertible and thus
$$\pi(u (p\oplus 0^{(k-n)})u^{-1})=\pi(u)(\pi(p\oplus 0^{k-n}))\pi(u)^{-1}=\pi(q\oplus 0^{k-m}).$$
Thus, $$u (p\oplus 0^{(k-n)})u^{-1}=q\oplus 0^{k-m}+R$$ for some
$R\in \mbox{Rad}(M_k({\mathcal A}^{\prime}(E_t)))$. Let
$W_1=2(q\oplus 0^{(k-m)})-I.$ Since $\sigma(Q\oplus
0^{(k-m)})\subseteq \{0,1\}$, then $W_1$ is invertible. Since we
have $R\in \mbox{Rad}(M_k({\mathcal A}^{\prime}(E_t)))$ and
$W^{-1}_1\in M_k({\mathcal A}^{\prime}(E_t)),$ then $RW^{-1}_1\in
\mbox{Rad}(M_k({\mathcal A}^{\prime}(E_t)))$, so $I+RW_1^{-1}$ is
invertible.  Set
$$W=2(q\oplus 0^{(k-m)})-I+R=W_1+R=(I+RW_1^{-1})W_1.$$
and $W$ is invertible. Since $p\oplus 0^{(k-n)}$ is an idempotent,
then $u(p\oplus 0^{(k-n)})u^{-1}$ is an idempotent, then $(q\oplus
0^{(k-m)})+R$ is an idempotent. Thus,
$$(q\oplus 0^{(k-m)})^2+(q\oplus 0^{(k-m)})R=R(q\oplus
0^{(k-m)})+R^2=(q\oplus 0^{(k-m)})+R.$$

Since $q\oplus 0^{(k-m)}$ is an idempotent, then
$$(q\oplus 0^{(k-m)})R+R(q\oplus 0^{(k-m)})+R^2=R.$$
So we have $$\begin{array}{llll} W((q\oplus 0^{(k-m)}+R)&=&(q\oplus
0^{(k-m)})+R(q\oplus 0^{(k-m)})+2(q\oplus 0^{(k-m)})R-R+R^2\\
&=&(q\oplus 0^{(k-m)})+(q\oplus 0^{(k-m)})R\\
&=&(q\oplus 0^{(k-m)})W.
\end{array}$$

And
$$u(p\oplus 0^{(k-n)})u^{-1}=(q\oplus 0^{(k-m)})+R=W^{_1}(q\oplus
0^{(k-m)})W.$$

It follows that $p\sim_a q$, and $\sigma$ is injective. At last, we
would show that $\sigma$ is surjective. For each $[\pi(p)]\in
{\mbox{\rm Vect}}({\mathcal A}^{\prime}(E_t)/\mbox{Rad}{\mathcal
A}^{\prime}(E_t))$ with $\pi(p) \in M_n({\mathcal
A}^{\prime}(E_t))/\mbox{Rad}(M_n({\mathcal A}^{\prime}(E_t)))$,
$p\in M_n({\mathcal A}^{\prime}(E_t))$ and $\pi^2(p)=\pi(p)$, we
have $$p^2-p=R_0, R_0\in \mbox{Rad}(M_n({\mathcal
A}^{\prime}(E_t))).$$ Note that $p=B+R$, where $B\in M_n({\mathcal
A}^{\prime}(E_t))$ is a block-diagonal matrix over $\mathbb{C}$,
$R\in \mbox{Rad}(M_n({\mathcal A}^{\prime}(E_t)))$. Then
$\pi(p)=\pi(B)$ and
$$R_0=p^2-p=(B+R)^2-(B+R)=B^2-B+(BR+RB+R^2-R). $$

Since $\mbox{Rad}(M_n({\mathcal A}^{\prime}(E_t))) $ is an ideal of
$M_n({\mathcal A}^{\prime}(E_t))$, then we have $$B^2-B\in
\mbox{Rad}(M_n({\mathcal A}^{\prime}(E_t))) .$$ Since $B$ is a
block-diagonal matrix, then we have $B$ is also an idempotent. Then
we have $$\sigma([B])=[\pi(p)].$$ That means $\sigma$ is also a
surjective. And we also can see that $\sigma$ is homomorphism. Then
$\sigma$ is an isomorphism and
 $${\mbox{\rm Vect}} ({\mathcal A}^{\prime}(E_t)) \cong {\mbox{\rm Vect}}({\mathcal A}^{\prime}(E_t)/\mbox{Rad}{\mathcal
A}^{\prime}(E_t)).$$
\end{proof}
We need two more lemmas, which have been already proved in \cite{Jia}, we reproduce them below.
\begin{lem} \label{UD}\mbox{\rm (\cite[Lemma 2.10]{Jia})}
For any holomorphic curve $t:\Omega\rightarrow Gr(n, {\mathcal H}),$ the following statements are equivalent.
\begin{enumerate}
\item  Assume that ${\mathcal H}$ has the decomposition 
${\mathcal H}=\bigoplus\limits^{k}_{i=1} {\mathcal H}_{i}^{(n_i)},$
$1\leq i \leq k.$  The holomorphic curve  $t$ is similar to $
\bigoplus\limits^{k}_{i=1} (P_ie)^{(n_i)},$ $k,n_i < \infty,$ where
$P_i:{\mathcal H}\rightarrow {\mathcal H}_{i}$ are idempotents such  that $P_{i}e$ is indecomposable, $P_{i}e \not\sim P_{j}e$ for
$i\neq j$ and $t^{(\ell)}$ admits a finite unique decomposition,
up to similarity, for $\ell\in \mathbb N.$
\item The algebra ${\mbox{\rm Vect}}({\mathcal
A}^{\prime}(t))$ is isomorphic to $\mathbb N^{(k)}$ via $h,$ which maps  $[I]$ to $(n_1, n_2, \cdots, n_k)$, that is,
$h([I])=n_1e_1+n_2e_2+\cdots +n_ke_k$, where $I$ is the unit of
${\mathcal A}^{\prime}(t),$ $0\not=n_i{\in}\mathbb N,$ and $e_i$ are the generators of  $\mathbb N^{(k)},$ $i=1, 2,
\cdots, k.$
\end{enumerate}
\end{lem}

\begin{lem} \label{KG}\mbox{\rm (\cite[Lemma 2.14]{Jia})}
$$\mbox{\rm Vect}({\mathcal H}^{\infty}(\mathbb{D}))\cong \mathbb{N},
K_0({\mathcal H}^{\infty}(\mathbb{D}))\cong \mathbb{Z}.$$
\end{lem}

\begin{prop}
Let $E_t$ and $E_{\widetilde{t}}$ be two quasi-homogeneous bundles
with matchable bundles $\{E_{t_i}\}^{n-1}_{i=0}$ and
$\{E_{s_i}\}^{n-1}_{i=0}$ respectively. If $\Lambda (t)<2$, then
$E_t$ and $E_{\widetilde{t}}$ are similarity equivalent if and only
if
$$K_0({\mathcal A}^{\prime}(E_t\oplus E_{\widetilde{t}}))\cong \mathbb{Z}.$$
If $\Lambda (t)\geq 2$, then $E_t$ and $E_{\widetilde{t}}$ are
similarity equivalent if and only if
$$K_0({\mathcal A}^{\prime}(E_t\oplus E_{\widetilde{t}}))\cong \mathbb{Z}^n.$$

\end{prop}

\begin{proof} Suppose that  $\Lambda (t)<2$.
Let $$S=\left (\begin{smallmatrix}
S_{0,0} & S_{0,1} & S_{0,2}&\cdots&S_{0,n-1}\\
&S_{1,1}&S_{1,2}&\cdots&S_{1,n-1} \\
&&\ddots&\ddots&\vdots\\
&&&S_{n-1,n-1}&S_{n-1,n}\\
&&&&S_{n,n}\\
\end{smallmatrix}\right )~~\mbox{and}~~X=\left ( \begin{smallmatrix}
X_{0,0} & X_{0,1} & X_{0,2}&\cdots&X_{0,n-1}\\
&X_{1,1}&X_{1,2}&\cdots&X_{1,n-1} \\
&&\ddots&\ddots&\vdots\\
&&&X_{n-1,n-1}&X_{n-1,n}\\
&&&&X_{n,n}\\
\end{smallmatrix}\right ).$$
{\sf Claim 1:}\,\,If $XS=SX$,then we have $X_{i,i}=X_{j,j},
~~\mbox{for any}~~i\neq j.$

In fact, for any $i=0,1,\cdots,n-1$,  we have
$$S_{i,i}X_{i,i+1}+S_{i,i+1}X_{i+1,i+1}=X_{i,i}S_{i,i+1}+X_{i,i+1}S_{i+1,i+1},$$
and
$$S_{i,i}X_{i,i+1}-X_{i,i+1}S_{i+1,i+1}=X_{i,i}S_{i,i+1}-S_{i,i+1}X_{i+1,i+1}=0.$$
Since  $X_{i,i}\in {\mathcal A}^{\prime}(E_{t_i})$ and each
$E_{t_i}$ induces a Hilbert functional space ${\mathcal H}_i$ with
reproducing kernel $\frac{1}{(1-z\overline{w} )^{\lambda_i}}$, then
we have ${\mathcal A}^{\prime}(E_{t_i})\cong {\mathcal
H}^{\infty}(\mathbb{D})$. Then there exists $\phi_{i,i}\in
{\mathcal H}^{\infty}(\mathbb{D})$ such that
$$X_{i,i}=\phi_{i,i}(S_{i,i}), i=0,1,\cdots,n-1.$$

Thus, we have
$$\phi_{i,i}(S_{i,i})S_{i,i+1}-S_{i,i+1}\phi_{i+1,i+1}(S_{i+1,i+1})=0.$$ Since
$S_{i,i}S_{i,i+1}=S_{i,i+1}S_{i+1,i+1}$, then
$$S_{i,i+1}(\phi_{i,i}-\phi_{i+1,i+1})(S_{i+1,i+1})=0.$$
Note that $S_{i,i+1}$ has a dense range, then we can set
$$\phi_{i,i}=\phi, i=0,1,\cdots,n-1.$$

{\sf Claim 2:}\,\,${\mathcal A}^{\prime}(E_t)/\mbox{Rad}{\mathcal
A}^{\prime}(E_t)\cong {\mathcal H}^{\infty}(\mathbb{D}).$

Recall
that $\mbox{Rad}{\mathcal A}^{\prime}(E_t)=\{S\in {\mathcal
A}^{\prime}(E_t)|\sigma_{{\mathcal A}^{\prime}(E_t)}(SS^{\prime})=0,\;S^{\prime} \in {\mathcal A}^{\prime}(E_t)\}.$ Any
$X\in \mathcal A^\prime(E_t)$ is upper triangular by Lemma
\ref{utXY} and ${\mathcal A}^{\prime}(E_t)/\mbox{\rm Rad}\,{\mathcal
A}^{\prime}(E_t)$ is commutative by Lemma \ref{qcomm}. Therefore
if $Y$ is in $\mbox{\rm Rad}\,{\mathcal
A}^{\prime}(E_t)$, then we have
$$Y=
\renewcommand\arraystretch{0.875}
\begin{pmatrix}
0 & Y_{0,1} & Y_{0,2}&\cdots&Y_{0,n-1}\\
&0&Y_{1,2}&\cdots&Y_{1,n-1} \\
&\bigzero&\ddots&\ddots&\vdots\\
&&&0&Y_{n-1,n}\\
&&&&0\\
\end{pmatrix}.\renewcommand\arraystretch{1}$$

Define a map $\Gamma : {\mathcal
A}^{\prime}(E_t)/\mbox{Rad}{\mathcal A}^{\prime}(E_t) \rightarrow
{\mathcal H}^{\infty}(\mathbb{D})$ by the rule:
$$\Gamma ([X])=\phi, \mbox{ where } X=((X_{i,j}))_{n\times n},
X_{i,i}=\phi(S_{i,i}).$$ Obviously, $\Gamma$ is well defined and if
$\Gamma ([X])=0$, then $\phi=0$. Then $X_{i,i}=0$, it follows that
$X\in \mbox{Rad}{\mathcal A}^{\prime}(E_t)$ and $[X]=0$. So $\Gamma$
is injective.

For any $\phi\in  {\mathcal H}^{\infty}(\mathbb{D})$,  set
$X_{i,i}=\phi(S_{i,i}), i=0,1,2,\cdots, n-1$. By Lemma 3.6, we can
construct the operators $X_{i,j},j\neq i$ such that
$X:=(X_{i,j})_{n\times n}\in {\mathcal A}^{\prime}(E_t)$. That means
$\Gamma$ is surjective. Then $\Gamma$ is an isomorphism and
$${\mathcal A}^{\prime}(E_t)/\mbox{Rad}{\mathcal
A}^{\prime}(E_t)\cong {\mathcal H}^{\infty}(\mathbb{D}).$$ By Lemma
\ref{UD} and Lemma \ref{KG}, we have  $$ {{\mbox{\rm
Vect}}}({\mathcal A}^{\prime}(E_t)){\cong}\mathbb{N}, K_0({\mathcal
A}^{\prime}(E_t))\cong \mathbb{Z}.$$ By Lemma \ref{UD}, we have
$E_t$ has a unique finite  decomposition up to similarity.
Similarly, $E_{\widetilde{t}}$ also has a unique finite
decomposition up to similarity.

If  $E_t\sim E_{\widetilde{t}}$, then $(t\oplus \widetilde{t} )\sim
t^{(2)} $. So we have
$${\mbox{\rm Vect}}({\mathcal A}^{\prime}(t\oplus \widetilde{t}))\cong {\mbox{\rm Vect}}({\mathcal
A}^{\prime}(t^{(2)}))\cong {\mbox{\rm Vect}}M_{2}({\mathcal
A}^{\prime}(t)))\cong \mathbb{N}$$ and $$K_{0}({\mathcal
A}^{\prime}(t\oplus \widetilde{t}))\cong \mathbb{Z}.$$  On the other
hand, Note that $t$ and $\widetilde{t}$ are both strongly
irreducible.  If $K_{0}({\mathcal A}^{\prime}(t\oplus
\widetilde{t}))\cong \mathbb{Z}$ and ${\mbox{\rm Vect}}({\mathcal
A}^{\prime}(t\oplus \widetilde{t}))\cong \mathbb{N}$, then by Lemma
\ref{UD}, we have $t\sim \widetilde{t}$, otherwise we will have
$${\mbox{\rm Vect}}({\mathcal A}^{\prime}(t\oplus \widetilde{t}))\cong \mathbb{N}^2.$$
This is a contradiction.
\end{proof}

\begin{proof}[Proof of Theorem \ref{uniqueSim}]
 When  $\Lambda(t)<2$, by Lemma \ref{keyL}, we have $E_t$
 is strongly irreducible. So there exists no non-trivial idempotent
 in ${\mathcal A}^{\prime}(E_t)$, which is the same as saying that  the vector bundle $E_t$ has no non-trivial sub-bundle.

When $\Lambda(t) \geq 2,$ by Lemma \ref{keyL}, we have $$E_r\sim
E_{t_0}\oplus E_{t_1}\oplus \cdots \oplus E_{t_{n-1}}.$$ Since
${\mathcal A}^{\prime}(E_{t_i})\cong {\mathcal
H}^{\infty}(\mathbb{D})$, we have
$${\mathcal A}^{\prime}(E_r)\cong {\mathcal
H}^{\infty}(\mathbb{D})^{(n)},$$ and by Lemma \ref{KG},
$${{\mbox{\rm Vect}}}({\mathcal A}^{\prime}(E_r)){\cong}\mathbb{N}^{(n)}, K_{0}({\mathcal
A}^{\prime}(E_r))\cong \mathbb{Z}^{(n)}.$$ Then by Lemma \ref{UD},
we have $E_t$ has a unique finite  decomposition up to similarity.
Then for any  non-trivial
 reducible sub-bundle of $E_r$ denoted by $E_r$, with $${\mathcal
 H}_r=\mbox{Span}_{w\in \Omega}\{E_{r}(w)\}.$$
Let  $P_t$ be the projection from ${\mathcal H}$ to ${\mathcal
H}_t$. Then
 $$E_t\sim E_r\oplus (E_t\ominus E_r)= P_rE_t\oplus (I-P_r)E_t.$$
 Let $$P_{t_i}:{\mathcal H}\rightarrow {\mathcal H}_{i}:=\mbox{Span}_{\lambda\in
 \Omega}\{E_{t_i}(w)\}, i=0,1,\cdots,n-1$$ be projections in ${\mathcal
 A}^{\prime}(E_r)$.  Then there exists $t_{k_i}, i=0,1,\cdots, s$
 such that
 $$P\sim \oplus_{i=0}^sP_{t_{k_i}}.$$ Then it follows that
 $$E_r\sim \oplus_{i=0}^sP_{t_{k_i}}E_t
\sim \oplus_{i=0}^sE_{t_{k_i}},$$  namely, there exists an
invertible operator $X$ such that
$E_r=X(\oplus_{i=0}^sE_{t_{k_i}})$. Suppose that
$$\oplus_{i=0}^{n-1}E_{t_i}=(\oplus_{i=0}^sE_{t_{k_i}})\oplus
(\oplus_{i=0}^{n-s}E_{t_{l_i}}).$$ Set
$$E_s=X(\oplus_{i=0}^{n-s}E_{t_{l_i}}),$$ then we have $$E_r\oplus E_s\sim E_t.$$
And if there exists another bundle $E_{s^{\prime}}$ such that
$$E_r\oplus E_{s^{\prime}}\sim E_t.$$
Since $E_r$ has a unique finite  decomposition up to similarity,
then we have $$E_{s^{\prime}}\sim\oplus_{i=0}^{n-s}E_{t_{l_i}}\sim
E_{s}.$$
\end{proof}

Before we give the proof of Theorem \ref{K^0=K_0}, we also need the following lemma from \cite{Jia}.

\begin{lem} \label{UD1}\mbox{\rm (\cite[Lemma 2.6]{Jia})} Let $\{P_1,\cdots,P_m,P_{m+1},\cdots, P_{N}\}$ and $\{Q_{1},\cdots,Q_{m+1},\cdots,Q_{N}\}$ be two distinct unique decompositions of the vector bundle $E_t$. Suppose that for $1\leq i\leq m$ and $w\in \Omega,$
\begin{enumerate}
\item  there exists an $X_{i}\in GL(P_i{\mathcal H}, Q_{i}{\mathcal H})$ satisfying $X_{i}P_it(w)=Q_{i}e(w)$ and that
\item there exists $Y\in GL({\mathcal A}^{\prime}(E_t))$  and a
permutation $\Pi\in S_{n}$ satisfying $Y^{-1}P_{i}Y=Q_{\Pi (i)}.$
\end{enumerate}
Then for $r\in\{m+1, \cdots n\}$ and given $Q_{r},$ there exists
$r^{\prime}\in\{m+1,\cdots, n\}$ and $Z_{r}\in GL(Q_{r}{\mathcal H},
P_{r^{\prime}}{\mathcal H}) $ satisfying
$Z_{r}Q_{r}t(w)=P_{r^{\prime}}t(w),$ $w\in \Omega$. Furthermore, if
$r_{1}\neq r_{2}$, then $r_{1}^{\prime}\neq r_{2}^{\prime}$.
\end{lem}

\begin{proof}[Proof of Theorem \ref{K^0=K_0}]
 Let $P\in P_{n}({\mathcal A}^{\prime}(E_t))=P({\mathcal
 A}^{\prime}(E_{t^n}))$ be an idempotent. Then we have  $PE_{t^n}$ be a sub-bundle of $E_{t^n}$. Define map
 $$\Gamma: V({\mathcal A}^{\prime}(E_t)))\rightarrow V^0(E_t)$$
 with $\Gamma([p]_0)=\overline{PE_{t^n}}.$

 Firstly, we will prove $\Gamma$ is well defined. In fact, for any
 $P\sim Q\in [P]_0$, there exists positive integer  $n$ such that $P, Q \in {\mathcal
 A}^{\prime}(E_{t^n})$. Since $Q=XPX^{-1}, X\in {\mathcal
 A}^{\prime}(E_t),$ then we have  $$QE_{t^n}=XPX^{-1}E_{t^n}\sim PX^{-1}E_{t^n}.$$
And Note that $X, X^{-1} \in {\mathcal
 A}^{\prime}(E_t)$, then we have  $$X^{-1}t^n(w)=t^n(w), \mbox{for any}~~ w\in \Omega.$$
 Thus,  $$QE_{t^n}\sim PXE_{t^n},$$ and
 $\overline{QE_{t^n}}=\overline{PE_{t^n}}.$  So $\Gamma$ is well
 defined.

 Secondly, we will prove that $\Gamma$ is surjective. Suppose that
 $E_r$ is a sub-bundle of $E_{t^n}$ with dimension $K$, where $n$ is positive integer.
 Suppose that $${\mathcal H}_r:=\bigvee\limits_{w\in
 \Omega}\{\gamma_1(w), \gamma_2(w),\cdots,
 \gamma_K(w)\},$$
and $P_r$ is the
 projection from ${\mathcal H}$ to ${\mathcal H}_r$, then we have
 $P_r\in {\mathcal
 A}^{\prime}(E_{t^n})$ and $$P_r E_{t^n}\sim E_r.$$ Then it follows
 that $\Gamma$ is surjective.

At last, we will prove that $\Gamma$ is also injective. Let $P,Q\in
{\mathcal A}^{\prime}(E_{t^n})$. Suppose that there exists an
invertible operator $X\in {\mathcal A}^{\prime}(E_{t^n})$ such that
$$XPE_{t^n}=QE_{t^n}.$$ Let $\{p_1,p_2,\cdots,p_m\}$ be a
decomposition of $P$. Then
$\{Xp_1X^{-1},Xp_2X^{-1},\cdots,Xp_mX^{-1}\}$ be a decomposition of
$Q$. In fact, we have
$$\begin{array}{llll}
Xp_1X^{-1}QE_{t^n}+Xp_2X^{-1}QE_{t^n}+\cdots+Xp_mX^{-1}QE_{t^n}&=&Xp_1E_{t^n}+Xp_2E_{t^n}+\cdots+Xp_mE_{t^n}\\
&=&XPE_{t^n}\\
&=&QE_{t^n}.
\end{array}
$$
Suppose that $\{p_{m+1},p_{m+2},\cdots,p_{N}\}$ and
$\{q_{m+1},q_{m+2},\cdots, q_{N}\}$ be the decompositions of
$(I-P)E_{t^n}$ and $(I-Q)E_{t^n}$ respectively. Then we have
$$ \{p_1,p_2,\cdots, p_N\}~\mbox{\rm{and}},~\{Xp_1X^{-1},Xp_2X^{-1},\cdots,Xp_mX^{-1}, q_{m+1},q_{m+2},\cdots,
q_N\}$$ are two different decompositions of $E_{t^n}$.  By the
uniqueness of decomposition of  $E_{t^n}$, there exists an
invertible bounded linear operator  $Y\in {\mathcal
A}^{\prime}(E_{t^n})$ such that $\{Y^{-1}P_iY\}$ is a rearrangement
of
$\{Xp_1X^{-1},Xp_2X^{-1},\cdots,Xp_mX^{-1}\},\{q_{m+1},q_{m+2},\cdots,
q_N\}.$ By Lemma \ref{UD1}, for any $v\in \{m+1,m+2,\cdots,n\},$ we
can find $p_{v^{\prime}}, v^{\prime}\in\{m+1,\cdots,n\}$ and $Z_v\in
GL(L(q_v{\mathcal H}, p_v{\mathcal H})$ such that
$$Z_vq_vE_{t^n}=p_{v^{\prime}}E_{t^n},
v^{\prime}_1=v^{\prime}_2,~~\mbox{when}~~v_1=v_2.$$ Set
$Z_k=X^{-1}|_{Xp_kX^{-1}{\mathcal H}},k=1,2,\cdots m$, then we have
that
$$Z=\sum\limits_{k=1}^m Z_k+\sum\limits_{v=m+1}^N Z_v\in GL{\mathcal A}^{\prime}(E_{t^n}),$$
and $$ZPZ^{-1}=Q.$$ It follows that $\Gamma$ is injective. Since
$\Gamma$ is also a homomorphism, then we have
$${\mbox{\rm Vect}}^0(E_t)\cong {\mbox{\rm Vect}}({\mathcal A}^{\prime}(E_{t^n}), K^0(E_t)\cong K_0({\mathcal
A}^{\prime}(E_t)).$$

\end{proof}

\subsection{\sf More general results on similarity}
The precise relationship between the non-vanishing holomorphic sections of the atoms $t_i, 0\leq i \leq n-1,$ of a quasi-homogeneous holomorphic curve $t$ and a holomorphic frame for $t$ mandated in Definition \ref{defq} is at the heart of the proof of  Theorem \ref{mainthmSim}. We now push the limits of this definition a little and see if we can replicate some of our results. We begin by making the observation that starting with a quasi-homogeneous holomorphic curve  $t,$ we always have an operator $T$ in the Cowen-Douglas class $B_n(\mathbb D).$ This operator has an upper triangular decomposition as in Lemma \ref{atomic}. However, the other way round, starting with an operator $T$ possessing such a decomposition, it may not be possible find a frame $\gamma$ for the holomorphic Hermitian vector bundle $E_T,$ which can be written as linear combinations of the non-vanishing sections of the atoms and their derivatives. In this section, we start with an operator $T$ in the Cowen-Douglas class $B_n(\mathbb D)$ assume that the operators appearing on the diagonal in its decomposition according to Theorem \ref{ut} are homogeneous operators in $B_1(\mathbb D).$
Finally, we require that unlike quasi-homogeneous operators, there exists a holomorphic frame $\gamma$ for $E_T,$ which is a linear combination of the non-vanishing sections of the atoms and its derivatives as in Definition \ref{defq} except that the coefficients $\mu_{i,j}$ are allowed to be holomorphic functions rather than constants. For the remaining portion of this subsection, let $\mathcal Q_n(\mathbb D)$ denote this class of operators, or for that matter, the corresponding holomorphic Hermitian vector bundles.

\begin{prop} \label{sim2} Let $E_T$ and $E_{\widetilde{T}}$
be two holomorphic Hermitian vector  bundles in $\mathcal Q_n(\mathbb D)$  with atoms
$T_i,\widetilde{T}_i\;i=0,1$ an atomic decomposition $\big(\!\!
\big (S_{i,j}\big )\!\!\big )$ and $\big(\!\! \big (
\widetilde{S}_{i,j}\big )\!\!\big ),$ respectively. Suppose that
$S_{i,i}=\widetilde{S}_{i,i}$, $i=0,1$.
Then $E_T$ and $E_{\widetilde{T}}$ are similar if and only if there exists an invertible holomorphic function $\phi\in {\mathcal H}^{\infty}(\mathbb{D})$ such that
$\tilde{S}_{0,1}=\phi(T_0)S_{0,1}$.
\end{prop}

For $i=1,2,$ let ${\mathcal H_i}$ be a Hilbert
space of holomorphic function on $\mathbb D$ possessing a reproducing kernel, say $K_i,$ and $T_i$ be the adjoint of the multiplication operator on $\mathcal H_i.$ Assume that $\mathcal H_0 \subseteq \mathcal H_1$ and let $\iota:\mathcal H_0 \to \mathcal H_1$ be the inclusion map. Then the adjoint $\iota^*$ of the inclusion map has the property $\iota^*(K_1(\cdot,w))=K_0(\cdot,w),\;w\in \mathbb D.$

\begin{lem}\label{inc*}
Assume that $K_i(z,w)=\frac{1}{(1-z{\overline
w})^{\lambda_i}}, i=0,1.$
Suppose that $S:{\mathcal H_0} \to {\mathcal H_1}$ is a bounded linear operator with the intertwining property $T_0S=S T_1.$ Then there exists a holomorphic function $\phi$ such that
$S=\phi(T_0)\iota^*.$
\end{lem}

\begin{proof} The operators $T_i,\; i=0,1$ are in $B_1(\mathbb D).$  If $S:{\mathcal H_0} \to {\mathcal H_1}$ is a bounded linear operator and $T_0S=ST_1,$ then there exists a holomorphic function
$\psi$ such that $S^*=M_{\psi}$. This is easily proved as in
\cite[Section 5]{KM0}. Let $\phi$  be the holomorphic function defined on the
unit disc by the formula $\overline{\phi(\overline{w})}= \psi(w),\;w
\in \mathbb{D}.$ For any $f\in {\mathcal H_0}$, we have that
$$\begin{array}{llll}
\langle f(z),\phi(T_0)\iota^*(K_1(z,w))\rangle&=&\langle f(z),\phi(\overline{w})K_0(z,w)\rangle\\
&=&\overline{\phi(\overline{w})}\langle f(z),\phi(\overline{w})K_0(z,w)\rangle\\
&=&\langle f(z),M^*_{\psi}(K_1(z,w))\rangle\\&=&\langle f(z),S(K_1(z,w))\rangle.
\end{array}
$$
Consequently, $S=\phi(T_0)i^*.$
\end{proof}

%
%
%
%
%
%

\begin{proof}[Proof of Proposition \ref{sim2}]

Let $T=\left (\begin{smallmatrix}T_0 & S_{0,1} \\
0 & T_1 \\
\end{smallmatrix}\right )$
and $\tilde{T}=\left ( \begin{smallmatrix}T_0 & \tilde{S}_{0,1} \\
0 & T_1 \\
\end{smallmatrix}\right )$ be the representations in
upper-triangular operator matrices of $E_t$ and $E_{\widetilde{t}}$
respectively with the following properties:

By Lemma \ref{atomic}, we know that $E_t\sim E_{\widetilde{t}}$ if and only
if $T\sim \widetilde{T}$. Then we only need to prove that $T$ and
$\tilde{T}$ are similarity equivalent if and only if
 there exists invertible holomorphic function
$\phi$ such that $\tilde{S}_{0,1}=\phi(T_0)S_{0,1}$.

To prove the necessity, note that there exists $\psi,
\widetilde{\psi}\in {\mathcal H}^{\infty}(\mathbb{D})$ such that
$$S_{0,1}=\psi(T_0)i^*,\widetilde{S}_{0,1}=\widetilde{\psi}(T_0)i^*$$
by Lemma \ref{inc*}. If there exists an
invertible operator $Y=\Big (\begin{smallmatrix}Y_{0,0} & Y_{0,1}\\
0 & Y_{1,1} \\
\end{smallmatrix}\Big )$ such that
\begin{equation} \label{2x2int}
\begin{pmatrix}Y_{0,0} & Y_{0,1} \\
0 & Y_{1,1} \\
\end{pmatrix}\begin{pmatrix}T_0 & \psi(T_0)i^* \\
0 & T_1 \\
\end{pmatrix}=\begin{pmatrix}T_0 & \widetilde{\psi}(T_0) i^* \\
0 & T_1 \\
\end{pmatrix}\begin{pmatrix}Y_{0,0} & Y_{0,1} \\
0 & Y_{1,1} \\
\end{pmatrix}, \end{equation}
then $Y_{0,0}$ and $Y_{1,1}$ belong to the commutant of
$T_0$ and $T_1,$ respectively. The operator $Y$ is invertible and its inverse $Y^{-1}$ is
upper-triangular. The two operators $Y_{0,0}$ and $Y_{1,1}$ are also invertible. From  Equation \eqref{2x2int}, we have that
$$Y_{0,0}\psi(T_0)i^*+Y_{0,1}T_1=T_0Y_{0,1}+\widetilde{\psi}(T_0)i^*Y_{1,1}.$$
As in the proof of the Lemma \ref{keyL}, we also have that
$$Y_{0,0}\psi(T_0)i^*=\widetilde{\psi}(T_0)i^*Y_{1,1}.$$
Since $Y_{0,0}$ and $Y_{1,1}$ belongs to the commutant of $T_0$ and
$T_1$ respectively, there exists invertible holomorphic functions
$\phi_{0,0}$ and $\phi_{1,1}\in {\mathcal H}^{\infty}(\mathbb{D})$
such that $$Y_{i,i}=\phi_{i,i}(T_i), i=0,1.$$ Consequently,
$$Y_{0,0}\psi(T_0)i^*-\widetilde{\psi}(T_0)i^*Y_{1,1}=(\phi_{0,0}(T_0)\psi(T_0)-\phi_{1,1}(T_0)\widetilde{\psi}(T_0))i^*=0.$$
Thus
$\widetilde{\psi}(T_0)=\phi^{-1}_{1,1}(T_0)\phi_{0,0}(T_0)\psi(T_0).$
Set $\phi=\phi^{-1}_{1,1}\phi_{0,0}$, then we have that
$\widetilde{S}_{0,1}=\phi(T_0)S_{0,1}$ and $\phi(T_0)$ is
invertible. This completes the proof of the necessary part.

We now prove the sufficiency. If there exists $\phi\in {\mathcal
H}^{\infty}(\mathbb{D})$ such that $\phi(T_0)$ is invertible
$\widetilde{S}_{0,1}=\phi(T_0)S_{0,1},$ then
$$\begin{pmatrix}\phi^{-1}(T_0) & 0 \\
0 & I\\
\end{pmatrix}\begin{pmatrix}T_0 & \widetilde{S}_{0,1} \\
0 & T_1 \\
\end{pmatrix}=\begin{pmatrix}T_0 & S_{0,1} \\
0 & T_1 \\
\end{pmatrix}\begin{pmatrix}\phi^{-1}(T_0) & 0\\
0 & I \\
\end{pmatrix}.$$ Therefore $\widetilde{T}$ is similar to $T$.

\end{proof}

The following proposition is similar to the one we have just proved for operators in $\mathcal Q_n(\mathbb D),$ $n=2.$ Here we give the proof only for $n=3.$  The proof for an arbitrary $n$ can be made up  without involving any new ideas. It requires more of the same but somewhat tedious computations which we choose to skip.

\begin{prop}
Let $E_t, E_{\widetilde{t}}$ be two holomorphic Hermitian  vector bundles in $\mathcal Q_n(\mathbb D)$ with
atomic decompositions
$\big(\!\! \big (S_{i,j}\big )\!\!\big )$ and $\big(\!\! \big (
\widetilde{S}_{i,j}\big )\!\!\big ),$ respectively. Assume that
$S_{i,i}=\widetilde{S}_{i,i}$, $i=0,1,\cdots, n-1$.
If there exist holomorphic functions $\phi_{i,j}\in {\mathcal
H}^{\infty}(\mathbb{D})$ such that
$\tilde{S}_{i,j}=\phi_{i,j}(T_i)S_{i,j}$, $i,j=0,1,\cdots,n-1, i<j,$ such that $\widetilde{S}_{i,j}=\phi_{i,j}(T_i)S_{i,j}, \phi_{i,j}\in {\mathcal H}^{\infty}(\mathbb{D}),$
then $E_t$ and $E_{\widetilde{t}}$ are similarity equivalent if and only if
$\phi_{i,j}(T_i)$ are all invertible and
$\phi_{i,j}=\phi_{i,i+1}\phi_{i+1,i+2}\cdots\phi_{j-1,j}.$
%
%
%
\end{prop}

%
%
%
%

\begin{proof}
Let $T=\left (\begin{smallmatrix}
T_0 & S_{0,1} & S_{0,2}\\
0&T_1&S_{1,2}&\\
0&0&T_2\\
\end{smallmatrix}\right ), \widetilde{T}=\left (\begin{smallmatrix}
T_0 & \widetilde{S}_{0,1} & \widetilde{S}_{0,2}\\
0&T_1&\widetilde{S}_{1,2}&\\
0&0&T_2\\
\end{smallmatrix}\right )$  be the atomic decomposition of $T$ and ${\widetilde{T}},$ respectively. We prove that  $T\sim \widetilde{T}$ if and only
if $\phi_{i,j}(T_i), i,j\leq 2$ are invertible and $\phi_{0,2}=\phi_{0,1}\phi_{1,2}.$
%
Since $T\sim \widetilde{T}$, by Lemma \ref{utXY}, there exists an
invertible operator $X=\big (\!\!\big (X_{i,j} \big )\!\!\big ),$
which is upper triangular and
such that
$$\left (\begin{matrix}
X_{0,0} & X_{0,1} & X_{0,2}\\
0&X_{1,1}&X_{1,2}&\\
0&0&X_{2,2}\\
\end{matrix}\right )\left (\begin{matrix}
T_0 & S_{0,1} & S_{0,2}\\
0&T_1&S_{1,2}&\\
0&0&T_2\\
\end{matrix}\right )=\left (\begin{matrix}
T_0 & \widetilde{S}_{0,1} & \widetilde{S}_{0,2}\\
0&T_1&\widetilde{S}_{1,2}&\\
0&0&T_2\\
\end{matrix}\right )\left (\begin{matrix}
X_{0,0} & X_{0,1} & X_{0,2}\\
0&X_{1,1}&X_{1,2}&\\
0&0&X_{2,2}\\
\end{matrix}\right ).$$
Then we have that
$$\left (\begin{matrix}
X_{0,0} & X_{0,1}\\
0&X_{1,1}
\end{matrix}\right )\left (\begin{matrix}
T_0 & S_{0,1}\\
0&T_1
\end{matrix}\right )=\left (\begin{matrix}
T_0 & \widetilde{S}_{0,1} \\
0&T_1
\end{matrix}\right )\left (\begin{matrix}
X_{0,0} & X_{0,1}\\
0&X_{1,1}
\end{matrix}\right )$$
and
$$\left (\begin{matrix}
X_{1,1}&X_{1,2}\\
0&X_{2,2}\\
\end{matrix}\right )\left (\begin{matrix}
T_1&S_{1,2}\\
0&T_2\\
\end{matrix}\right )=\left (\begin{matrix}
T_1&\widetilde{S}_{1,2}\\
0&T_2
\end{matrix}\right )\left (\begin{matrix}
X_{1,1}&X_{1,2}\\
0&X_{2,2}\\
\end{matrix}\right ).$$
The inverse $X^{-1}$ is also  upper-triangular.
Now, both $\left (\begin{smallmatrix}
X_{0,0} & X_{0,1}\\
0&X_{1,1}
\end{smallmatrix}\right )$ and $\left (\begin{smallmatrix}
X_{1,1}&X_{1,2}\\
0&X_{2,2}\\
\end{smallmatrix}\right )$ are invertible and consequently, by Proposition \ref{sim2}, we see that $\phi_{0,1}$ and $\phi_{1,2}$ must be invertible.

Set $\phi_0=1, \phi_1=\phi_{0,1}, \phi_2=\phi_{0,1}\phi_{1,2}$,
$X_i:=X_{i,i}=\phi_i(T_i),$ $\phi_i\in {\mathcal
H}^{\infty}(\mathbb{D}).$  We have that
$$X\widetilde{T}X^{-1}=\begin{pmatrix}
X_0T_0X_0^{-1} & X_0\widetilde{S}_{0,1}X^{-1}_1 & X_0\widetilde{S}_{0,2}X^{-1}_2\\
0&X_1T_1X^{-1}_1&X_1\widetilde{S}_{1,2}X^{-1}_2&\\
0&0&X_2T_2X^{-1}_2\\
\end{pmatrix}.$$
Consequently,
$$\begin{array}{llll}\widetilde{T}&\sim& \begin{pmatrix}
X_0T_0X_0^{-1} & X_0\widetilde{S}_{0,1}X^{-1}_1 & X_0\widetilde{S}_{0,2}X^{-1}_2\\
0&X_1T_1X^{-1}_1&X_1\widetilde{S}_{1,2}X^{-1}_2&\\
0&0&X_2T_2X^{-1}_2\\
\end{pmatrix}\\
&=&\begin{pmatrix}
X_0T_0X_0^{-1} & X_0\phi_{0,1}(T_0)S_{0,1}X^{-1}_1 & X_0\phi_{0,2}(T_0)S_{0,2}X^{-1}_2\\
0&X_1T_1X^{-1}_1&X_1\phi_{1,2}(T_0)S_{1,2}X^{-1}_2&\\
0&0&X_2T_2X^{-1}_2\\
\end{pmatrix}\\
 &=&\begin{pmatrix}
T_0 & S_{0,1} & X_0\phi_{0,2}(T_0)S_{0,2}X^{-1}_2\\
0&T_1&S_{1,2}&\\
0&0&T_2\\
\end{pmatrix}\\
&=&\begin{pmatrix}
T_0 & S_{0,1} & \phi_{0,2}(T_0)S_{0,2}\phi^{-1}_2(T_2)\\
0&T_1&S_{1,2}&\\
0&0&T_2\\
\end{pmatrix}
\end{array}$$

Now set $$\overline{T}=\begin{pmatrix}
T_0 & S_{0,1} & \phi_{0,2}(T_0)S_{0,2}\phi^{-1}_2(T_2)\\
0&T_1&S_{1,2}&\\
0&0&T_2\\
\end{pmatrix},$$ and $\overline{S}_{0,2}= \phi_{0,2}(T_0)S_{0,2}\phi^{-1}_2(T_2)$. Since $T\sim
\widetilde{T}\sim\overline{T}$, by Lemma \ref{commutant2}, we find
an invertible operator $X=\left ( \begin{smallmatrix}
X_{0,0} & X_{0,1} & X_{0,2}\\
0&X_{1,1}&X_{1,2}&\\
0&0&X_{2,2}\\
\end{smallmatrix}\right )$ such that
$$\begin{pmatrix}
T_0 & S_{0,1} & S_{0,2}\\
0&T_1&S_{1,2}&\\
0&0&T_2\\
\end{pmatrix}\begin{pmatrix}
X_{0,0} & 0 & X_{0,2}\\
0&X_{1,1}&0&\\
0&0&X_{2,2}\\
\end{pmatrix}=
\begin{pmatrix}
X_{0,0} &0 & X_{0,2}\\
0&X_{1,1}&0&\\
0&0&X_{2,2}\\
\end{pmatrix}
\begin{pmatrix}
T_0 & S_{0,1} & \overline{S}_{0,2}\\
0&T_1&S_{1,2}&\\
0&0&T_2\\
\end{pmatrix}.
$$

Then we have
\begin{equation} \label{range1}  X_{0,0}\overline{S}_{0,2}-S_{0,2}X_{2,2}=S_{0,0}X_{0,2}-X_{0,2}S_{2,2}, \end{equation}
and $X_{i,i}\in {\mathcal A}^{\prime}(T_i).$ In the following, we
will describe $X_{0,0}\overline{S}_{0,2}-S_{0,2}X_{2,2}$.

Note that $X_{i,i}(t_i)(w)=\phi(w)t_i(w),\,  w\in \Omega$,
where $\phi$ is a holomorphic function on $\mathbb{D}$ and
\begin{equation} \label{range2}
\begin{array}{lllllll}
X_{0,0}\overline{S}_{0,2}-S_{0,2}X_{2,2}(t_2)&=&X_{0,0}(\phi_{0,2}(T_2)S_{0,2}(\phi_2^{-1}(T_2)(t_2)))-S_{0,2}(\phi
t_2)\\
&=&-X_{0,0}(\phi_{0,2}(T_0)\phi^{-1}_2t^{(1)}_0+\phi t^{(1)}_0\\
&=&-\phi^{-1}_2X_{0,0}(\phi_{0,2}(T_0)t^{(1)}_0)+\phi t^{(1)}_0\\
&=&-\phi^{-1}_2X_{0,0}(\phi_{0,2}t^{(1)}_0+\phi^{(1)}_{0,2}t_0)+\phi
t^{(1)}_0\\
&=&-\phi^{-1}_2X_{0,0}(\phi_{0,2}t^{(1)}_0)-\phi^{-1}_2\phi^{(1)}_{0,2}X_{0,0}(t_0)+\phi
t^{(1)}_0\\
&=&-\phi^{-1}_2\phi_{0,2}X_{0,0}(t^{(1)}_0)-\phi^{-1}_2\phi^{(1)}_{0,2}(\phi
t_0)+\phi
t^{(1)}_0\\
&=&-\phi^{-1}_2\phi_{0,2}((\phi
t_0)^{(1)})-\phi^{-1}_2\phi^{(1)}_{0,2}(\phi t_0)+\phi
t^{(1)}_0\\
&=&-\phi^{-1}_2\phi_{0,2}(\phi
t_0^{(1)})-\phi^{-1}_2\phi_{0,2}(\phi^{(1)}
t_0)-\phi^{-1}_2\phi^{(1)}_{0,2}(\phi t_0)+\phi
t^{(1)}_0\\
&=&\phi(1-\phi^{-1}_2\phi_{0,2})
t_0^{(1)}-\phi^{-1}_2\phi_{0,2}\phi^{(1)}
t_0-\phi^{-1}_2\phi^{(1)}_{0,2}\phi t_0
\\
&=&\phi(-1+\phi^{-1}_2\phi_{0,2})
S_{0,2}(t_2)-S_{0,1}S_{1,2}(\phi^{-1}_2(\phi_{0,2}\phi)^{(1)})S_{2,2}(t_2)
.
\end{array}
\end{equation}

Let $S\in {\mathcal L}({\mathcal H}_2, {\mathcal H}_0)$ satisfies
that $S(t_2)=\phi(\phi^{-1}_2\phi_{0,2}-1) S_{0,2}(t_2).$ Let
$\psi=\phi(\phi^{-1}_2\phi_{0,2}-1)$, then we have
$S(t_2)=\psi(S_{0,0}) S_{0,2}(t_2)$. Note that
$S_{0,1}S_{1,2}(\phi^{-1}_2(\phi_{0,2}\phi)^{(1)})S_{2,2}\in
\mbox{\rm ran}\,\sigma_{T_0,T_2}$. By \eqref{range1} and
\eqref{range2}, there exist linear bounded operator $Z$ such that
$$S(t_2)=(T_0Z-ZT_2)(t_2).$$
That means
$$\begin{array}{llll}
S_{0,2}(t_2)&=&T_0Z(\psi^{-1}(T_2)(t_2))-ZT_2(\psi^{-1}(T_2)(t_2)\\
&=&T_0Z\psi^{-1}(T_2)(t_2)-Z\psi^{-1}(T_2)T_2(t_2). \end{array}$$

This is a contradiction to the fact
$S_{0,2}\not\in\mbox{\rm ran}\,\sigma_{T_0,T_2}.$ So we have that
$\phi_2=\phi_{0,2}=\phi_{0,1}\phi_{1,2}$.

To prove the proposition the other way round, assume that $\widetilde{T}=\begin{pmatrix}
T_0 & \phi_{0,1}(T_0)S_{0,1} & \phi_{0,2}(T_0)S_{0,2}\\
0&T_1&\phi_{1,2}(T_0)S_{1,2}&\\
0&0&T_2\\
\end{pmatrix}$ and $\phi_{0,1}$, $\phi_{1,2}$ and
$\phi_{0,2}=\phi_{0,1}\phi_{1,2}$ are all invertible. Then we have
$$\widetilde{T}\sim  \begin{pmatrix}
T_0 & S_{0,1} & \phi_{0,2}(T_0)S_{0,2}(\phi_{0,1}\phi_{1,2})^{-1}(T_2)\\
0&T_1&S_{1,2}&\\
0&0&T_2
\end{pmatrix} = \begin{pmatrix}
T_0 & S_{0,1} & \phi_{0,2}(T_0)S_{0,2}\phi^{-1}_{0,2}(T_2)\\
0&T_1&S_{1,2}&\\
0&0&T_2
\end{pmatrix}.$$ We also have
\begin{eqnarray*}
\lefteqn{
\begin{pmatrix}
\phi^{-1}_{0,2}(T_0) &0& 0\\
0&\phi^{-1}_{0,2}(T_1)&0\\
0&0&\phi^{-1}_{0,2}(T_2)\\
\end{pmatrix}\begin{pmatrix}
T_0 & S_{0,1} & \phi_{0,2}(T_0)S_{0,2}\phi^{-1}_{0,2}(T_2)\\
&T_1&S_{1,2}&\\
&&T_2\\
\end{pmatrix}\begin{pmatrix}
\phi_{0,2}(T_0) &0& 0\\
0&\phi_{0,2}(T_1)&0\\
0&0&\phi_{0,2}(T_2)\\
\end{pmatrix}}\\
&\phantom{GadadharJiangJi}=&\begin{pmatrix}
T_0 & \phi^{-1}_{0,2}(T_0)S_{0,1}\phi_{0,2}(T_1) & \phi^{-1}_{0,2}(T_0)\phi_{0,2}(T_0)S_{0,2}\phi^{-1}_{0,2}(T_2)\phi_{0,2}(T_2)\\
0&T_1&\phi^{-1}_{0,2}(T_1)S_{1,2}\phi_{0,2}(T_2)&\\
0&0&T_2\\
\end{pmatrix}\\
&\phantom{GadadharJiangJi}=&\begin{pmatrix}
T_0 & S_{0,1} & S_{0,2}\\
0&T_1&S_{1,2}&\\
0&0&T_2\\
\end{pmatrix}.
\end{eqnarray*}
\end{proof}
\subsection{\sf The Halmos' question} The well-known question of Halmos asks if $\varrho:\mathbb C[z] \to \mathcal L(\mathcal H)$ is a continuous (for $p\in \mathbb C[z],$ the norm $\|p\| =\sup_{z\in\mathbb D} |p(z)|$) algebra homomorphism induced by an operator $S$, that is, $\varrho(p)=p(S),$ then does there exist an invertible linear operator $L$ and a contraction $T$ on the Hilbert space $\mathcal H$ so that $S = L T L^{-1}.$ After the question was raised in \cite[Problem 6]{Hal}, an affirmative answer for several classes of operators were given. A counter example was found by Pisier in 1996 (cf. \cite{Pi}). It was pointed out in a recent paper of the third author with Kor\'{a}nyi \cite{KM} that the Halmos' question has an affirmative answer for homogeneous operators in the Cowen-Douglas class $B_n(\mathbb D)$. This was based on the description of equivalence classes of homogeneous operators under invertible bounded linear transformations. In the terminology of this paper, (multiplicity free) homogeneous operators are irreducible and also  strongly reducible.  Now, we have this for quasi-homogeneous operators, see Theorem \ref{Sirrd}. Thus it is natural to ask if the Halmos' question has an affirmative answer for quasi-homogeneous operators. If $\Lambda(t) \geq 2,$ the answer is evidently ``yes'':

In this case, the quasi-homogeneous operator $T$ is similar to the $n$- fold direct sum of the homogeneous operators $T_i$ (adjoint of the multiplication operator) acting on the weighted Bergman spaces $\mathbb A^{(\lambda_i)}(\mathbb D),$ $i=0,1,\ldots , n-1.$  Now, if  $\lambda_0 \geq 1,$ this direct sum is contractive and we are done. If $\lambda_0 < 1,$ then $T_0$ is not even power bounded and therefore neither is the operator $T$. So, there is nothing to prove when $\lambda_0<1.$

If $\Lambda(t) < 2,$ then the operator $T$ is strongly irreducible. Therefore, we can't answer the Halmos' question purely in terms of the atoms of the operator $T$. Never the less, the answer is affirmative even in this case. To show this, we need a preparatory lemma.
\begin{lem}
Suppose that $t$ is a quasi-homogeneous holomorphic curve. Assume that $\Lambda(t) <2$ and $\lambda_0 \geq 1.$
Then the operator $T$ is not power bounded.
\end{lem}
\begin{proof} The top $2\times 2$ block in the atomic decomposition of the quasi-homogeneous operator $T$ is of the form $\Big ( \begin{smallmatrix}T_0 & S_{0,1}\\0& T_1 \end{smallmatrix}\Big ).$ As always, we assume that the operators $T_0$ and $T_1$ are the adjoints of the multiplication operator on the weighted Bergman spaces $\mathbb A^{(\lambda_0)}(\mathbb D)$ and  $\mathbb A^{(\lambda_1)}(\mathbb D),$ respectively. The operator $S_{0,1}$ has the intertwining property $T_0 S_{0,1} = S_{0,1} T_1.$

Let $\iota$ denote the inclusion map from $\mathbb A^{(\lambda_0)}(\mathbb D)$ to $\mathbb A^{(\lambda_1)}(\mathbb D)$. Then $\iota^*(t_1)(w)=t_0(w),\; w\in \mathbb D,$ and the operator $S_{0,1}$ must be of the form $\phi(T_{0})\iota^{*}$ for some holomorphic function $\phi$ on the unit disc $\mathbb D$, as we have shown in Lemma \ref{inc*}.
Indeed, $S_{0,1}(t_1(w))=\phi(w) t_1(w) = \phi(T_0)\iota^{*}(t_1(w)).$

Without loss of generality, we assume that
$\phi(w)=\sum\limits_{i=0}^{\infty}\phi_iw^i$ and $\phi_0\neq 0$.
For  $j=0,1$,  the set of vectors $e^{(\lambda_{j})}_\ell:=
\sqrt{a_\ell(\lambda_j)}\,z^\ell,\;\ell\geq 0,$ is an orthonormal
basis in $\mathbb A^{(\lambda_j)}(\mathbb D).$  Then we have that
$$T^{n-1}_{0}(e_{\ell}(\lambda_0))=\prod\limits_{i=\ell-n+1}^{\ell-1}w_i(\lambda_0)e_{\ell-n+1}(\lambda_0), S_{0,1}(e_{\ell}(\lambda_1))
=\phi_0\frac{\prod\limits_{i=0}^{\ell-1}w_i(\lambda_1)}{\prod\limits_{i=0}^{\ell-1}w_i(\lambda_0)}e_{\ell}(\lambda_0).$$
Consequently,
$$nT^{n-1}_{0}S_{0,1}(e_{\ell}(\lambda_1))=n\phi_0\frac{\prod\limits_{i=0}^{\ell-1}w_i(\lambda_1)}{\prod\limits_{i=0}^{\ell-n}w_i(\lambda_0)}
e_{\ell-n+1}(\lambda_0)$$
Since $w_i(\lambda_0)=\sqrt{\frac{i+1}{i+\lambda_0}}$ and
$w_i(\lambda_1)=\sqrt{\frac{i+1}{i+\lambda_1}}$, it follows that
$$\prod\limits_{i=0}^{\ell-1}w_i(\lambda_1)\sim \big (({\ell-1})^{\frac{1-\lambda_1}{2}}\big )\;\mbox{\rm and}\; \prod\limits_{i=0}^{\ell-n}w_i(\lambda_0)\sim \big ({(\ell-n)}^{\frac{1-\lambda_0}{2}}\big )$$
implying
$$\frac{\prod\limits_{i=0}^{\ell-n}w_i(\lambda_1)}{\prod\limits_{i=0}^{\ell-1}w_i(\lambda_0)}\sim
\Big (\frac{(\ell-n)^{\frac{\lambda_0-1}{2}}}{(\ell-1)^{\frac{\lambda_1-1}{2}}}\Big ).$$
If we choose $\ell=2n+1$, then we have
$$\frac{(\ell-n)^{\frac{\lambda_0-1}{2}}}{(\ell-1)^{\frac{\lambda_1-1}{2}}}\sim
\Big (\frac{1}{n^{\frac{\lambda_1-\lambda_0}{2}}}\Big)\;\mbox{\rm for large }\; n.$$
Hence $||nT^{n-1}_{0}S_{0,1}
||\rightarrow \infty$ as $n\rightarrow \infty.$

Let $T_{|_{2\times 2}}$  denote the top $2\times 2$ block $\Big ( \begin{smallmatrix}
T_{0}&S_{0,1}\\
0&T_{1}
\end{smallmatrix}\Big )$ in the operator $T$.   Since
 $T_{|_{2\times 2}}^{n}=\begin{pmatrix}
T^n_{0}&nT^{n-1}_{0}S_{0,1}\\
0&T^n_{1}
\end{pmatrix},$ and $||T_{|_{2\times 2}}^n||\geq ||nT^{n-1}_{0}S_{0,1}
||,$ it follows that $||T_{|_{2\times 2}}^n||\rightarrow \infty$ as $n\rightarrow
\infty.$ Clearly, $\|T^n\| \geq \|T^n_{|_{2\times 2}}\|$ completing the proof.
\end{proof}
Since a quasi-homogeneous operator for which $\lambda_0 < 1$ can't be power bounded, the lemma we have just proved shows that if $T$ is quasi-homogeneous and $\Lambda(t) < 2,$ then the operator $T$ is not power bounded. Therefore we have proved the following theorem answering the Halmos' question in the affirmative.
\begin{thm}
If a quasi-homogeneous operator $T$ has the property $\|p(T)\|_{\rm op} \leq K \|p\|_{\infty, \mathbb D},$ $p\in\mathbb C[z],$ then it must be similar to a contraction.
\end{thm}


\end{document}